\documentclass[graybox]{svmult}

\usepackage{graphicx}
\usepackage{mathptmx}       % selects Times Roman as basic font
\usepackage{helvet}         % selects Helvetica as sans-serif font
\usepackage{courier}        % selects Courier as typewriter font
\usepackage{type1cm}        % activate if the above 3 fonts are
\usepackage{amsmath}
\usepackage{amssymb}
\usepackage[all,2cell]{xy}\UseAllTwocells\SilentMatrices
\usepackage{multirow}

\def\shf{\mathcal}
\def\dshf{\mathfrak}
\def\col{\mathcal}
\def\cl{\textrm{cl }}

\def\pr{\textrm{pr}}
\def\id{\textrm{id}}

\makeindex

\begin{document}

\title*{Sheaf and duality methods for analyzing multi-model systems}
\author{Michael Robinson}
\institute{Michael Robinson \at American University, 4400 Massachusetts Ave NW, Washington, DC 20016, \email{michaelr@american.edu}}
\maketitle

\abstract*{There is an interplay between models, specified by variables and equations, and their connections to one another.   This dichotomy should be reflected in the abstract as well.  Without referring to the models directly -- only that a model consists of spaces and maps between them -- the most readily apparent feature of a multi-model system is its topology.  We propose that this topology should be modeled first, and then the spaces and maps of the individual models be specified in accordance with the topology.  Axiomatically, this construction leads to \emph{sheaves}.  Sheaf theory provides a toolbox for constructing predictive models described by systems of equations.  Sheaves are mathematical objects that manage the combination of bits of local information into a consistent whole.  The power of this approach is that complex models can be assembled from smaller, easier-to-construct models.  The models discussed in this chapter span the study of continuous dynamical systems, partial differential equations, probabilistic graphical models, and discrete approximations of these models.}

\section{Introduction}

Complex predictive models are notoriously hard to construct and to study.  Sheaf theory provides a toolbox for constructing predictive models described by systems of equations.  Sheaves are mathematical objects that combine bits of local information into a consistent whole.  The power of this approach is that complex models can be assembled from smaller, easier-to-construct models.  The models discussed in this chapter span the study of continuous dynamical systems, partial differential equations, probabilistic graphical models, and discrete approximations of these models.  Shadows of the sheaf theoretic perspective are apparent in a variety of disciplines, for instance in the construction of volume meshers (which construct pullbacks and pushforwards of sheaves of functions), finite element solvers (which construct the space of global sections of a sheaf), and loopy belief propagation (which iteratively determines individual global sections).  

Multi-model systems can be constructed from diagrams of individual models.  It is therefore helpful to abstract this idea into a convenient formalism, in which the basic features are captured.  There is an interplay between the models themselves and their connections to one another.  This dichotomy should be reflected in the abstract as well.  Without referring to the models directly -- only that a model consists of spaces and maps between them -- the most readily apparent feature of a multi-model system is its topology.  We propose that this topology should be modeled first, and then the spaces and maps of the individual models be specified in accordance with the topology.  Axiomatically, this construction leads to \emph{sheaves}.  Our overall modeling approach consists of specifying sheaves \emph{of} a certain type \emph{on} a certain kind of topological space.  We could aim for complete generality, though this leads to considerable computational and analytic difficulties.  We will build our theory on \emph{partial orders} (or \emph{posets}, for short); these lend themselves to a balance between theoretical expressiveness and computational facility.  We will specify a multi-model system as a sheaf of smooth, finite-dimensional manifolds on a poset.  A sheaf model is a natural way to express the topological relationship among state variables and the equations relating them.  Solver algorithms for the set of equations enforce consistency constraints among the state variables, which is precisely the process of computing the space of global sections of a sheaf.

Encoding models as sheaves allows one to realize two rather different capabilities:
\begin{enumerate}
\item Combining vastly different dynamical models into a multi-model system in a systematic way, and
\item Analyzing homological invariants to study locally- and globally-consistent states of the system.
\end{enumerate}

Both of these leverage the topological structure already inherent in and among the models.  For instance:
\begin{enumerate}
\item The base topological space is often essentially a material volume, envisioned abstractly.  It can be partitioned into a cellular space, such as finite element meshers already do.  All modern solid modelers store an explicit, topological model of the model volume \cite[Chap. 2]{Hoffmann_1989}.  In building sheaves, the topology can be refined (cells subdivided, for instance) in order to construct discretizations.  Some solid modeling/meshing APIs can do this natively \cite[Sec. 3]{Li_2005}.
\item The local data represented in a sheaf are the state variables in the interior of each mesher cell, exactly as the finite element solver represents them.
\item The equations are encapsulated in (not necessarily linear) maps deriving boundary values from the parameters known about the interiors of each cell.
\end{enumerate}

\section{Modeling systems with diagrams}
\label{sec:diagram}

Dynamical models usually involve a collection of state variables and equations that determine their admissible values.  For instance, the famous Lorenz system specifies the values of three variables -- each of which is a function of time -- constrained by three equations
\begin{equation}
  \label{eq:lorenz}
  \begin{cases}
    \frac{dx}{dt} = a(y-x),\\
    \frac{dy}{dt} = x(b-z) - y,\\
    \frac{dz}{dt} = xy - cz,\\
  \end{cases}
\end{equation}
where $a,b,c$ are constants.  The values of $x$ and $y$ determine the future values of $x$, but all three are implicated in determining $y$ and $z$.  The solutions exhibit intricate behavior because the values of $z$ constrain the values of $x$ even though there is not a direct causal relationship.

\begin{figure}
  \begin{center}
    \includegraphics[width=3in]{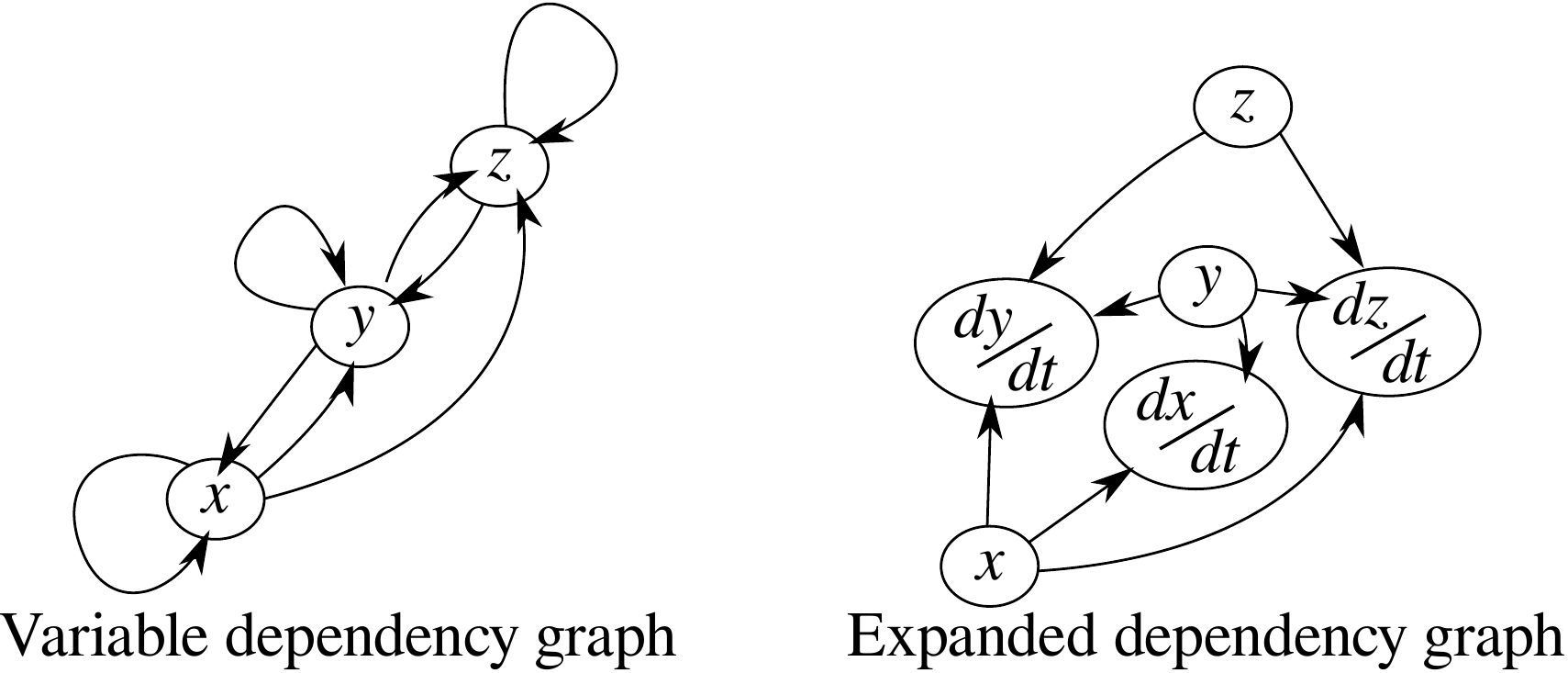}
    \caption{Dependency relations among state variables in the Lorenz system (left), and among variables and their derivatives (right)}
    \label{fig:lorenz1}
  \end{center}
\end{figure}

One way to gain an understanding of the behavior of solutions to \eqref{eq:lorenz} is to build a visual representation of the causal relationships between the state variables.  The left frame of Figure \ref{fig:lorenz1} shows one such representation, where an arrow from one variable $x \to y$ to another indicates that $x$ partially determines future values of $y$.  

This representation isn't entirely true to the way the equations in \eqref{eq:lorenz} are written, because the equations also involve derivatives of the state variables.  If we include derivatives of state variables as new state variables in their own right, then we obtain a rather larger diagram, such as the one in the right frame of Figure \ref{fig:lorenz2}.  This new diagram is a bit more instructive, in that it is the derivatives that are determined by the values of the state variables.  However, it still leaves unstated the relationship between the derivative of a state variable and the state variable itself.  For instance, the derivative $dx/dt$ is determined independently both by the values of $x$ (alone) and by the values of $x$ and $y$ through the first equation of \eqref{eq:lorenz}.  It would be useful to encode all of this information into the diagram.

The way to perform this encoding is to reinterpret the meaning of the arrows in the dependency graphs.  Instead of an arrow indicating that the variable on the head is determined in part by the variable on the tail, it is better to demand that arrows be actual functional relations.  This stronger requirement is not satisfied by either of the diagrams in Figure \ref{fig:lorenz1}.  The problem is that in \eqref{eq:lorenz}, the formula for $dx/dt$ depends \emph{jointly} on the variables $x$ and $y$.  Therefore, the functional dependence between $x$, $y$, and $dx/dt$ needs to be from \emph{pairs} of values $(x,y)$.  When we perform this transformation to the dependency diagram, we obtain the diagrams in Figure \ref{fig:lorenz2}.

\begin{figure}
  \begin{center}
    \includegraphics[width=3in]{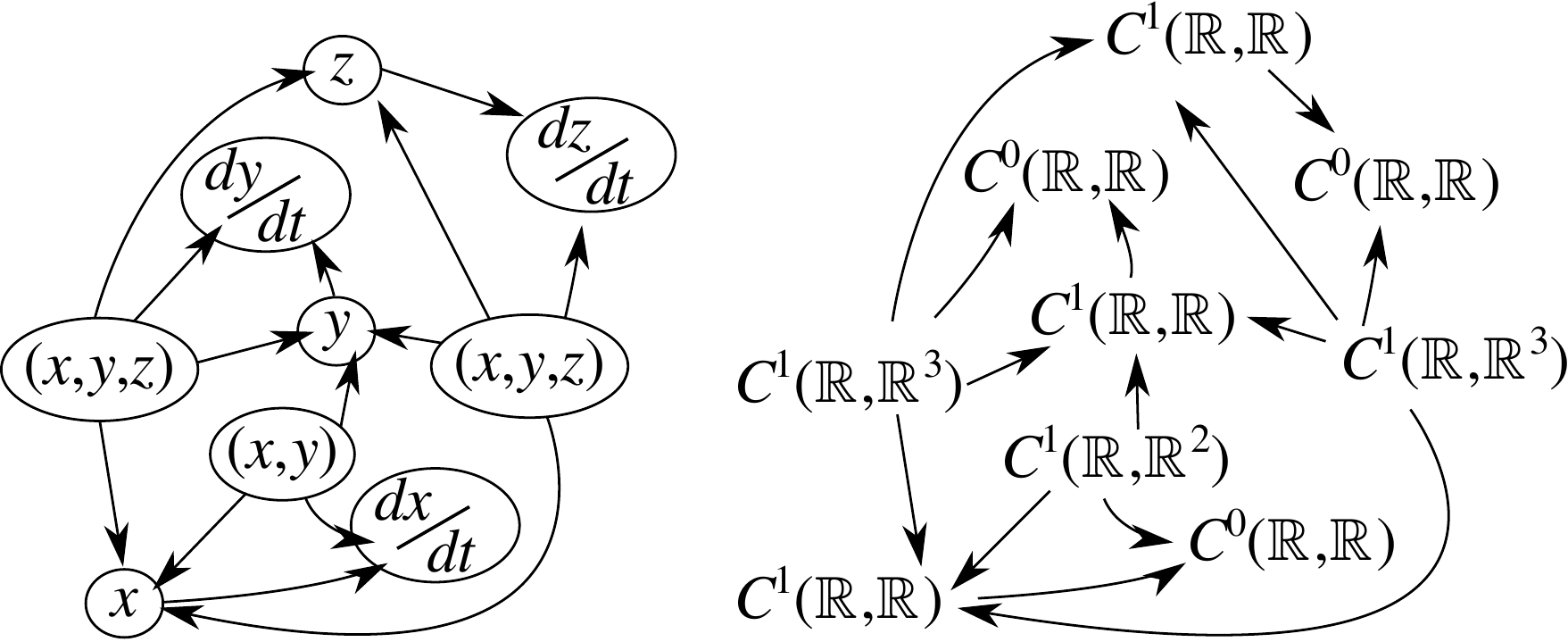}
    \caption{Functional dependencies among state variables and their derivatives in the Lorenz system, according to variable names (left) and according to the spaces of values involved (right) }
    \label{fig:lorenz2}
  \end{center}
\end{figure}

There are a number of pleasing features about a functional dependency diagram like the ones in Figure \ref{fig:lorenz2}.  The most obvious -- and most trivial -- is that the arrows (on the right frame) are actual functions, and could be labeled as such.  The arrows out of the spaces corresponding to tuples of variables are projections, while the others are determined by one of the equations \eqref{eq:lorenz} and by the definition of the derivative.  Everything about \eqref{eq:lorenz} is captured in the diagram on the right, in that the equations can be recovered from the diagram.  The in-degree of a variable in Figure \ref{fig:lorenz2} specifies the number of functional equations that constrain its value.  This means that the independent variables listed in Figure \ref{fig:lorenz2} are those with no arrows pointing into them.  It is easy to see that these are the pairs $(x,y)$ and the triples $(x,y,z)$.  This does \emph{not} mean that there are no constraints on these independent variables, just that there are no functional dependencies from the outset.  Constraints on these independent variables arise by demanding that each listed variable in the diagram take exactly one value.  Then if a variable is determined by two functional equations, the independent variables in those two equations must be chosen compatibly.  Notice that there are some values that are completely dependent ($dx/dt$, $dy/dt$, and $dz/dt$), in that they have no arrows going out of them, while there are also intermediate variables ($x$, $y$, and $z$), that have arrows going in and out.  

The diagram in the left frame of Figure \ref{fig:lorenz2} is that of a partially ordered set.  The partial order ranks the variables appearing in \eqref{eq:lorenz} according to their  ``independence'' of one another.  The independent variables are the minimal elements of the partial order, while the completely dependent ones are the maximal elements.  Therefore, the arrows in the diagram in the Figure point from lower variables to higher ones in the partial order.

The diagram on the right frame of Figure \ref{fig:lorenz2} has the same structure as the partial order, but is labeled a bit differently.  This kind of diagram is that of a \emph{sheaf}, which is a mathematical way to represent local consistency relationships.  Although the study of sheaves over general topological spaces can be quite technical, sheaves over partially ordered sets are much more tractable.  Sheaves have a number of useful invariants that provide descriptive power for systems of equations, and often the mere act of encoding a system as a sheaf is illuminating.  For instance, writing differential equations along a stratified manifold requires delicate management of boundary conditions of various sorts.  The sheaf encoding described in this chapter makes specifying the correct kind of boundary conditions almost effortless.  This chapter discusses a number of techniques for performing sheaf encodings of systems, explains some of the relationships among these encodings, and describes some of the analytical techniques that can be used on sheaf-encoded models.

\section{Mathematical constructions of sheaves}

Since topological spaces in their full generality tend to admit rather pathological properties that are not reflected in practical models, it is wise to apply constraints.  There are several other possibilities for the topological space, and they vary in expressiveness.  In the author's experience, locally finite topological spaces, cell complexes, abstract simplicial complexes, and partial orders are the most useful for problems involving models of systems.  Of these, locally finite topological spaces are the most general, but nearly every useful computational example can be expressed more compactly with partial orders.  For instance, graphs and abstract simplicial complexes are both special cases of \emph{hypergraphs}, which are merely subsets of the power set of some set $V$.  A hypergraph $H$ in which each element is a finite subset of $V$ can be thought of as locally finite partial order with the partial order coming from inclusion: for $A,B \in H$, $A \le B$ precisely when $A \subseteq B$.

We first build a rather specialized description of sheaves on partial orders, since they are the primary mathematical tool in this chapter.  We then generalize to sheaves on general topological spaces, since they describe the spaces of functions we will discretize in later sections.

\subsection{Sheaves on partial orders}

We will mostly deal with locally finite posets, since these avoid computational difficulties.

\begin{definition}
A \emph{partial order} on a set $P$ is a relation $\le$ on that set that is 
\begin{enumerate}
\item Reflexive: $x \le x$ for all $x \in P$,
\item Antisymmetric: if $x \le y$ and $y \le x$, then $x = y$, and 
\item Transitive: if $x \le y$ and $y \le z$ then $x \le z$.
\end{enumerate}
We call the pair $(P,\le)$ a \emph{partially ordered set} or a \emph{poset}.  When the relation is clear from context, we shall usually write $P=(P,\le)$.  A poset is \emph{locally finite} if for every pair $x,y \in P$, the set $\{z \in P : x \le z \le y\}$ is finite.

Given a partially ordered set $P=(P,\le)$, there is also the \emph{dual partial order} $\le^{op}$ on $P$, for which $x \le^{op} y$ if and only if $y \le x$.  The partially ordered set $P^{op}=(P,\le^{op})$ is called the \emph{dual poset to $P$}.
\end{definition}

\begin{figure}
\begin{center}
\includegraphics[width=4in]{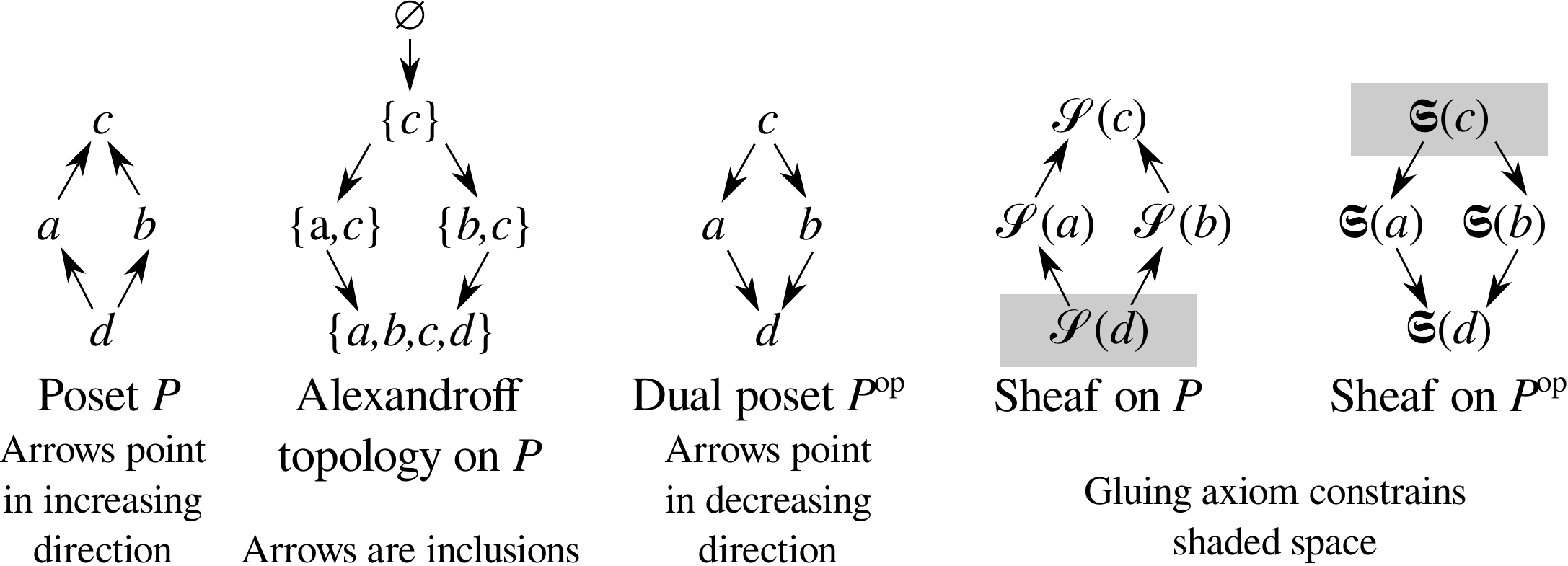}
\caption{A poset $P$, its Alexandroff topology (Definition \ref{df:alexandroff}), the dual poset $P^{op}$, and sheaves over each (Definition \ref{df:sheaf_poset}).  The gluing axiom appears in Definition \ref{df:sheaf}.}
\label{fig:lattice_sheaves}
\end{center}
\end{figure}

Given a subset $A \subseteq P$ of a poset, the \emph{infimum of $A$} is the unique greatest element in $P$ less than or equal to each element of $A$ if such an element exists.  We will write $\bigwedge A$ for the infimum of $A$ if it exists.  Similarly, the \emph{supremum of $A$}, written $\bigvee A$ is the unique least element in $P$ greater than or equal to each element of $A$, if such an element exists.  If $A=\{A_1,A_2\}$, we usually write $\bigwedge A = A_1 \wedge A_2$ and $\bigvee A = A_1 \vee A_2$.   

\begin{example}
Figure \ref{fig:lattice_sheaves} shows a poset $P$ with four elements at left and its dual poset $P^{op}$ at center.  The diagram is to be read that $d \le a \le c$ in $P$.  In $P$, the infimum of $A=\{a,b\}$ is $d$, while the supremum is $c$.
\end{example}

The definition of a \emph{sheaf} is rather crisply stated in terms of the diagram of a poset, where the vertices represent elements and arrows point from lesser elements to greater ones.  Merely replace each vertex by a set or a space and each arrow by a function so that the composition of functions in the diagram is path independent.  If all of the functions' inputs are at the tails of each arrow, then the diagram is that of a \emph{sheaf on the Alexandroff topology} for that poset.  (We generalize to arbitrary topologies in Section \ref{sec:topo_sheaves}.)  If all of the functions' inputs are at the heads of each arrow, then the diagram defines a sheaf over the dual poset.  When discussing a particular poset, we will emphasize this difference by setting sheaves over that poset in script type, and by setting sheaves over the dual poset in fraktur type.

\begin{definition}
\label{df:sheaf_poset}
  Suppose that $P=(P,\le)$ is a poset.  A \emph{sheaf $\shf{S}$ of sets on $P$ with the Alexandroff topology} (briefly, a \emph{sheaf $\shf{S}$ on $P$}) consists of the following specification:
\begin{enumerate}
\item For each $p \in P$, a set $\shf{S}(p)$, called the \emph{stalk at} $p$, 
\item For each pair $p \le q \in P$, there is a function $\shf{S}(p \le q):\shf{S}(p)\to\shf{S}(q)$, called a \emph{restriction function} (or just a \emph{restriction}), such that
\item For each triple $p \le q \le r \in P$, $\shf{S}(p \le r) = \shf{S}(q \le r) \circ \shf{S}(p \le q)$.
\end{enumerate}
When the stalks themselves have structure (they are vector spaces or topological spaces, for instance) one obtains a sheaf \emph{of} that type of object when the restrictions or extensions preserve that structure.  For example, a \emph{sheaf of vector spaces} has linear functions for each restriction, while a \emph{sheaf of topological spaces} has continuous functions for each restriction.

Similarly, a \emph{sheaf $\dshf{C}$ of sets on the dual poset $P^{op}$ with the Alexandroff topology} (briefly, a \emph{dual sheaf $\dshf{C}$ on $P$}) consists of the same kind of thing, just backwards.  Namely,
\begin{enumerate}
\item For each $p \in P$, a set $\dshf{D}(p)$, called the \emph{stalk at} $p$, 
\item For each pair $p \le q \in P$, there is a function $\dshf{D}(p \le q):\dshf{D}(q)\to\dshf{D}(p)$, called an \emph{extension function} (or just an \emph{extension}), such that
\item For each triple $p \le q \le r \in P$, $\dshf{D}(p \le r) = \dshf{D}(p \le q) \circ \dshf{D}(q \le r)$.
\end{enumerate}

If either of the conditions (3) above are not satisfied, we call the construction a \emph{diagram} instead of a sheaf. 
\end{definition}

\begin{example}
  \label{ex:simple_sheaves}
  In Figure \ref{fig:lattice_sheaves}, choosing
  \begin{equation*}
    \shf{S}(a) = \mathbb{R},\;\shf{S}(b) = \mathbb{R},\; \shf{S}(c) = \mathbb{R},\; \shf{S}(d) = \mathbb{R},
  \end{equation*}
  with
  \begin{equation*}
    \left(\shf{S}(d \le a)\right)(x) = 2x,\;\left(\shf{S}(d \le b)\right)(x) = x,\;\left(\shf{S}(a \le c)\right)(x) = x,\;\left(\shf{S}(b \le c)\right)(x) = 2x,
  \end{equation*}
  results in a sheaf.  On the other hand,
  \begin{equation*}
    \left(\shf{S}(d \le a)\right)(x) = x,\;\left(\shf{S}(d \le b)\right)(x) = x,\;\left(\shf{S}(a \le c)\right)(x) = x,\;\left(\shf{S}(b \le c)\right)(x) = -x,
  \end{equation*}
  is merely a diagram, because the composition of the maps on the left ($d \to a \to c$) is the identity map, while the other composition ($d \to b \to c$) is not.
\end{example}

Encoding a multi-model system as a diagram is a useful exercise, since consistencies and inconsistencies between the component models are thereby formalized.  Those elements of the stalks that are mutually consistent across the entire system, formalized as a sheaf, are called \emph{sections}.  Sections are what the combined multi-model system produces as output, and amount to the simultaneous solution of a number of equations (see Section \ref{sec:simultaneous}).

\begin{definition}
\label{df:section}
A \emph{global section} of a sheaf $\shf{S}$ on a poset $P$ is an element $s$ of the direct product\footnote{Which is in general \emph{not} the direct sum, since $P$ may be infinite!} $\prod_{x \in P}\shf{S}(x)$ such that for all $x \le y \in P$ then $\shf{S}(x \le y)\left(s(x)\right) = s(y)$.  A \emph{local section} is defined similarly, but is defined only on a subset $Q \subseteq P$.

Dually, a \emph{global section} of a sheaf $\dshf{C}$ on the dual poset $P^{op}$ is an element $c$ of the direct product such that $\prod_{x \in P}\dshf{C}(x)$ such that for all $x \le y \in P$ then $c(x) = \dshf{C}(x \le y)\left(c(y)\right)$.  A \emph{local section} of such a sheaf is defined only on a subset $Q \subseteq P$.
\end{definition}

\begin{example}
  Continuing with the sheaf $\shf{S}$ from Example \ref{ex:simple_sheaves}, the space of global sections is given by
    $\{(2x,x,2x,x) \in \shf{S}(a) \times \shf{S}(b) \times\shf{S}(c) \times \shf{S}(d) \cong \mathbb{R}^4\},$
  which is itself isomorphic to $\mathbb{R}$.  On the other hand, the space of local sections over $\{a,b\}$ is just $\shf{S}(a) \times \shf{S}(b) \cong \mathbb{R}^2$ since there are no further constraints.
\end{example}

The structure of a particular mathematical object is better understood in context, by looking at structure-preserving transformations between them.  The \emph{morphisms} between sheaves provide this context.  Most authors tend to focus on the class of morphisms between sheaves over the same space, though for our purposes is quite essential to study sheaf morphisms involving different spaces.  

\begin{definition} (\cite{TSPBook} or \cite[Sec. I.4]{Bredon})
  \label{df:morphism}
  Suppose that $\shf{R}$ is a sheaf on a poset $Y$ and that $\shf{S}$ is a sheaf on $X$. 
  A \emph{sheaf morphism} $m: \shf{R}\to \shf{S}$ along an order preserving map $f: X \to Y$ (careful: $m$ and $f$ \emph{go in opposite directions!}) consists of a set of functions $m_x: \shf{R}(f(x)) \to \shf{S}(x)$ for each $x \in X$ such that the following diagram commutes
  \begin{equation*}
    \xymatrix{
      \shf{R}(f(y)) \ar[r]^{m_y} &\shf{S}(y)\\
      \shf{R}(f(x)) \ar[r]_{m_x} \ar[u]^{\shf{R}(f(x) \le f(y))}&\shf{S}(x) \ar[u]_{\shf{S}(x \le y)}\\
      }
  \end{equation*}
  for each $x \le y$.  We usually call the functions $m_x$ the \emph{components} of the sheaf morphism.  A sheaf morphism is said to be \emph{injective} (or \emph{surjective}) if each component is injective (or surjective).
\end{definition}

\begin{proposition}
  \label{prop:morphism}
  A sheaf morphism $m:\shf{R}\to\shf{S}$ along an order preserving $f:X \to Y$ induces a function taking global sections of $\shf{R}$ (a sheaf over $Y$) to global sections of $\shf{S}$ (a sheaf over $X$).
\end{proposition}
\begin{proof}
  Suppose $r$ is a section of $\shf{R}$. If $x\in X$, then let $s(x) = m_x(r(f(x)))$.  Then, $s$ is a section of $\shf{S}$ because whenever $x \le y \in X$, 
\begin{eqnarray*}
\left(\shf{S}(x \le y)\right)s(x) &=& \left(\shf{S}(x \le y)\right) m_{x}(r(f(x)))\\
&=& \left(\shf{S}(x \le y) \circ m_{x}\right)(r(f(x)))\\
&=& \left(m_y \circ \shf{R}(f(x) \le f(y)) \right)(r(f(x)))\\
&=& m_y (r(f(y)))\\
&=&s(y)
\end{eqnarray*}
by the definition of a sheaf morphism.
\end{proof}

Generalizing a bit further, it is also useful to be able to map the stalks of a dual sheaf into the stalks of a sheaf -- providing a notion of a \emph{hybrid morphism} from dual sheaves into sheaves.  This plays an important role in understanding discretizations.

\begin{definition}
  \label{df:hybridmorphism}
  Suppose that $\dshf{D}$ is a dual sheaf on a poset $Y$ and that $\shf{S}$ is a sheaf on $X$. 
  A \emph{hybrid morphism} $m: \dshf{D}\to \shf{S}$ along an order preserving map $f: X \to Y$ consists of a set of functions $m_x: \dshf{D}(f(x)) \to \shf{S}(x)$ for each $x \in X$ such that the following diagram commutes
  \begin{equation*}
    \xymatrix{
      \dshf{D}(f(y)) \ar[r]^{m_y} \ar[d]_{\dshf{D}(f(x) \le f(y))} &\shf{S}(y)\\
      \dshf{D}(f(x)) \ar[r]_{m_x} &\shf{S}(x) \ar[u]_{\shf{S}(x \le y)}\\
      }
  \end{equation*}
  for each $x \le y$.  We usually call the functions $m_x$ the \emph{components} of the hybrid morphism.  A hybrid morphism is said to be \emph{injective} (or \emph{surjective}) if each component is injective (or surjective).
\end{definition}

Like sheaf morphisms, hybrid morphisms transform local (or global) sections of a dual sheaf to local (or global) sections of a sheaf.

\subsection{Sheaves on topological spaces}
\label{sec:topo_sheaves}

This section explains the appropriate generalization of sheaves on posets with the Alexandroff topology to sheaves over arbitrary topological spaces.  Topological spaces and partial orders are closely related, because every topology defines a unique partial order.

\begin{definition}
A \emph{topology} on a set $X$ consists of a collection $\col{T}$ of subsets of $X$ that satisfy the following four axioms
\begin{enumerate}
\item $\emptyset \in \col{T}$,
\item $X \in \col{T}$,
\item If $U, V \in \col{T}$, then $U \cap V \in \col{T}$, and
\item If $\col{U} \subseteq \col{T}$, then $\cup \col{U} = \{x\in X: x \in U \text{ for some }U\in\col{U}\} \in \col{T}$.
\end{enumerate}
We will call $(X,\col{T})$ a \emph{topological space}.
\end{definition}

\begin{definition}
  \label{df:open}
Every topological space $X=(X,\col{T})$ defines a poset ${\bf Open}(X,\col{T})=(\col{T},\subseteq)$ on the open sets, partially ordered by the subset relation.  When the topology $\col{T}$ is clear from context, we shall usually write ${\bf Open}(X)={\bf Open}(X,\col{T})$.
\end{definition}

The axioms for a topology ensure that in ${\bf Open}(X,\col{T})$, infima of finite sets exist, namely via $U_1 \wedge U_2 = U_1 \cap U_2$, and that suprema of any collection $\col{U}$ of open sets exists via $\bigvee \col{U} = \bigcup \col{U}$.  We note that a poset in which all infima and suprema exist for finite collections is usually called a \emph{lattice}, so every topological space defines a lattice of open sets. 

Although the Definitions above suggest that we can merely focus on posets, avoiding mention of topological spaces, this is only partially true.  If we take a given poset as ${\bf Open}(X,\col{T})$, this alone does not completely define a topological space. 

\begin{example}
\label{ex:same_poset}
Consider the set $X=\{a,b,c\}$ with two topologies, 
\begin{equation*}
\col{T}_1 = \{\{a,b,c\}, \{a,b\},\{c\},\emptyset\}
\end{equation*}
and 
\begin{equation*}
\col{T}_2=\{\{a,b,c\},\{a\},\{c\},\emptyset\}.
\end{equation*}
Both of these topologies have the same poset of open sets, namely 
${\bf Open}(X,\col{T}_1)$ at left below and ${\bf Open}(X,\col{T}_2)$ at right
\begin{equation*}
\xymatrix{
&\emptyset\ar[dl]\ar[dr]& & &\emptyset\ar[dl]\ar[dr]&\\
\{a,b\}\ar[dr]&&\{c\}\ar[dl] & \{a\}\ar[dr] &&\{c\}\ar[dl]\\
&\{a,b,c\}& & &\{a,b,c\}\\
}
\end{equation*}
Yet $(X,\col{T}_1)$ and $(X,\col{T}_2)$ are quite different as topological spaces.  Observe that both $\{a,b\} \vee \{c\} = \{a,b,c\}$ in ${\bf Open}(X,\col{T}_1)$ and $\{a\} \vee \{c\} = \{a,b,c\}$ in ${\bf Open}(X,\col{T}_2)$.  However, only in $\col{T}_1$ is the union of these two elements $\{a,b\} \cup \{c\} = \{a,b,c\}$.
\end{example}

One particularly important topology that can be built from a partial order is the \emph{Alexandroff} topology.

\begin{definition} \cite{Alexandroff_1937}
  \label{df:alexandroff}
In a poset $(P,\le)$, the collection of sets of the form
\begin{equation}
\label{eq:upset}
U_x = \{y \in P : x \le y\}
\end{equation}
for each $x \in P$ forms a base for a topology, called the \emph{Alexandroff} topology, shown in Figure \ref{fig:lattice_sheaves}.
\end{definition}

\begin{proposition}
\label{prop:alexandroffness}
Every intersection of open sets in the Alexandroff topology on a poset $P$ is open.
\end{proposition}
\begin{proof}
Suppose that $\col{U}$ is a collection of open sets in the Alexandroff topology and that $x \in \cap \col{U}$.  This means that $x$ is in every open set of $\col{U}$.  Now each of these open sets contains at least $U_x$, since these are the sets of the base.  Thus $U_x \subseteq \cap \col{U}$, which therefore shows that $\cap \col{U}$ is a neighborhood of each of its points.
\end{proof}

In the Alexandroff topology for a poset, the usual topological notions of closures, interiors, and frontiers\footnote{By \emph{frontier} of a set $A$, we mean $\cl A \cap \cl A^c$.} have straightforward interpretations in terms of the poset itself.  Additionally, because of Proposition \ref{prop:alexandroffness}, the concept of a \emph{star} is also available.

\begin{corollary}
If $A \subseteq X$ is a subset of a topological space, the \emph{star} of $A$ is the smallest open set containing $A$.  In general, stars need not exist, but in the Alexandroff topology for a poset there is a star of every subset.
\end{corollary}
\begin{proof}
It suffices to observe that the star over $A$ is the intersection of the collection of all open sets containing $A$.  
\end{proof}

It is very easy to see that order preserving maps $P\to Q$ between two posets are continuous when both $P$ and $Q$ are given the Alexandroff topology.

\begin{example}
Although the Alexandroff topology on $P$ defines a partial order on its open sets, this is both typically larger than $P$ and dual to $P$ in a particular way.  Consider the small example shown at left in Figure \ref{fig:lattice_sheaves}.  The poset $P$ contains four elements, and is a lattice.  The Alexandroff topology consists of five open sets -- it includes the empty set -- each of which happens to be the star over each original element of $P$.  As the Figure shows, the diagram of the Alexandroff topology contains the dual poset, namely $P^{op}$.  
\end{example}

\begin{definition}
  \label{df:presheaf}
Suppose $(X,\col{T})$ is a topology.  A \emph{presheaf $\shf{S}$ of sets on $(X,\col{T})$} consists of the following specification:
\begin{enumerate}
\item For each open set $U \in \col{T}$, a set $\shf{S}(U)$, called the \emph{stalk at} $U$, 
\item For each pair of open sets $U \subseteq V$, there is a function $\shf{S}(U \subseteq V):\shf{S}(V)\to\shf{S}(U)$, called a \emph{restriction function} (or just a \emph{restriction}), such that
\item For each triple $U \subseteq V \subseteq W$ of open sets, $\shf{S}(U \subseteq W) = \shf{S}(U \le V) \circ \shf{S}(V \le W)$.
\end{enumerate}
When the stalks themselves have structure (they are vector spaces or topological spaces, for instance) one obtains a presheaf \emph{of} that type of object when the restrictions or extensions preserve that structure.  For example, a \emph{presheaf of vector spaces} has linear functions for each restriction, while a \emph{presheaf of topological spaces} has continuous functions for each restriction.
\end{definition}

A sheaf $\shf{S}$ on a poset $P$ with the Alexandroff topology given by Definition \ref{df:sheaf_poset} is almost a presheaf on $(P,\col{A})$, where $\col{A}$ is the Alexandroff topology on $P$.  The only issue is that the stalks on unions of stars are not defined yet, but these will be defined in Proposition \ref{prop:poset_sheaves}. 

As Definition \ref{df:presheaf} makes clear, presheaves on a topological space are only sensitive to the poset of open sets, and \emph{not} to the points in those open sets.  Therefore, we can use Definition \ref{df:section} to define \emph{sections} of a presheaf on a topological space.

Because of the situation in Example \ref{ex:same_poset}, the set of global sections of a presheaf on the whole topological space may be quite different from the set of local sections over all open subsets.  It is for this reason that when studying presheaves over topological spaces, an additional \emph{gluing axiom} is included to remove this distinction.

\begin{definition}
\label{df:sheaf}
Let $\shf{P}$ be a presheaf on the topological space $(X,\col{T})$.  We call $\shf{P}$ a \emph{sheaf on $(X,\col{T})$} if for every open set $U \in \col{T}$ and every collection of open sets $\col{U}\subseteq \col{T}$ with $U = \cup \col{U}$, then $\shf{P}(U)$ is isomorphic to the space of sections over the set of elements $\col{U}$.
\end{definition}

\begin{example}
Recall Example \ref{ex:same_poset}, in which two topologies were considered on the set $X=\{a,b,c\}$, and consider the diagram
\begin{equation*}
\xymatrix{
&\{0\}&\\
\mathbb{R}\ar[ur]&&\mathbb{R}\ar[ul]\\
&\mathbb{R}\ar[ul]^{\id}\ar[ur]^{\id}&\\
}
\end{equation*}
where $\id$ is the identity function.  This diagram defines a presheaf for both $(X,\col{T}_1)$ and $(X,\col{T}_2)$, but only a sheaf on $(X,\col{T}_2)$.  Specifically, since $\{\{a,b\},\{c\}\}$ is a cover for $\{a,b,c\}$ in $\col{T}_1$, the stalk on $\{a,b,c\}$ must be the global sections on $\{\{a,b\},\{c\}\}$, which is $\mathbb{R}^2$, yet the stalk there is $\mathbb{R}$.  However, the analogous structure in $\col{T}_2$ is $\{\{a\},\{c\}\}$, which is not a cover for $\{a,b,c\}$, so the gluing axiom does not apply.
\end{example}

\begin{figure}
  \begin{center}
    \includegraphics[width=3.5in]{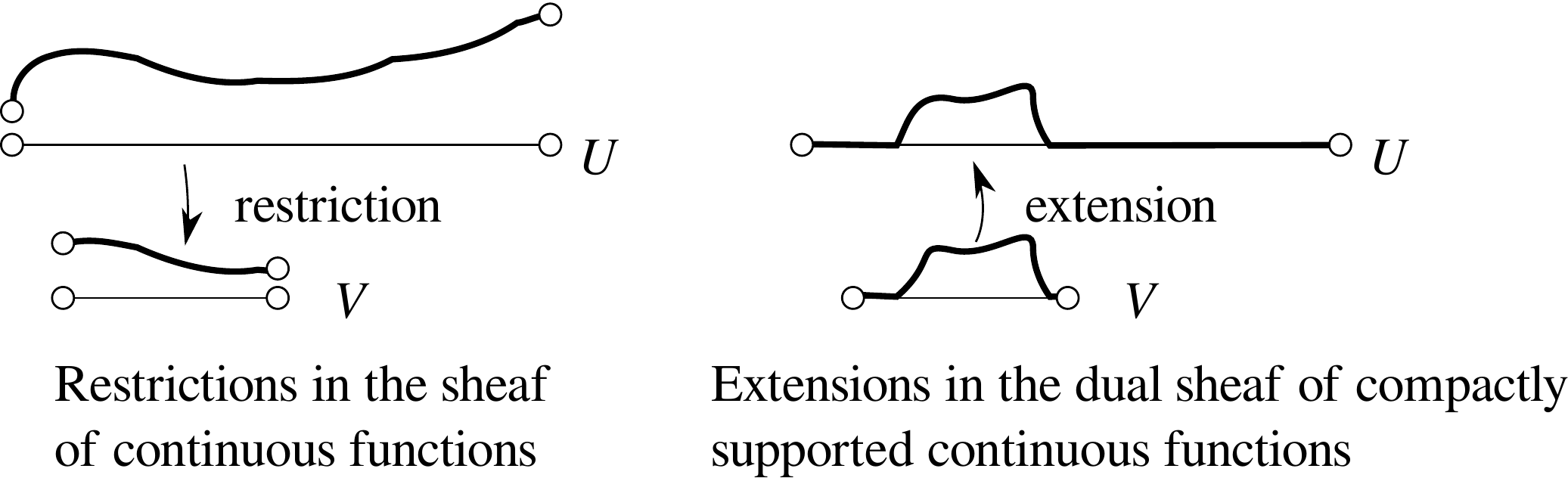}
    \caption{Some stalks in the sheaf of continuous functions $\mathbb{R}\to \mathbb{R}$ (left) and the dual sheaf of compactly supported functions $\mathbb{R}\to\mathbb{R}$ (right)}
    \label{fig:contfunctions}
  \end{center}
\end{figure}

\begin{example}
  \label{eg:contfunctions1}
  Let $(X,\col{T})$ and $(Y,\col{S})$ be topological spaces.  The space $C(X,Y)$ of continuous functions $X\to Y$ has the structure of a sheaf $\shf{C}(X,Y)$ on $(X,\col{T})$.  As the left frame of Figure \ref{fig:contfunctions} shows, the stalk over $U \in \col{T}$ is $C(U,Y)$ and if $U \le V$, then restricting the domain induces a restriction function $C(U,Y)\to C(V,Y)$.  The gluing axiom expresses the well-known fact that whenever two continuous functions with overlapping domains are equal on the overlap, then they extend to a common continuous function over the union.
\end{example}

If $(P,\le)$ is a poset with the Alexandroff topology, the distinction between sheaves and presheaves vanishes.

\begin{proposition}
\label{prop:poset_sheaves}
Let $\shf{R}$ be a sheaf on a poset $(P,\le)$ with nonempty stalks (Definition \ref{df:sheaf_poset}).  There is a sheaf (Definition \ref{df:sheaf}) $\shf{R}'$ on $(P,\col{A})$ where $\col{A}$ is the Alexandroff topology given by
\begin{enumerate}
\item $\shf{R}'(U_x) = \shf{R}(x)$ for each $x\in P$,
\item $\shf{R}'(U_y \subseteq U_x) = \shf{R}(x \le y) : \shf{R}(x) \to \shf{R}(y)$ for each pair of elements $x \le y$ in $P$,
\item $\shf{R}'(\bigcup_{i\in I} U_{x_i})$ is the space of sections of $\shf{R}$ over $\{x_i\}_{i\in I}$ for any collection of elements $\{x_i\}_{i\in I}$ in $P$, and
\item restrictions $\shf{R}'(U_{x_i} \subseteq \bigcup_{i\in I} U_{x_i})$ are given by projection maps.
\end{enumerate}
\end{proposition}

Proposition \ref{prop:poset_sheaves} justifies our terminology ``sheaf on a poset with the Alexandroff topology.''  Throughout this chapter, if $(P,\le)$ is a poset, we will assume it has the Alexandroff topology unless explicitly noted.  We will therefore not distinguish between \emph{presheaves} and \emph{sheaves} on $P$ unless a different topology is explicitly specified.  For dual sheaves, we will \emph{always} use the Alexandroff topology for the dual poset.  Care is needed, if the poset $P$ is ${\bf Open}(X,\col{T})^{op}$ for some topology $\col{T}$, because then the Alexandroff topology on $P$ will generally be different from $\col{T}$!

\begin{proof}
First, observe that $\shf{R}'$ is a presheaf on $(P,\col{A})$ by construction.  Conditions (3) and (4) result in $\shf{R}'$ satisfying the gluing axiom on unions of stars.  The gluing axiom leaves the stalks over maximal elements of $P$ unconstrained, so we only need to investigate the other elements addressed by conditions (1) and (2).  For a given element $x \in P$, suppose that $y_1, \dotsc$ are the elements strictly greater than $x$ in $P$.  Observe that $\bigcup_i U_{y_i}$ cannot cover $U_x$, because in the Alexandroff topology, the only way that $U_{y_1}, \dotsc$ covers $x$ is if one of them contains the star over $x$.  

Thus the gluing axiom requires us to compute the space of sections of $\shf{R}$ over at least the star over $x$, namely $\{x,y_1, \dotsc\}$, which is given by
\begin{equation*}
\{(a,b_1,\dotsc) \in \shf{R}(x) \times \prod_{i} \shf{R}(y_i) : b_i = \left(\shf{R}(x \le y_i)\right)(a)\},
\end{equation*}
which is evidently in bijective correspondence with $\shf{R}(x)$.  Notice that the third axiom in Definition \ref{df:presheaf} ensures that the above construction is well-defined.
\end{proof}

\begin{example}
  \label{eg:contfunctions2}
  The compactly supported continuous functions $C_c(X,\mathbb{R})$ on some topological space $(X,\col{T})$ are best organized in a sheaf $\dshf{C}_c(X,\mathbb{R})$ over the poset ${\bf Open}(X)$, which is a dual sheaf on ${\bf Open}(X)^{op}$.  (Notice the use of the fraktur font, and beware that we are using the Alexandroff topology on ${\bf Open}(X)$ in this example!)  Similar to the situation in Example \ref{eg:contfunctions1}, the stalk over $U$ is $C_c(U,\mathbb{R})$.  But instead of restricting along $U \le V$, one can \emph{extend} by zero, obtaining a function $C_c(V,\mathbb{R}) \to C_c(U,\mathbb{R})$ as the right frame of Figure \ref{fig:contfunctions} shows.

  $\dshf{C}_c(X,\mathbb{R})$ has only one global section: the zero function.  Observe that the infimum of all of ${\bf Open}(X)$ is the empty set.  Thus, the gluing axiom implies that the stalk over the empty set should be the trivial vector space.  For finitely many open sets $\{U_1, U_2, \dotsc, U_n\}$, the local sections are given by $C_c(U_1 \cap \dotsb U_n,\mathbb{R})$, which the gluing axiom asserts is the stalk over $U_1 \cap \dotsb U_n$.  

The situation is quite different for infinite collections of open sets, since they might not have an open intersection.  For instance, the intersection of the set of shrinking intervals $\{(0,1/n)\}_{n=1}^\infty$ is the (non-open) singleton $\{0\}$, but in ${\bf Open}(X)$ the infimum of the set $\{(0,1/n)\}_{n=1}^\infty$ is the empty set.  The only compactly supported continuous function on this is the zero function, which is also the stalk over the empty set.  On the other hand, the infimum of the set $\{(0,1+1/n)\}_{n=1}^\infty$ does not exist in ${\bf Open}(X)$, so the gluing axiom is mute about sections over this collection.
\end{example}

To motivate Definition \ref{df:morphism} of a sheaf morphism, consider a continuous function $F:(X,\col{T}_X) \to (Y,\col{T}_Y)$ from one topological space to another.  Suppose that $\shf{P}$ is a sheaf on $(X,\col{T}_X)$ and $\shf{Q}$ is a sheaf on $(Y,\col{T}_Y)$.  A \emph{sheaf morphism $m:\shf{P} \to \shf{Q}$ along $F$} consists of a set of maps $m_U$, one for each $U\in\col{T}_Y$, such that 
\begin{equation*}
    \xymatrix{
      \shf{P}(F^{-1}(U)) \ar[r]^{m_U} &\shf{Q}(U)\\
      \shf{P}(F^{-1}(V)) \ar[r]_{m_V} \ar[u]^{\shf{P}(F^{-1}(U) \subseteq F^{-1}(V))}&\shf{Q}(V) \ar[u]_{\shf{Q}(U \subseteq V)}\\
      }
\end{equation*}
for each pair of open sets $U \subseteq V$ in $\col{T}_Y$.

\begin{remark}
Sheaf morphisms are closely related to the concept of a \emph{pullback sheaf} along an order preserving map (Definition \ref{df:pullback}) as we will see later in the chapter.
\end{remark}

In order to focus on sheaves over posets, observe that $F$ induces an order-preserving map $f: {\bf Open}(Y,\col{T}_Y) \to {\bf Open}(X,\col{T}_X)$ given by
\begin{equation*}
f(U) = F^{-1}(U)
\end{equation*}
for each $U \in \col{T}_Y$.  Notice that $f$ and $F$ go in opposite directions!

\section{Discretization of functions}
\label{sec:discretization}

The best place to start any discussion about numerical analysis is with discretization.  There are two main ways to do this: by discretizing the domain or by discretizing the space of functions.  Given two topological spaces $X$ and $Y$, consider the space $C(X,Y)$ of continuous functions $X\to Y$.  As Example \ref{eg:contfunctions1} showed, this space can also be thought of as a sheaf $\shf{C}(X,Y)$ over the topological space $X$.  Following Definition \ref{df:open}, we define ${\bf Open}(X)^{op}$ as the poset whose elements are the open sets of $X$, and whose order relation is given by subsets: $U \le V$ if $V \subseteq U$.  

The two discretizations of functions in $C(X,Y)$ involve replacing the poset ${\bf Open}(X)^{op}$ with a new \emph{locally finite} poset $P$ and translating the sheaf $\shf{C}(X,Y)$ on $X$ into a new sheaf $\shf{D}$ or dual sheaf $\dshf{D}$ on $P$ with the Alexandroff topology.  There are two basic ways to do this, namely
\begin{enumerate}
\item By \emph{sampling} via a surjective sheaf morphism $\shf{C}(X,Y) \to \shf{D}$ or
\item By \emph{collapsing} via an injective hybrid morphism $\dshf{D} \to \shf{C}(X,Y)$.
\end{enumerate}
As is described in later sections of this chapter, sampling corresponds to finite difference methods, while collapsing corresponds to finite element methods.  The easiest way to construct a suitable $\shf{D}$ or $\dshf{D}$ is via \emph{pullbacks} and \emph{pushforwards} along order preserving maps, respectively.

\begin{definition}
\label{df:pullback}
  If $f: X\to Y$ is an order preserving function on posets and $\shf{S}$ is a sheaf on $Y$, then the \emph{pullback along $f$} is a sheaf $f^* \shf{S}$ on $X$ whose
  \begin{enumerate}
  \item Stalks are given by $f^* \shf{S}(x) = \shf{S}(f(x))$, and whose
  \item Restrictions are given by $f^* \shf{S}(x \le y) = \shf{S}(f(x) \le f(y))$, which is well-defined because $f$ is order preserving.
  \end{enumerate}
  This construction results in a surjective sheaf morphism $\shf{S} \to f^*\shf{S}$ in which the component maps are identity functions.
\end{definition}

We note that every sheaf morphism factors uniquely into the composition of a pullback morphism with a morphism between sheaves on the same space (see \cite[Prop 3.2]{TSPBook} and \cite[I.4]{Bredon} for a precise statement and proof).

Let us examine sampling first.  Sampling arises from specifying an order-preserving $S: P \to {\bf Open}(X)^{op}$ -- going the opposite way from the morphism we intend to induce.  Then, the pullback $S^* \shf{C}(X,Y)$ is a sheaf over $P$.  Although the poset $P$ for this new sheaf may be smaller than ${\bf Open}(X)^{op}$, the stalks are not necessary much smaller than in the original.  Therefore, we generally are interested in subsheaves of $S^* \shf{C}(X,Y)$ with finite dimensional stalks.  Methodologically, these subsheaves are examined via surjective morphisms $S^* \shf{C}(X,Y) \to \shf{D}$ (see \cite{Robinson_SampleBook}).  In all cases, we must specify the poset $P$ with care.  Usually, it suffices to choose $P={\bf Open}(Z)^{op}$ for some topological space $Z$ with a coarser topology than $X$ as the next few examples show.

\begin{figure}
\begin{center}
\includegraphics[width=3in]{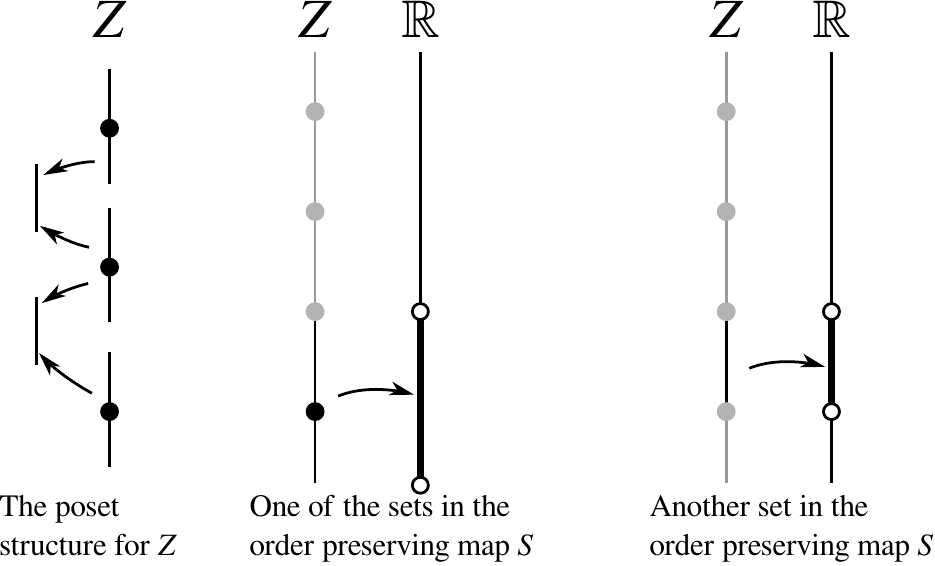}
\caption{The poset $P$ in Example \ref{eg:sample_line} (left), and two images of the sampling map $S$ (middle, right)}
\label{fig:sample_line_z}
\end{center}
\end{figure}

\begin{example}
  \label{eg:sample_line}
  Evenly-spaced discretization of the real line $X=\mathbb{R}$ can be performed by constructing $P$ as the poset consisting of two kinds of sets: $(n,n+1)$ and $(n-1,n+1)$ and for which $(n-1,n+1) \le (n,n+1)$ and $(n-1,n+1) \le (n-1,n)$ for all $n \in \mathbb{P}$ as shown in the left frame of Figure \ref{fig:sample_line_z}.  We construct the sampling function $S: P \to {\bf Open}(\mathbb{R})^{op}$ that reinterprets each element of $P$ as an actual interval of $\mathbb{R}$.  Then the stalk over $(n-1,n+1)$ of the pullback $S^* \shf{C}(\mathbb{R},Y)$ is $C((n-1,n+1),Y)$, while the stalk over $(n,n+1)$ is $C((n,n+1),Y)$.  Altogether, the pullback sheaf $S^* \shf{C}(\mathbb{R},Y)$ is given by the diagram
  \begin{equation*}
    \xymatrix{
  \dotsb    & C((0,1),Y) && C((1,2),Y) &\dotsb \\
C((-1,1),Y) \ar[ur] && C((0,2),Y) \ar[ul]\ar[ur] && C((1,3),Y)\ar[ul] \\
      }
  \end{equation*}
  
  The global sections of the pullback sheaf $S^* \shf{C}(\mathbb{R},Y)$ are precisely the continuous functions $C(\mathbb{R},Y)$.  This means that although we have discretized the topology, there is still more work to be done to reduce a function to a set of function values.  Although there are many ways to do this, even spacing is performed by a surjective morphism
    \begin{equation*}
    \xymatrix{
  \dotsb    & C((0,1),Y) \ar[d] && C((1,2),Y) \ar[d] &\dotsb \\
  C((-1,1),Y) \ar[ur]\ar[d] &Y^n& C((0,2),Y)\ar[d] \ar[ul]\ar[ur] &Y^n& C((1,3),Y)\ar[d]\ar[ul] \\
  Y^{n+1} \ar[ur]^{\sigma_+} && Y^{n+1}  \ar[ul]^{\sigma_-}\ar[ur]^{\sigma_+} && Y^{n+1} \ar[ul]^{\sigma_-}
      }
    \end{equation*}
    where $Y^n$ is the product of $n$ copies of $Y$, the vertical maps evaluate the continuous functions at either $n$ or $n+1$ points, and
    \begin{equation*}
      \sigma_+(y_0, \dotsc,y_n) = (y_0, \dotsc,y_{n-1}), \text{ and } \sigma_-(y_0, \dotsc,y_n) = (y_1, \dotsc,y_n).
    \end{equation*}
    The global sections of the bottom sheaf are infinite sequences of elements of $Y$.
\end{example}

\begin{figure}
\begin{center}
\includegraphics[width=2in]{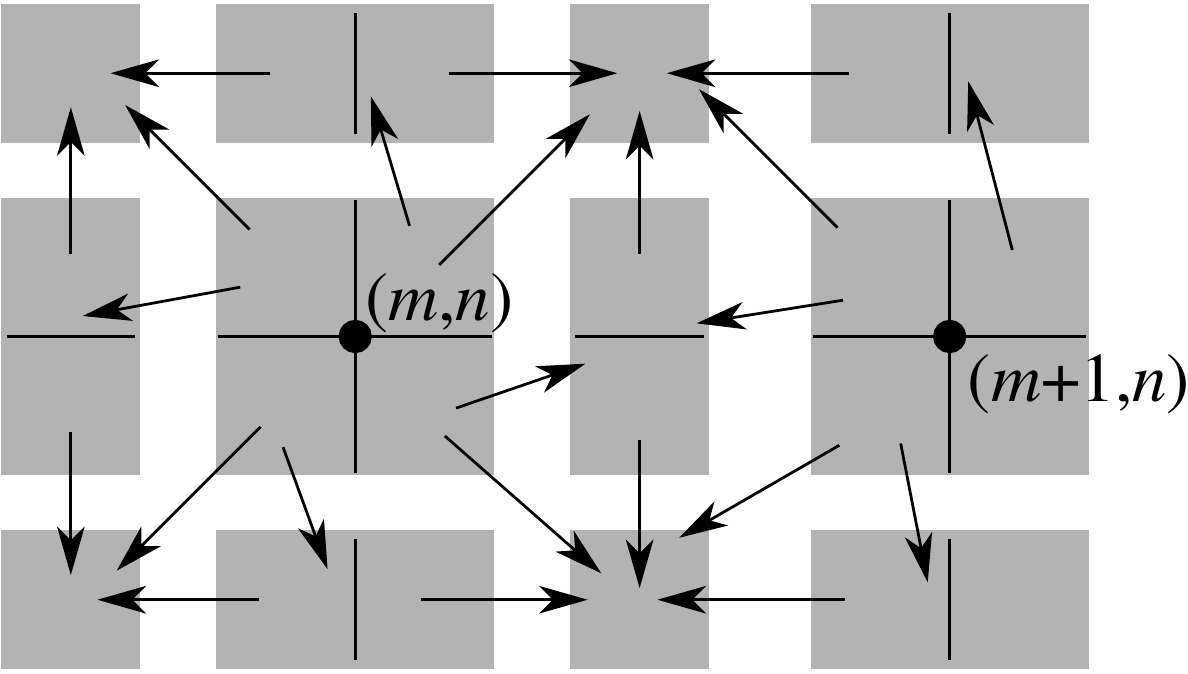}
\caption{The poset $P$ for Example \ref{eg:sample_plane}}
\label{fig:sample_plane_z}
\end{center}
\end{figure}

\begin{figure}
\begin{center}
\includegraphics[width=4in]{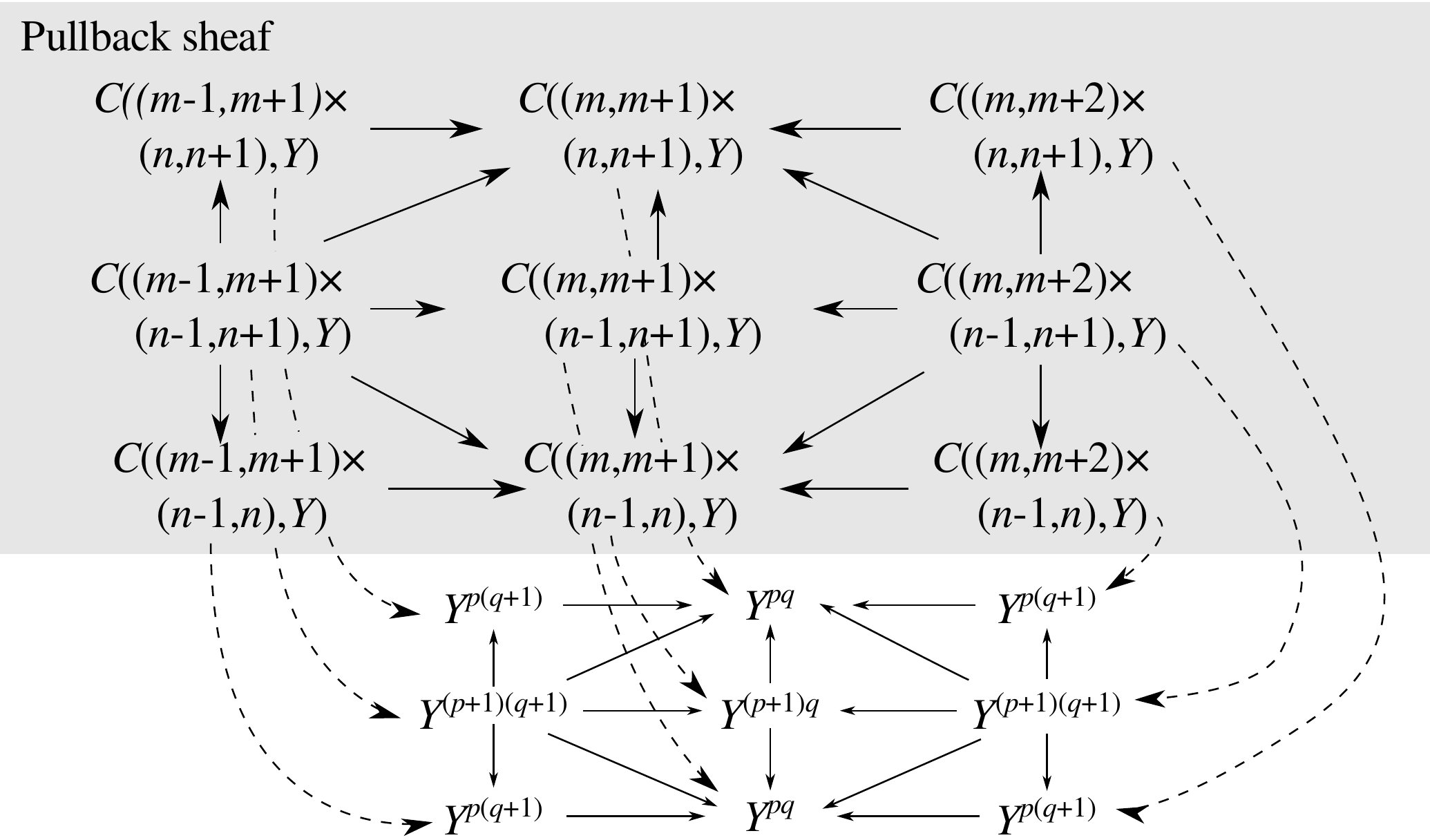}
\caption{The pullback sheaf (shaded) and sampled discretization sheaf (not shaded) for Example \ref{eg:sample_plane} connected by a surjective sheaf morphism (dashed arrows)}
\label{fig:sample_plane_sheaves}
\end{center}
\end{figure}

\begin{example}
  \label{eg:sample_plane}
  Similar to the previous example, if $X=\mathbb{R}^2$, then we can construct a poset $P$ consisting of various rectangular subsets of the plane as shown in Figure \ref{fig:sample_plane_z} and a function $S: P \to {\bf Open}(\mathbb{R}^2)^{op}$ again given by reinterpreting the elements of $P$ as actual subsets of $\mathbb{R}^2$.

  This results in a diagram for $S^* \shf{C}(\mathbb{R}^2,Y)$ like the one shown in the shaded box in Figure \ref{fig:sample_plane_sheaves}.  Again, the space of global sections of $S^* \shf{C}(\mathbb{R}^2,Y)$ is exactly $C(\mathbb{R}^2,Y)$.  Discrete samples of these functions are easily extracted via a sheaf morphism, like the one shown in Figure \ref{fig:sample_plane_sheaves} in which there are $p$ rows and $q$ columns of points on which the functions are evaluated on each unit square.
\end{example}

There are many other choices for $P$ that can be used to discretize $\mathbb{R}^d$ that correspond to cellular decompositions of $\mathbb{R}^d$.

Now let us consider the opposite discretization, which arises by dualization of the previous discretization.  Given $\shf{C}(X,Y)$, we construct the dual sheaf $\dshf{C}(X,Y)$ by taking linear duality of all spaces and maps.  Specifically, for open sets $U \le V$ in $X$ (recall $V \subseteq U$),
\begin{enumerate}
\item The stalk over $U$ is $\dshf{C}(X,Y)(U) = (C(U,Y))^*$, the space of continuous linear functionals $C(U,Y) \to \mathbb{C}$, and
\item The extension from $V$ to $U$ is given by $\dshf{C}(X,Y)(U \le V): (C(V,Y))^* \to (C(U,Y))^*$, the dual of the linear map $C(U,Y)) \to (C(V,Y))$ induced by restricting the domains of the continuous functions.
\end{enumerate}

\begin{example}
Consider the case of $\dshf{C}(X,\mathbb{C})$, whose stalks consist of complex-signed measures that act on continuous functions $C(X,\mathbb{C})$.  Specifically, if $m \in \dshf{C}(X,\mathbb{C})$, then $m$ is a linear functional $C(U,\mathbb{C}) \to \mathbb{C}$, which we can formally write as an integral
\begin{equation*}
  m(f) = \int_U f(x) dm(x).
\end{equation*}
Then, the extension maps of $\dshf{C}(X,\mathbb{C})$ are obtained by extending the measure $m$ by zero.  So if $U \le V$, which means $V \subseteq U$, then for $A \subseteq V$,
\begin{equation*}
  \left(\left(\dshf{C}(U \le V)\right)m\right)(A) = m(A \cap U).
\end{equation*}
\end{example}

If we dualize a surjective sheaf morphism $\shf{C}(X,Y) \to \shf{D}$, we then obtain a morphism between dual sheaves $\dshf{D} \to \dshf{C}(X,Y)$, which plays the role of discretizing of the functions themselves.

Unless $X$ is compact, there is no appropriate hybrid morphism $\dshf{C}(X,Y) \to \shf{C}(X,Y)$ to complete the story, which complicates matters.  There are a variety of ways out of this situation, but the most common one in numerical analysis amounts to considering a more well-behaved subsheaf of $\dshf{C}(X,Y)$ for which a hybrid morphism exists.  For instance, if we restrict our attention to $Y=\mathbb{C}$, the dual sheaf of compactly supported continuous functions $\dshf{C}_c(X,\mathbb{C})$ (see Example \ref{eg:contfunctions2}) is usually the start of a finite elements method.   Each compactly supported continuous function on $U$ is both a continuous function and a continuous linear functional, so each component map of the hybrid morphism $\dshf{C}_c(X,\mathbb{C}) \to \shf{C}(X,\mathbb{C})$ is an identity map as Figure \ref{fig:contfunctions_hybrid} suggests.

\begin{figure}
\begin{center}
\includegraphics[width=3.5in]{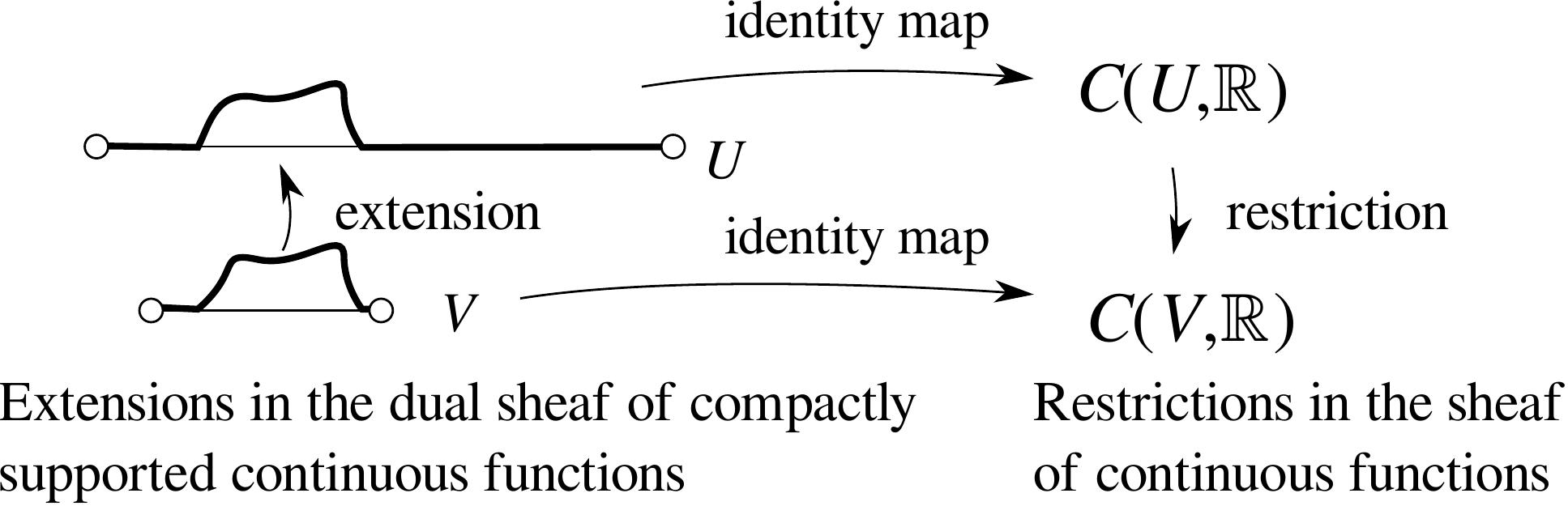}
\caption{The hybrid morphism taking the dual sheaf of compactly supported continuous functions into the sheaf of continuous functions}
\label{fig:contfunctions_hybrid}
\end{center}
\end{figure}

Usually, the poset on which $\dshf{C}_c(X,\mathbb{C})$ is constructed is too large because the topology on $X$ is too fine.  Consider coarsening it to a new poset $P$ along an order preserving map $C: {\bf Open}(X)^{op} \to P$.  This results in a pushforward dual sheaf $C_* \dshf{C}(X,Y)$ (defined below), but we will first address the construction of $P$.  Consider another topological space $Z$ whose points are the same as $X$, and for which the identity map $X \to Z$ is continuous.  This means that the topology of $Z$ is no finer than the topology on $X$.  An appropriate coarsening map $C:{\bf Open}(X)^{op} \to {\bf Open}(Z)^{op}$ is given by
\begin{equation*}
  C(U) = \text{interior}_Z(U),
\end{equation*}
which is well defined since the interiors of sets in a topological space are always uniquely defined.  $C$ is order preserving because $V \subseteq U$ implies $\text{interior}_Z(V) \subseteq \text{interior}_Z(U)$.  

\begin{definition}
\label{df:pushforward_dualsheaf}
  Suppose $f:X\to Y$ is an order preserving function between posets and that $\dshf{R}$ is a dual sheaf on $X$.  The \emph{pushforward} $f_* \dshf{R}$ is a dual sheaf on $Y$ in which
  \begin{enumerate}
    \item Each stalk $(f_* \dshf{R})(c)$ is the space of sections over the set $f^{-1}(c) \subseteq X$, and 
    \item The extension maps $(f_* \dshf{R})(a\le b)$ are given by extending a section $s$ over $f^{-1}(b)$ to one over $f^{-1}(a)$.
  \end{enumerate}
  This construction yields a dual sheaf morphism $f_*\dshf{R} \to \dshf{R}$.
\end{definition}

As in the case of the pullback, we are generally not interested in $C_* \dshf{C}_c(X,\mathbb{C})$, as it serves more as an upper bound on the discretization.  We are more interested in subsheaves $\dshf{D}$ of $C_*\dshf{C}_c(X,\mathbb{C})$, thought of as injective morphisms $\dshf{D} \to C_*\dshf{C}_c(X,\mathbb{C})$.  By composing morphisms, we obtain practical discretizations $\dshf{D} \to C_*\dshf{C}_c(X,\mathbb{C})$ of the original (non-dual) sheaf $\shf{C}(X,\mathbb{C})$.  Evidently $C_* \dshf{C}_c(X,\mathbb{C})$ and $S^* \shf{C}(X,\mathbb{C})$ are quite different, and correspond to very different kinds of discretization methods as the following examples show (compare Examples \ref{eg:sample_line} - \ref{eg:sample_plane}).

\begin{figure}
\begin{center}
\includegraphics[width=4in]{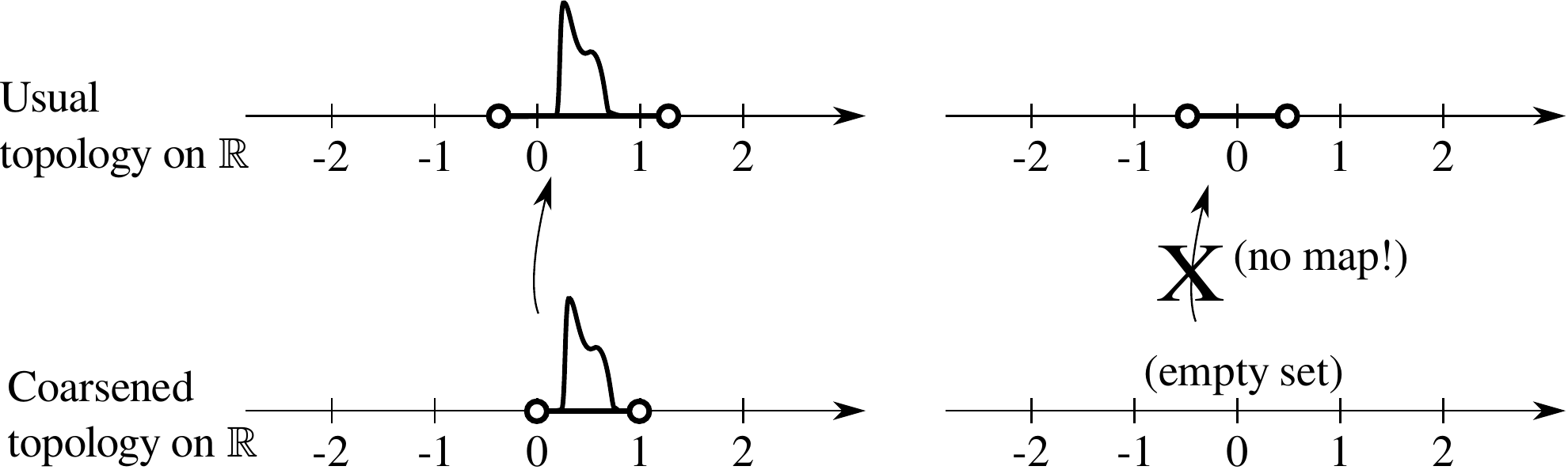}
\caption{The coarsening pushforward morphism on the dual sheaf of compactly supported functions.  Some elements of $\dshf{C}_c(\mathbb{R},\mathbb{R})$ (top row) and their counterparts in $C_c(\mathbb{R},\mathbb{R})$ (bottom row) as described in Example \ref{eg:collapse_line}}
\label{fig:collapse_line1}
\end{center}
\end{figure}

\begin{example}
  \label{eg:collapse_line}
Consider the case of continuous on functions the real line $\mathbb{R}$.  As noted above, Figure \ref{fig:contfunctions_hybrid} shows the hybrid morphism $\dshf{C}_c(\mathbb{R},\mathbb{R})\to \shf{C}(\mathbb{R},\mathbb{R})$.  The next step is to pushforward $\dshf{C}_c(\mathbb{R},\mathbb{R})$ to a coarser topology, such as the topological space $Z=(\mathbb{R},\col{T})$ in which the topology $\col{T}$ is generated by sets of the form $(n-1,n+1)$.  This results in a coarsening map $C:{\bf Open}(\mathbb{R},\text{usual})^{op}\to {\bf Open}(\mathbb{R},\col{T})^{op}$ given by
\begin{equation*}
C(U) = \text{interior}_Z(U),
\end{equation*}
in which small sets get taken to the empty set in ${\bf Open}(\mathbb{R},\col{T})^{op}$.  The relationship between the dual sheaves is suggested by Figure \ref{fig:collapse_line1}.  The dual sheaf $C_*\dshf{C}_c(\mathbb{R},\mathbb{R})$ can be summarized in the diagram
\begin{equation*}
\includegraphics[width=3in]{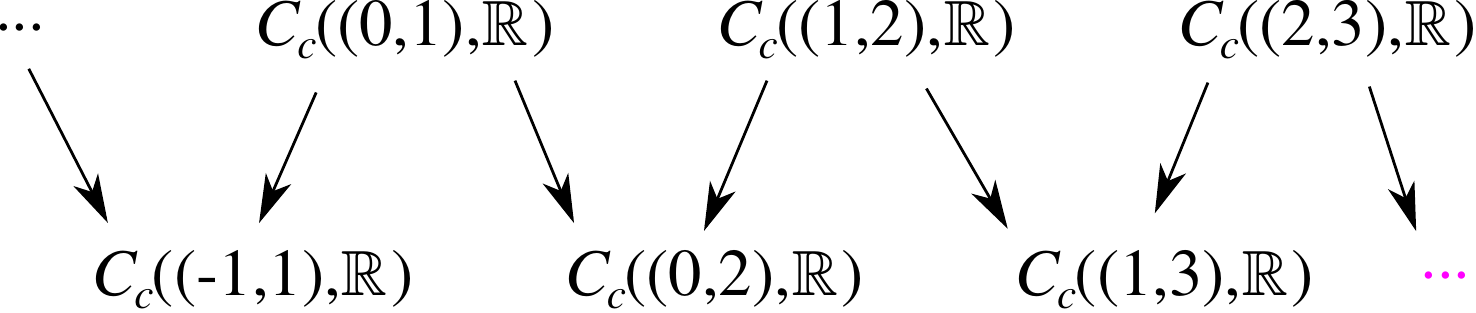}
\end{equation*}
in which the arrows correspond to extending by zero.  Each stalk is still an infinite-dimensional function space, so it is convenient to replace these with smaller, or at least more convenient, spaces.  Ideally, we would like a dual sheaf morphism $\dshf{D}\to C_*\dshf{C}_c(\mathbb{R},\mathbb{R})$, which would provide a smaller description for each stalk.  The primary constraint is that the following kind of diagram commute
\begin{equation*}
\xymatrix{
\dshf{D}((n,n+1)) \ar[d]\ar[r] & C_c((n,n+1),\mathbb{R}) \ar[d]\\
\dshf{D}((n-1,n+1)) \ar[r] & C_c((n-1,n+1),\mathbb{R})\\
}
\end{equation*}
where the vertical arrows correspond to extending by zero.  This means that ideally $\dshf{D}((n,n+1))$ is a subspace of $\dshf{D}((n-1,n+1))$.  This requirement is neatly satisfied by multi-scale functions, such as continuous wavelet bases or spline bases.  For instance, we could let $\dshf{D}((n,n+1))$ be spanned by the set of raised cosines
\begin{equation*}
1-\cos (2m \pi x) \text{ for } x \in (n,n+1)
\end{equation*}
and $\dshf{D}((n-1,n+1))$ be spanned by similar kinds of functions.
\end{example}

\begin{example}
  \label{eg:collapse_plane}
  Rectangular sampling of the plane is achieved in exactly the same sort of way as in the previous example.  One may obtain the diagram of the dual sheaf $C_* \dshf{C}_c(\mathbb{R}^2,\mathbb{R})$ by reversing the arrows in Figure \ref{fig:sample_plane_z} and treating each arrow as an extension by zero.  Although the resulting dual sheaf is fairly large, more practical spline or wavelet bases are easily constructed in a dual sheaf $\dshf{D}$ along with a dual sheaf morphism into $C_* \dshf{C}_c(\mathbb{R}^2,\mathbb{R})$.
\end{example}

\begin{lemma}
  \label{lem:pushpull}
  Pulling back is a contravariant functor, while and pushing forward is a covariant functor.  Explicitly, if $f:X \to Y$ and $g: Y \to Z$ then $f^* g^* = (g \circ f)^*$ and $g_* f_* = (g \circ f)_*$.
\end{lemma}
\begin{proof}
Suppose that $\shf{S}$ is a sheaf on $Z$, which we will pull back to $X$.  According to Definition \ref{df:pullback}, we can construct pullbacks in sequence $\shf{S} \to g^*\shf{S} \to f^* g^* \shf{S}$ or all at once $\shf{S} \to (g \circ f)^* \shf{S}$.  In both cases, the 
  \begin{enumerate}
  \item Stalks are given by $\left(f^* (g^* \shf{S})\right)(x) = \left(g^* \shf{S}\right)(f(x)) = \shf{S}(g(f(x)) = (g \circ f)^* \shf{S} (x)$, and
  \item Restrictions are given by $\left(f^* (g^* \shf{S})\right)(x \le y) = (g^*\shf{S})(f(x) \le f(y)) = \shf{S}((g(f(x)) \le g(f(y))) = (g \circ f)^* \shf{S} (x \le y)$,
  \end{enumerate}
  which establishes the first statement.

  Suppose that $\dshf{R}$ is a dual sheaf on $X$, which we will push forward along $f$ and $g$.  Following Definition \ref{df:pushforward_dualsheaf}, we can construct the sequence of dual sheaf morphisms $g_*(f_* \dshf{R}) \to f_* \dshf{R} \to \dshf{R}$.  We can also construct a morphism $(g \circ f)_* \dshf{R} \to \dshf{R}$.  If we use the notation $\dshf{R}(f^{-1}(x))$ to represent the space of sections of $\dshf{R}$ over $f^{-1}(x)$, we have the following
\begin{equation*}
\left(g_*(f_* \dshf{R})\right)(x) =(f_* \dshf{R})(g^{-1}(x))=\dshf{R}\left(f^{-1}(g^{-1}(x))\right)=\dshf{R}\left((g\circ f)^{-1}(x)\right)=((g\circ f)_*\dshf{R})(x)
\end{equation*}
for the stalks.  A similar derivation establishes that the extensions in both dual sheaves are the same.
\end{proof}

\section{Simultaneous systems of equations}
\label{sec:simultaneous}

Consider a multi-model system that consists of a set of variables $V$ whose values lie in sets $W_v$ for $v\in V$, and are interrelated through a set of equations $E$.  Each equation $e\in E$ specifies a list of variables $V_e \subset V$ and a subset $S_e \subseteq \prod_{v \in V_e} W_v$ of solutions.

\begin{example}
  \label{eg:aggregation_1}
Consider the following system of equations in three variables $V=\{x,y,z\}$
\begin{eqnarray*}
  x^2 + y^2 - 4 &=& 0,\\
  y &=& x^2 + z^2 + 1.
\end{eqnarray*}
In our formalism, the spaces of values for the variables should be specified.  For instance, $W_x = \mathbb{R}$, $W_y = \mathbb{R}$, $W_z = \mathbb{R}$.  The two equations need to be labeled, so something like $E=\{1,2\}$ will do.  Given these labels, the set of variables involved in each equation are $V_1 = \{x,y\}$ and $V_2 = \{x,y,z\}$.  Given that, the set of solutions for each are easily described, namely $S_1 = \text{circle of radius }2$ and $S_2 = \text{paraboloid}$.
\end{example}

There are natural projection functions $\pr_x: \prod_{v \in V_e} W_v \to W_x$ for each $x \in V_e$.  Since these projection functions restrict to functions on $S_e$, like $\pr_x: S_e \to W_x$, it is natural to define the following poset structure.  Let $P = V \sqcup E$, so elements of $P$ are either variables or equations, and define $e \le v$ if $v \in V_e$.  This is generally called a \emph{factor graph} in the literature.  If we assume that $\le$ is reflexive, then this defines a partial order on $P$.  A sheaf $\shf{E}'$ on $(P,\le)$ can then be given by specifying that
\begin{enumerate}
\item $\shf{E}'(v)=W_v$ for each variable $v$,
\item $\shf{E}'(e) = \prod_{v \in V_e} W_v$ for each equation $e$, and
\item $\shf{E}'(e \le v) = \pr_v$ whenever $e \le v$.  
\end{enumerate}

\begin{definition}
  The sheaf $\shf{E}'$ is called an \emph{aggregation sheaf} associated to the collection of variables $V$ and equations $E$.
\end{definition}

\begin{example}
  \label{eg:aggregation_2}
  Continuing Example \ref{eg:aggregation_1}, the system of equations yields the following diagram for the sheaf $\shf{E}'$
\begin{equation*}
  \xymatrix{
    \shf{E}'(x) = W_x = \mathbb{R}& \shf{E}'(y) = W_y = \mathbb{R} & \shf{E}'(z) = W_z = \mathbb{R} \\
    \shf{E}'(1) = W_x \times W_y = \mathbb{R}^2 \ar[u]\ar[ur] & \shf{E}'(2) = W_x \times W_y \times W_z = \mathbb{R}^3 \ar[ul] \ar[u] \ar[ur] \\
     }
\end{equation*}
The poset structure for $P=\{x,y,z,1,2\}$ is clear: the top row is for the variables and the bottom row is for equations.  Each of the arrows in the diagram is a projection onto the space of values of a variable.  The sections of $\shf{E}'$ are determined by elements of $W_x \times W_y \times W_x \times W_y \times W_z$ in which the two $x$ components agree and the two $y$ components agree.  
\end{example}

\begin{proposition}
Assuming that each variable $v$ appears in at least one equation, the set of sections of $\shf{E}'$ is in one-to-one correspondence with $\prod_{v \in V} W_v$.
\end{proposition}
\begin{proof}
  Certainly each section of $\shf{E}'$ specifies all values of all variables, since each variable is in $P$ and its stalk is its respective space of values.  On the other hand, specifying the value of each variable certainly specifies a section of $\shf{E}'$.  
\end{proof}

\begin{definition}
Clearly the aggregation sheaf $\shf{E}'$ does not account for the \emph{actual equations}, since it merely specifies which variables are involved.  To remedy this information loss, let us construct the following subsheaf $\shf{E}$ of $\shf{E}'$, called the \emph{solution sheaf} of the system of equations:
\begin{enumerate}
\item $\shf{E}(v)=W_v$ for each variable $v$,
\item $\shf{E}(e) = S_e$ for each equation $e$ (recall that $S_e \subseteq \prod_{v \in V_e} W_v$ is the set of solutions to $e$), and
\item $\shf{E}(e \le v) = \pr_v$ whenever $e \le v$.  
\end{enumerate}
\end{definition}

\begin{proposition}
  \label{prop:simultaneous}
  Sections of $\shf{E}$ consist of solutions to the simultaneous system of equations.
\end{proposition}
\begin{proof}
  A section $s$ of $\shf{E}$ specifies an element of $s(e) \in S_e$ for each equation $e \in E$ which satisfies that equation.  Conversely, if we start with a solution to the simultaneous system of equations, that is a specification of an element $x \in \prod_{v\in V} W_v$ for which the projection of $x$ onto $\prod_{v \in V_e} W_v$ lies in $S_e$.  This can be translated to an assignment onto each variable $v\in V$ given by
  \begin{equation*}
    s(v) = \pr_v x,\text{ and } s(e) = \pr_{\shf{E}'(e)} x
  \end{equation*}
  which by construction we observe $s(e) \in S_e = \shf{S}(e)$.
\end{proof}

\begin{example}
Continuing where we left off with Example \ref{eg:aggregation_2}, the sheaf $\shf{E}$ is a subsheaf of $\shf{E}'$, whose diagram is given by
\begin{equation*}
  \xymatrix{
    \shf{E}(x) = \mathbb{R}& \shf{E}(y) = \mathbb{R} & \shf{E}(z) = \mathbb{R} \\
    \shf{E}(1) = \{(x,y):x^2 + y^2 - 4 = 0\} \ar[u]\ar[ur] & \shf{E}(2) = \{(x,y,z):y = x^2 + z^2 + 1\} \ar[ul] \ar[u] \ar[ur] \\
     }
\end{equation*}
where again the arrows are coordinate projections.  Sections are determined by the elements $(x,y,z) \in \mathbb{R}^3$ that lie on the intersection between the cylinder $\{(x,y,z):x^2 + y^2 - 4 = 0\}$ (notice that $z$ is now present!) and the paraboloid $\{(x,y,z):y = x^2 + z^2 + 1\}$.  
\end{example}

\begin{example}
  \label{eg:stoichiometry}
The sheaf model of multi-equation systems is helpful in organizing complicated systems, and is related to labeled, permutation-directed hypergraphs.  For instance, labeled, directed hypergraphs\footnote{A \emph{hypergraph} is literally a set of sets of vertices.  Each element of a hypergraph is called a \emph{hyperedge}.  A hypergraph is given a direction by specifying the order of vertices in each hyperedge.} can be used to model systems of stochiometric equations \cite{Benko_2009}.  The interrelations between reagents and reactions give rise to a system of equations, which in turn defines a solution sheaf.

For instance, a very simple model of photosynthesis and combustion is given by the two reactions
\begin{eqnarray*}
p: CO_2 +2 H_2 O &\to& CH_2 O + O_2,\\
c: 2 H_2 + O_2 &\to& 2 H_2 O.
\end{eqnarray*}
In order to encode this as a system of equations, we consider the set of concentrations of each compound $\{CO_2, H_2O, CH_2O, O_2, H_2\}$ and the two reactions $\{p, c\}$.  Thinking of the compounds as vertices and the reactions as hyperedges, we are led to consider the hypergraph 
\begin{equation*}
\{p=[CO_2,H_2O,CH_2O,O_2],c=[H_2,O_2,H_2O]\}
\end{equation*}
where the square brackets indicate that order of vertices is important within a hyperedge.  Diagrammatically, one usually thinks of the inclusion structure of a hypergraph, leading to the diagram
\begin{equation*}
\xymatrix{
& p & & c  &\\
CO_2 \ar[ur]& H_2O \ar[u] \ar[urr] & CH_2O \ar[ul] & O_2 \ar[u] \ar[ull]& H_2 \ar[ul] \\
}
\end{equation*}
because the arrows represent subset relations.  However, for the actual values of the variables, concentrations of the compounds, it is more natural to consider the dual diagram in which the arrows represent projections.  In either case, the resulting diagram is a finite poset.  The reactions can be encoded in a sheaf by the use of appropriate spaces of values.  For simplicity, let us consider the state of chemical equilibrium, in which each concentration is a constant, non-negative real number.  The stalk over $p$ or $c$ should be a subspace on which the reaction equation is satisfied.  Namely, if
\begin{equation*}
S_p = \{(a,b,c,d) \in (\mathbb{R}^+)^4 : a + 2 b - c - d = 0\}
\end{equation*}
and
\begin{equation*}
S_c = \{(a,b,c) \in (\mathbb{R}^+)^4 : 2a + b - 2c = 0\},
\end{equation*}
then the sheaf describing the chemical equilibrium is
 \begin{equation*}
\xymatrix{
\mathbb{R}^+& \mathbb{R}^+& \mathbb{R}^+& \mathbb{R}^+& \mathbb{R}^+\\
    & S_p \ar[ur] \ar[u] \ar[ul] \ar[urr]  &      & S_c \ar[ull] \ar[u] \ar[ur] &    \\
}
\end{equation*}
in which all arrows are coordinate projections.
\end{example}

We have thus far been concerned with systems of arbitrary equations, but often there is more structure available.  When this happens the stalks of the sheaf $\shf{E}$ over the variables can be reduced in size, which results in computational savings.  Many numerical approximation schemes (and other models, too) are written in an explicit form, in which each equation looks like
\begin{equation*}
  v_{n+1} = f(v_1,\dots, v_n).
\end{equation*}
In this case, one often represents the relationship among the variables using a \emph{dependency graph}.
\begin{definition}
  \label{def:explicit}
  A system of equations $E$ on variables $V$ is called \emph{explicit} if there is an injective function $\gamma: E \to V$ selecting a specific variable from each equation so that each equation $e \in E$ has the form
  \begin{equation*}
    \gamma(e) = f_e(v_1,\dotsc,v_n),
  \end{equation*}
  so that $\gamma(e)\in V_e$ and $\gamma(e) \notin \{v_1,\dotsc,v_n\}$.  Any variable outside the image of $\gamma$ is said to be \emph{free} or \emph{independent}. Those variables in the image of $\gamma$ are called \emph{dependent}.

A \emph{variable dependency graph} for an explicit system is a directed graph $G$ whose vertices are given by the union $E \cup (V \backslash \gamma(E))$ consisting of the set of equations and free variables, such that the following holds: 
\begin{enumerate}
\item Free variables have in-degree zero,
\item If $e$ is a vertex of $G$ corresponding to an equation whose incoming edges are given by $(e_1 \to e)$, ..., $(e_n \to e)$, then the equation $e\in E$ is of the form
  \begin{equation*}
    \gamma(e) = f_e(\gamma(e_1),\dotsc,\gamma(e_n)),
  \end{equation*}
where we have abused notation slightly to allow $\gamma(v) = v$ for a free variable $v$.
\end{enumerate}
\end{definition}

\begin{example}
  \label{eg:lorenz3}
The Lorenz system defined by \eqref{eq:lorenz} in Section \ref{sec:diagram} is an explicit system.  Its dependency graph is shown in Figure \ref{fig:lorenz1}.
\end{example}

\begin{example}
  \label{eg:aggregation_explicit_1}
Consider the system of equations given by 
\begin{eqnarray*}
u_1 &=& f(u_2,u_3),\\
u_2 &=& g(u_3,u_4).
\end{eqnarray*}
Notice that this is an explicit system with two free variables $u_3, u_4$ and no over-determined variables.  The variable dependency graph for this system is given by
\begin{center}
\includegraphics[width=1.25in]{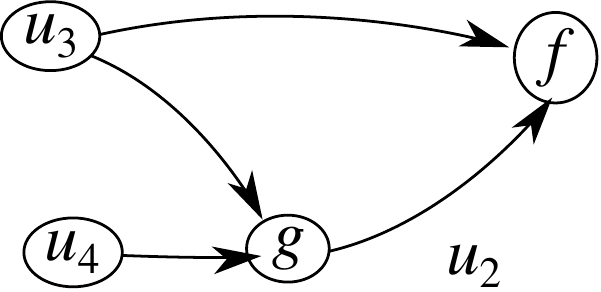}
\end{center}
Notice that the non-free variables $u_1, u_2$ do not appear as vertices in the variable dependency graph.  Rather, $u_2$ is present as the edge out of $g$, while $u_1$ is not shown at all.
\end{example}

Explicit systems need not have acyclic dependency graphs, as the next example shows.

\begin{example}
\label{eg:aggregation_explicit_1a}
The explicit system given by the system
\begin{center}
\includegraphics[width=2in]{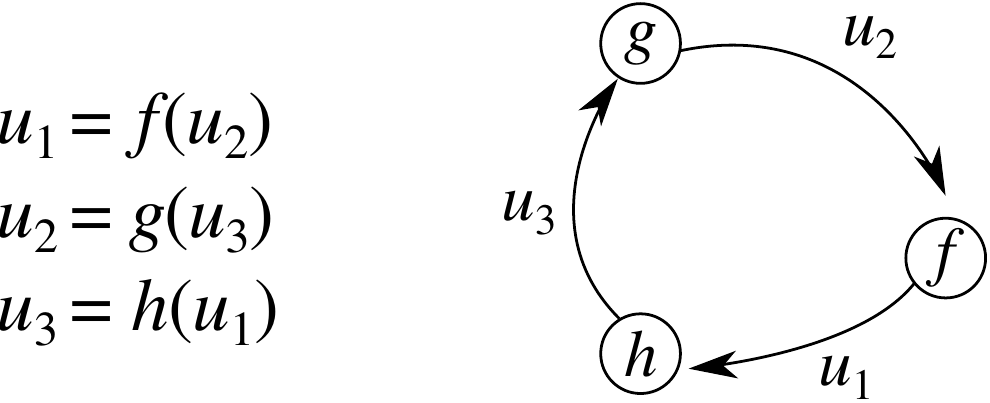}
\end{center}
has a variable dependency graph with a cycle:
\end{example}

\begin{definition}
If $E$ is an explicit system of equations with variables in $V$, then we can construct the \emph{explicit solution sheaf} $\shf{G}$ whose sections are the simultaneous solutions of $E$ using a slight modification of the recipe for $\shf{E}$.  The underlying poset for $\shf{G}$ is still given by the union of the variables and the equations, but the stalks and restrictions are different
\begin{enumerate}
\item $\shf{G}(v) = W_v$ for each variable $v\in V$, just as before
\item $\shf{G}(e) = \prod_{x \in V_e \backslash \gamma(e)} W_x$ (recall that $\gamma(e) \in V_e$),
\item $\shf{G}(e \le \gamma(e)) = f_e$, and
\item $\shf{G}(e \le v): \prod_{x \in V_e \backslash \gamma(e)} W_x \to W_v$ is given by an appropriate projection if $v \not= \gamma(e)$.
\end{enumerate}
\end{definition}

\begin{example}
  \label{eg:lorenz4}
The explicit solution sheaf for the Lorenz system defined by \eqref{eq:lorenz} in Section \ref{sec:diagram} is shown in Figure \ref{fig:lorenz2}.
\end{example}

\begin{example}
  \label{eg:aggregation_explicit_2}
  Continuing Example \ref{eg:aggregation_explicit_1}, the explicit solution sheaf $\shf{G}$ has diagram at left below, while the solution sheaf $\shf{E}$ has diagram at right below,
\begin{center}
\includegraphics[width=3.75in]{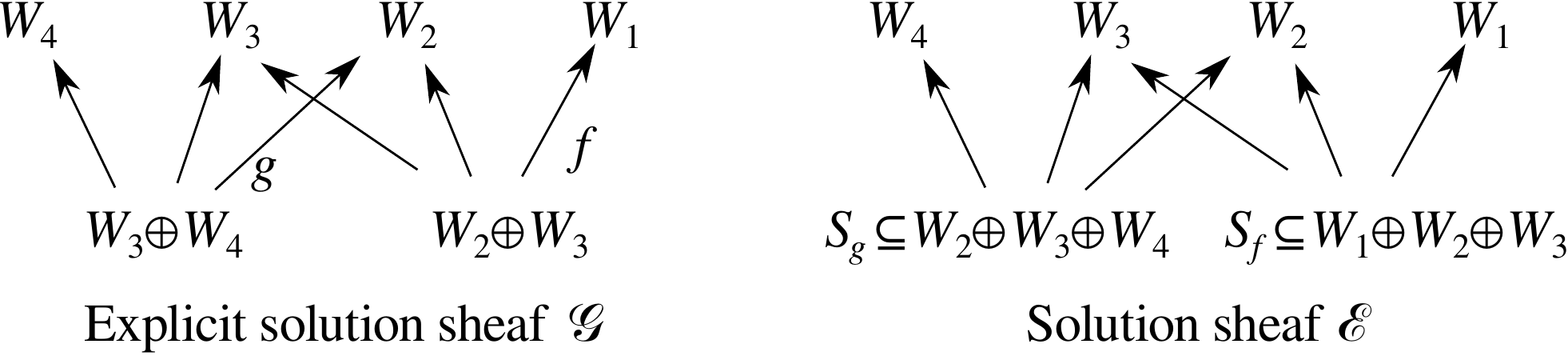}
\end{center}  
where the unlabeled arrows are projection functions.  Notice that the primary difference is in the stalks over the equations; the explicit solution sheaf has a somewhat simpler construction.
\end{example}

\begin{proposition}
  The sections of an explicit solution sheaf $\shf{G}$ are in one-to-one correspondence with the simultaneous solutions of its system of equations.
\end{proposition}
\begin{proof}
  Suppose that $e$ is an equation in the explicit system of the form
  \begin{equation*}
      \gamma(e) = f_e( \gamma(e_1), \dotsc, \gamma(e_n)).
  \end{equation*}
It suffices to notice that $S_e=\{(v_1,\dotsc,v_n,f_e(v_1,\dotsc,v_n)): v_i \in W_{\gamma(e_i)}\}$ so the Proposition follows directly from Proposition \ref{prop:simultaneous}.
\end{proof}

\section{Ordinary differential equations}
\label{sec:ode}

The framework developed in the previous section works equally well for differential equations.  Differential equations give rise to sheaves of solutions \cite{Ehrenpreis_1956,Spencer_1969}, which admit various analytical techniques.  Consider the case of an autonomous ordinary differential equation given by
\begin{equation}
  \label{eq:ex:diff_1}
  u' = f(u),
\end{equation}
where $u \in C^1(\mathbb{R}, \mathbb{R}^d)$ is a continuously differentiable function.  We have essentially two options: to consider $u$ and $u'$ as two separate variables or to consider them as one variable.  Considering them as one variable amounts to rewriting \eqref{eq:ex:diff_1} as
\begin{equation*}
  0=F(u) = f(u) - \frac{d}{dt}u.
\end{equation*}
Then, the solutions of \eqref{eq:ex:diff_1} are sections of the sheaf given by the diagram
\begin{equation*}
  \xymatrix{
    C^1(\mathbb{R},\mathbb{R}^d)\\
    \{u : F(u) = 0\} \subseteq C^1(\mathbb{R},\mathbb{R}^d) \ar[u]^{\id} \\
  }
\end{equation*}
From an analytic standpoint, this kind of sheaf is not particularly helpful as too much of the structure of \eqref{eq:ex:diff_1} has been ``buried'' in the function $F$.

At first glance, considering $u$ and $u'$ as separate variables yields a similar construction, namely
\begin{equation*}
  \xymatrix{
    C^0(\mathbb{R},\mathbb{R}^d) & C^1(\mathbb{R},\mathbb{R}^d)\\
    C^1(\mathbb{R},\mathbb{R}^d)\ar[u]^f \ar[ur]^{\id}\\}
\end{equation*}
where $u'$ is on the top left and $u$ is on the top right.  Although solutions of \eqref{eq:ex:diff_1} are indeed sections of this sheaf, the converse is not true: there are many sections that are not solutions.  The issue is that there is another equation that links $u$ and $u'$, namely that they are related through differentiation.  Including this relationship leads to the sheaf $\shf{S}$
\begin{center}
  \includegraphics[width=1.5in]{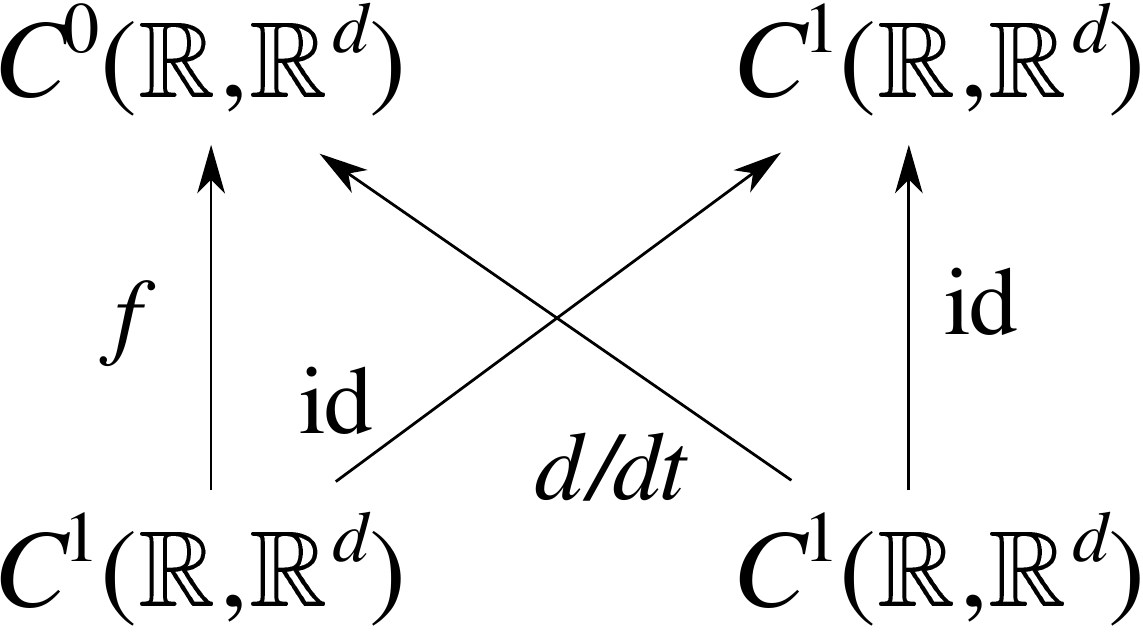}
\end{center}
whose sections are precisely solutions to the differential equation.

A benefit of formulating a differential equation as a sheaf is that it exposes a number of structural properties when we try to approximate it.  For instance, we can obtain consistency conditions for numerical methods.  Suppose that we wanted to discretize $u$ in finding our solution to \eqref{eq:ex:diff_1}.  If we had an actual solution $u$, this would merely be a function $\Delta_h: C^0(\mathbb{R},\mathbb{R}^d) \to (\mathbb{R}^d)^{\mathbb{Z}}$, taking functions to sequences, given by something like
\begin{equation*}
  (\Delta_h u)_n = u( hn )
\end{equation*}
for some step size $h>0$.  We would like to apply this discretization to every stalk in our sheaf, to obtain a sheaf morphism of the form given by the dashed lines in exactly the form posited in Section \ref{sec:discretization}
\begin{center}
  \includegraphics[width=3in]{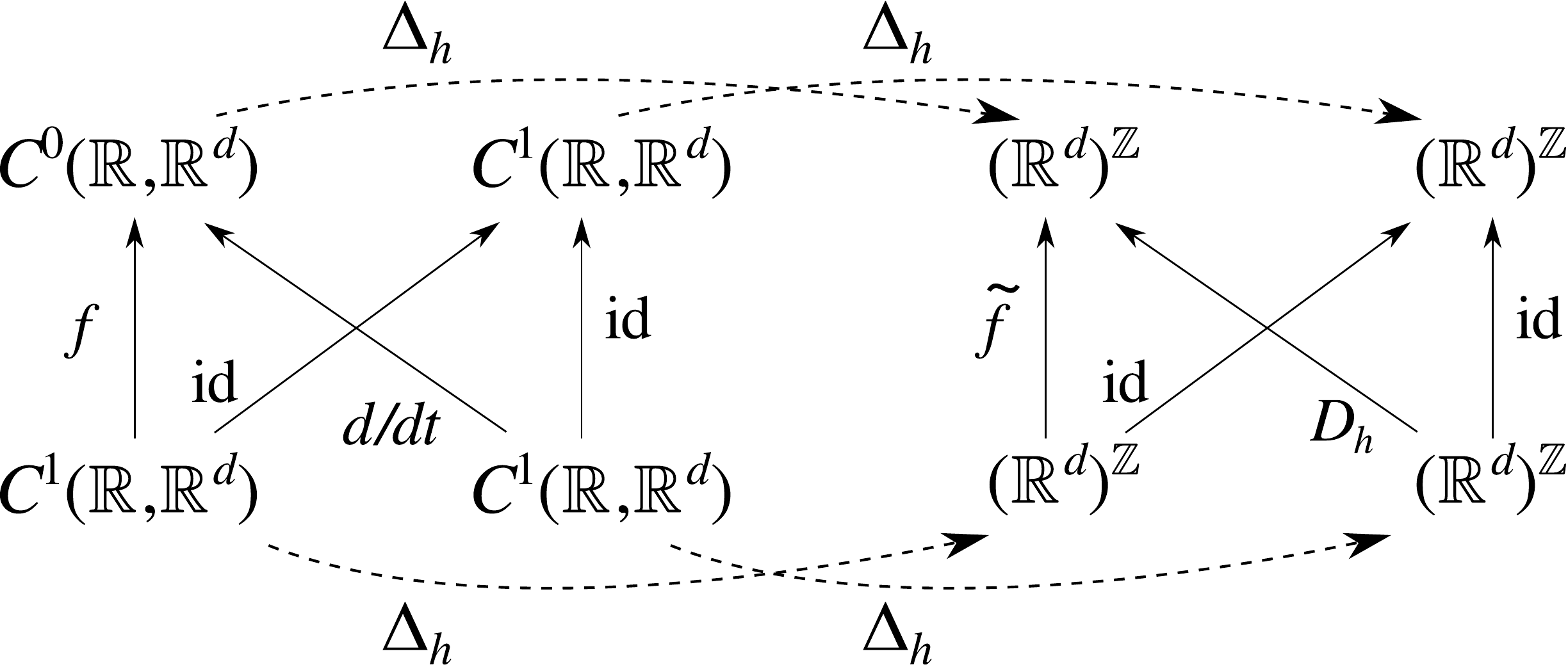}
\end{center}
in which $D_h$ is a discretized derivative and $\tilde{f}$ is a discretized version of $f$.  The sheaf morphism condition asserts that two compatibility conditions hold; the first one is straightforward, that
\begin{equation*}
  \tilde{f} \circ \Delta_h = \Delta_h \circ f,
\end{equation*}
which asserts a kind of translation invariance on the equation.  In particular, if $f$ is given by a function $\mathbb{R}^d \to \mathbb{R}^d$, then $\tilde{f}$ is given by
\begin{equation*}
  \tilde{f} u_n = \left(fu\right)(nh).
\end{equation*}
The other condition that
\begin{equation*}
  D_h \circ \Delta_h = \Delta_h \circ \frac{d}{dt}
\end{equation*}
is considerably more subtle.  Written in more simple notation, for a function $u$, this means that
\begin{equation}
  \label{eq:discrete_deriv}
  u'(nh) = D_h u_n.
\end{equation}
That is, the discretization $D_h$ of the derivative operator \emph{exactly} recovers the derivative.  Of course this is an unreasonable requirement, so we usually expect \eqref{eq:discrete_deriv} to hold only approximately!  There are various ways to manage this issue, which are discussed at length elsewhere.  The usual approach is to attempt to minimize the discretization error in some fashion, by trying to ensure that the operator norm  
\begin{equation*}
  \left\|D_h \circ \Delta_h - \Delta_h \circ \frac{d}{dt}\right\|
\end{equation*}
remains is small.  Although this method is often effective in ordinary differential equations, it can cause problems for partial differential equations.  We point the interested reader to the work of Arnold \cite{Arnold_2006,Arnold_2010} on finite exterior differential systems in which consistency equations like \eqref{eq:discrete_deriv} are enforced. 

Discretizing functions into sequences is formally convenient, but often it is useful to be a bit more explicit.  This is quite helpful when we generalize to partial differential equations in Section \ref{sec:pde}, since we will want to handle various irregular discretizations of the domain.  As discussed in Section \ref{sec:discretization}, discretization of the domain still amounts to a morphism out of the sheaf describing the differential equation and is usually related to an appropriate pullback.  Because of the need to describe the construction of $D_h$ more explicitly, the discretized sheaf must become somewhat more complicated.  It is not unreasonable to suppose that $D_h$ involves only finitely many terms when approximating a derivative.  Let us consider the case where $D_h$ is determined by $N$ terms
\begin{equation*}
  v_n = D_h(u_{n-i_1}, \dotsc, u_{n-i_N}).
\end{equation*}
This specifies a sequence of equations, which can be represented diagrammatically as
\begin{center}
  \includegraphics[width=2.5in]{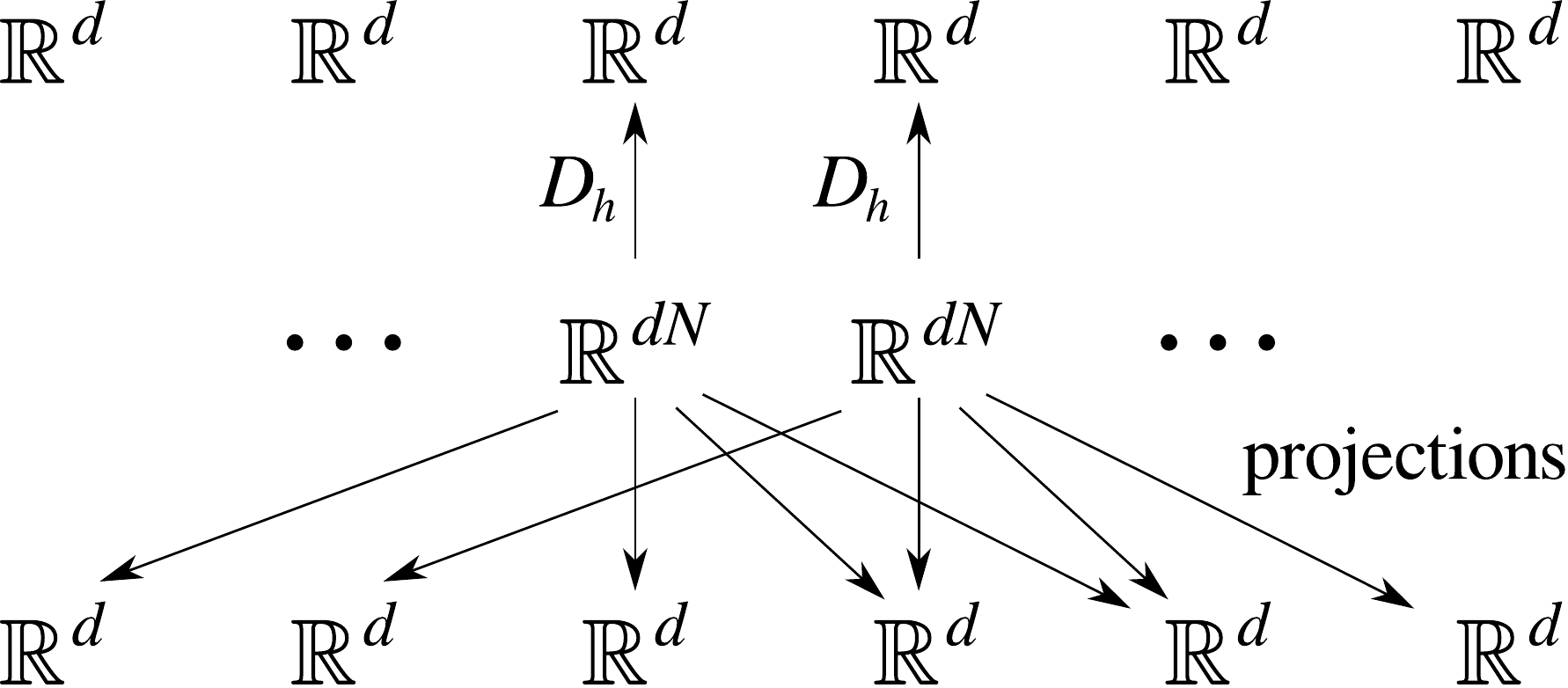}
\end{center}
in which the $v_n$ are given in the top row, while the $u_n$ are given in the bottom row.  Exactly the same kind of diagram is used to specify the new formulation of the function $\tilde{f}$, which in a similar way has its term-by-term dependencies called out explicitly.  Again, the appropriate consistency requirements for the discretization are encoded by a sheaf morphism from the sheaf of solutions of the differential equation.  As might be imagined, the resulting diagram is quite complicated. But for a single timestep, with $N=4$, and $f: \mathbb{R}^d\to\mathbb{R}^d$, the diagram is
\begin{center}
  \includegraphics[width=3.5in]{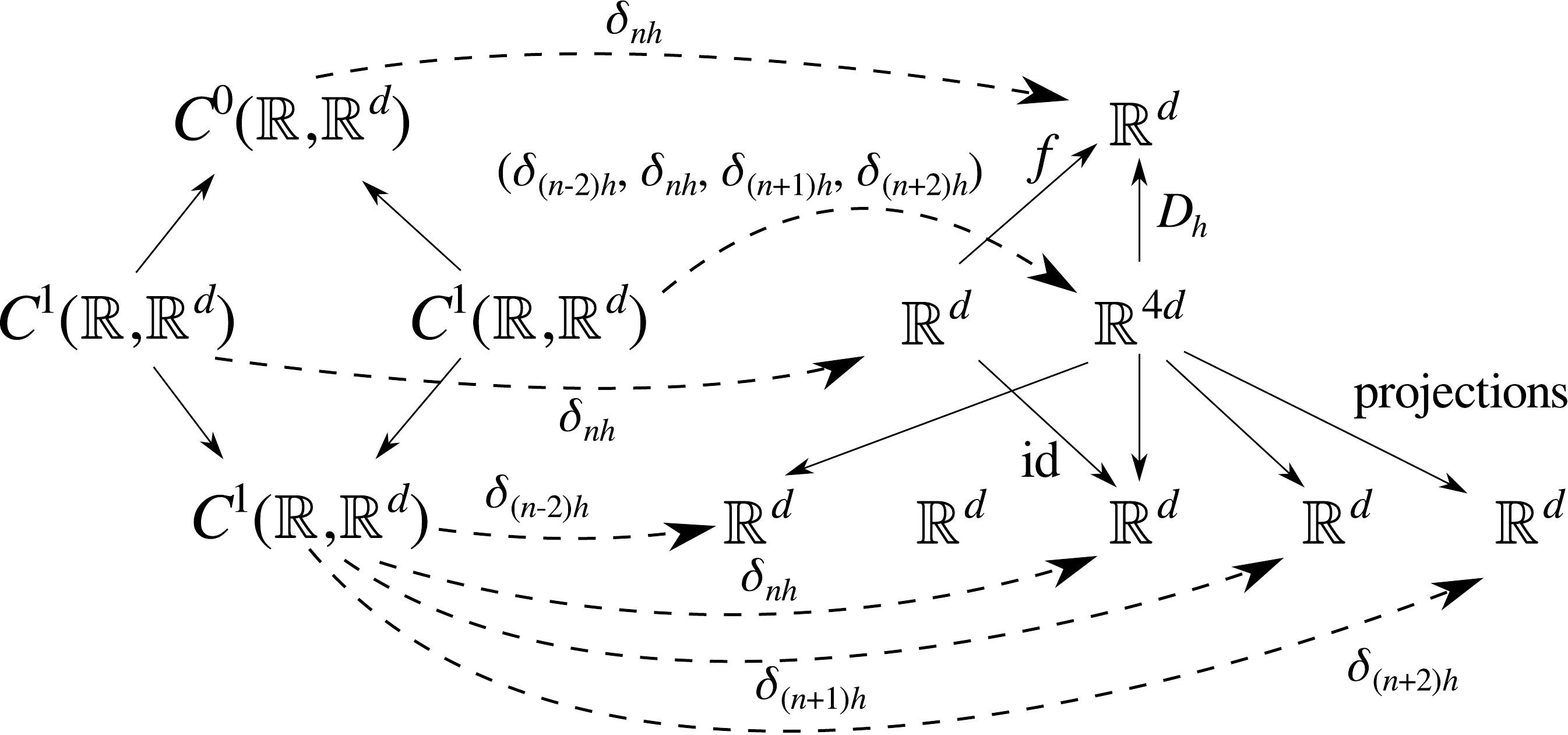}
\end{center}
where $\delta_{x}$ evaluates a function at $x$. 

A reverse way to look at the solutions of differential equations is instead to consider morphisms \emph{into} the sheaf encoding of the equation.  This provides a connection to finite element methods, as suggested in Section \ref{sec:discretization}.  We again use the same basic diagram.  Suppose that we have an $N$ dimensional subspace $B$ of $C^1(\mathbb{R},\mathbb{R}^d)$.  This can be interpreted as a linear function $b: \mathbb{R}^N \to B \subseteq C^1(\mathbb{R},\mathbb{R}^d)$.  Then, the appropriate morphism can be written
\begin{center}
  \includegraphics[width=2.5in]{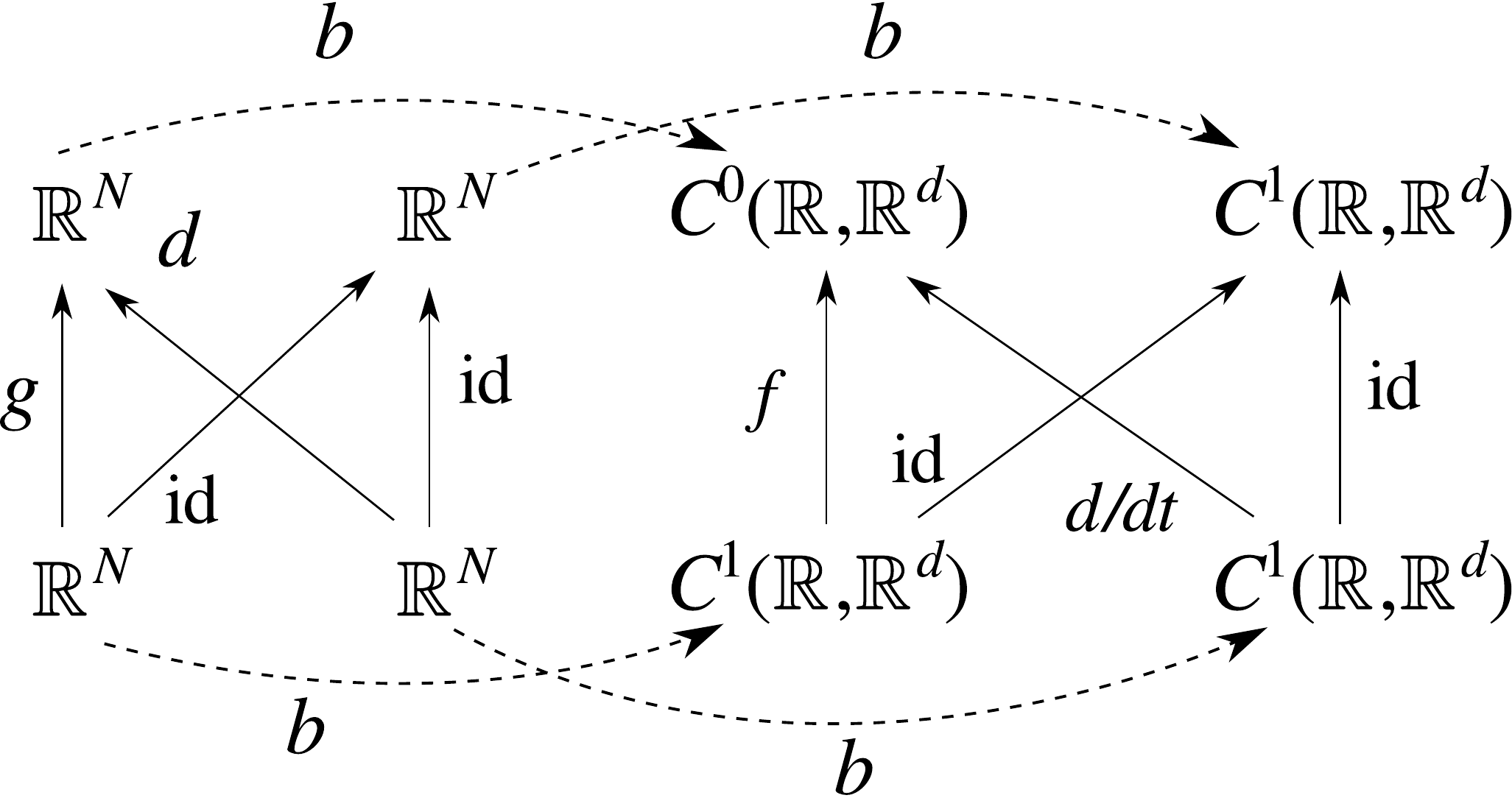}
\end{center}
As before, we can derive two consistency constraints from the commutativity of the diagram.  Unlike the previous case, the derivative constraint
\begin{equation*}
  b \circ d = \frac{d}{dt} \circ b 
\end{equation*}
is now easy to satisfy by choosing our subspace $B$ of $C^1(\mathbb{R},\mathbb{R}^d)$ so that it is invariant with respect to differentiation.  For instance, a basis of monomials $\{1,x,x^2, \dotsc\}$ or trigonometric functions works well enough.  The other constraint, that
\begin{equation*}
  b \circ g = f \circ b
\end{equation*}
is much harder to satisfy, though, because it requests that the subspace $B$ is invariant under $f$.  If $f$ is nonlinear, this is unlikely to be true!  Generally this problem must be handled by selecting $B$ and $g$ to minimize the difference between the two sides of the above equation.

\section{Partial differential equations}
\label{sec:pde}

Partial differential equations can be handled similarly to ordinary differential equations.  As before, the procedure is to list all relevant variables and equations, add appropriate auxiliary equations to relate functions to their (partial) derivatives, and then encode these as sheaves through a factor graph.  Suppose that we are interested in a system of partial differential equations on a manifold $M$.  The equations take the form
\begin{equation}
  \label{eq:pde_basic}
  0=f_i(u(x),\partial_{I_{i1}}u(x),\partial_{I_{i2}}u(x), \dotsc) \text{ for all }x\in M,
\end{equation}
where the $I_{ij}$ are multi-indices specifying the particular partial derivatives involved.  For each partial derivative $\partial_{I_{ij}}u$, we supply an auxiliary equation
\begin{equation*}
  v_{ij}(x) = \frac{\partial}{\partial x_{I_{ij}}} u(x)\text{ for all }x\in M, 
\end{equation*}
so that we can reinterpret the original equations as being in terms of $v_{ij}$.  Now given that $u$ and each $v_{ij}$ lies in a space of appropriately differentiable functions on $M$, we can simply follow the recipe in Section \ref{sec:simultaneous} to obtain a sheaf $\shf{S}$ whose global sections are the solutions to the system \eqref{eq:pde_basic}.

Recalling the discussion in Section \ref{sec:discretization}, discretization of \eqref{eq:pde_basic} involves specifying an appropriate cell decomposition of $M$ to which the sheaf of solutions can be moved.  There are essentially two ways to do this: (1) by looking at a cellular stratification of $M$ and (2) by looking at a topology on the cells of the stratification.  The first way leads to a finite differences model via a sheaf morphism, while the second leads to a finite elements model via a hybrid morphism.  

Let us see how one can construct the sheaf $\shf{S}$ and a discretization of it by way of an example.

\begin{example}
  Consider the case of the Helmholtz equation
  \begin{equation*}
    \Delta u + k^2 u = 0
  \end{equation*}
  on a Riemannian manifold $M$.  Following the recipe in Section \ref{sec:simultaneous}, we obtain a sheaf given by the diagram
  \begin{equation*}
    \xymatrix{
      C^\infty(M,\mathbb{R}) & C^\infty(M,\mathbb{R}) \\
      S\subseteq C^\infty(M,\mathbb{R}) \times C^\infty(M,\mathbb{R}) \ar[u]^{\pr_1} \ar[ur]^(0.25){\pr_2} & C^\infty(M,\mathbb{R})\ar[ul]_(0.25){\Delta} \ar[u]^{\id}\\
      }
  \end{equation*}
  in which the set $S$ is given by
  \begin{equation*}
    S = \{(u,v) \in C^\infty(M,\mathbb{R}) \times C^\infty(M,\mathbb{R}) : u(x) + k^2 v(x) = 0\text{ for all }x\in M\}.
  \end{equation*}
  Realizing that we can collapse several stalks of the sheaf without disrupting the space of global sections, we obtain a somewhat less redundant formulation:
  \begin{equation*}
    \xymatrix{
      C^\infty(M,\mathbb{R})&\\
      &C^\infty(M,\mathbb{R}) \ar[ul]^{\Delta}\\
      S\subseteq C^\infty(M,\mathbb{R}) \times C^\infty(M,\mathbb{R}) \ar[uu]^{\pr_1} \ar[ur]^{\pr_2}
      }
  \end{equation*}

  Following the construction of the pullback in Section \ref{sec:discretization}, let us choose an open cover $\col{U}$ of $M$ and pull back each $C^\infty(M,\mathbb{R})$ to this cover.  Assuming that $n$ is the dimension of $M$, and that $x \in U \in \col{U}$, this results in a diagram like the following
  \begin{equation*}
\includegraphics[width=4.5in]{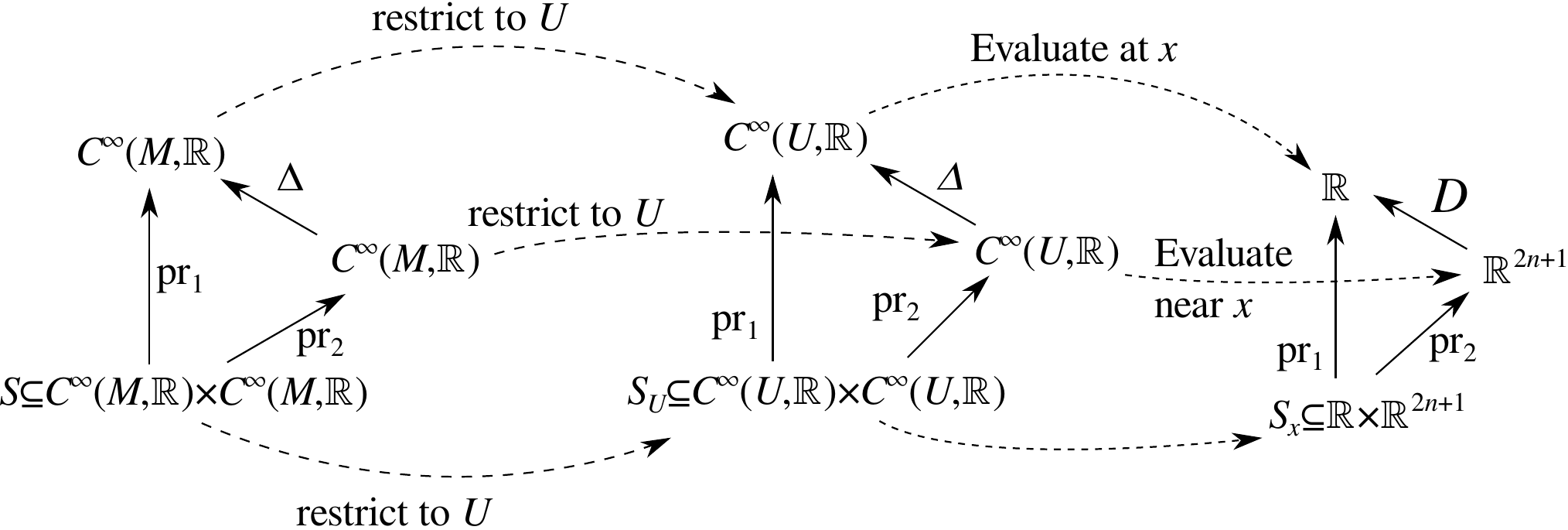}
  \end{equation*}
for each $U \in \col{U}$.  In the diagram, the fully continuous solutions are global sections of the sheaf on the left, and fully discretized solutions appear on the right.  The arrow labeled ``Evaluate near $x$'' takes a smooth function $f:U \subseteq \mathbb{R}^n \to \mathbb{R}$ to the following vector
\begin{equation*}
\left(f(x), f(x+e_1), f(x-e_1), \dotsc, f(x+e_n),f(x-e_n) \right),
\end{equation*}
where $e_i$ is the $i$-th coordinate vector.  Given this information, the arrow labeled $D$ computes the following discrete approximation to the Laplacian
\begin{equation*}
D(a_0,a_1,\dotsc,a_{2n}) = \frac{1}{2n}\left(\sum_{i=1}^{2n} a_i\right) - a_0.
\end{equation*}
Finally, the set $S_x$ of local solutions to the discretized problem near $x$ is given by
\begin{equation*}
S_x = \{(u,v_0,v_1,\dotsc,v_{2n}) \in \mathbb{R}^{2n+2} : u + k^2 v_0 = 0\}.
\end{equation*}
Notice that even though only $v_0$ and $u$ appear in the specification of $S_x$, the other $v_i$ are constrained in the global sections of the sheaf of discretized solutions.
\end{example}

\begin{example}
Consider the case of a nonlinear heat equation with a heat source, specified by
\begin{equation}
\label{eq:nonlinear_heat}
\frac{\partial}{\partial t} u(x,t) - \Delta u(x,t) + K u^2(x,t) = f(x,t).
\end{equation}
where $x\in M$ is a point in an $n$-dimensional manifold, $t\in \mathbb{R}$.  We have several options for treating the nonlinearity -- either it can be encapsulated into the solution space, or it can be broken out as another variable.  Breaking it out as another variable has the advantage that the nonlinearity is then encoded as a restriction map in the sheaf, which makes later analysis a little easier.  Therefore, we can rewrite \eqref{eq:nonlinear_heat} as the following explicit system
\begin{eqnarray*}
f(x,t) & = & T(x,t) - L(x,t) + K V(x,t), \\
V(x,t) &=& u^2(x,t), \\ 
T(x,t) & = & \frac{\partial}{\partial t} u(x,t), \\
L(x,t) & = & \Delta u(x,t). \\
\end{eqnarray*}
The diagram for the resulting explicit solution sheaf $\shf{H}$ is
\begin{equation*}
\includegraphics[width=3.5in]{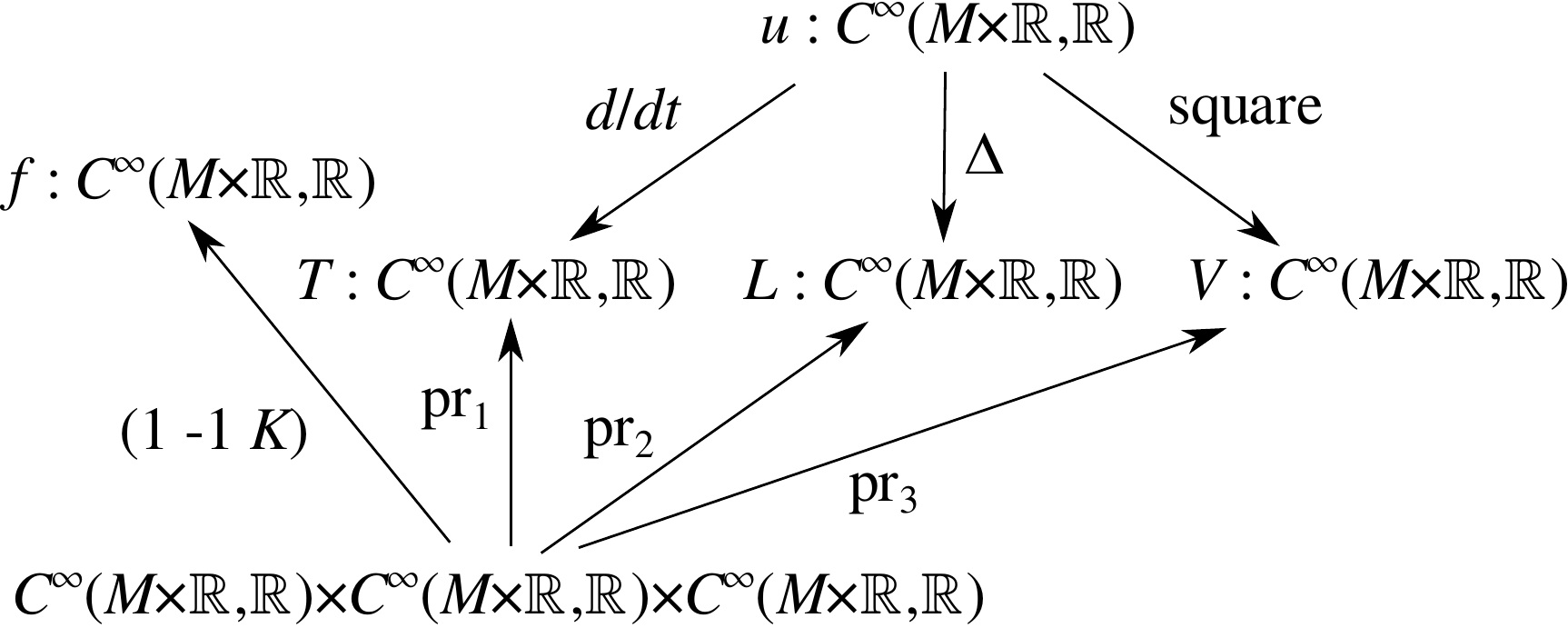}
\end{equation*}
in which the arrow labeled ``square'' represents the function taking $u$ to $u^2$.

On each open $U \subseteq M$ and $(a,b) \subseteq \mathbb{R}$, we can follow Section \ref{sec:discretization} to construct a hybrid morphism that takes
\begin{equation*}
  \dshf{C}_c(M\times \mathbb{R},\mathbb{R})(U\times (a,b)) = C_c(U\times (a,b), \mathbb{R}) \to C(U\times (a,b), \mathbb{R})
\end{equation*}
in each stalk in the diagram above.  Since the nonlinear squaring map is indeed a function $C_c(U\times (a,b), \mathbb{R}) \to C_c(U\times (a,b), \mathbb{R})$, we can indeed construct the hybrid morphism. 

Although the approach of discretizing using compactly supported smooth functions is attractive, there is a distinct problem: the solutions to \eqref{eq:nonlinear_heat} are typically not localized in space $x$ or in time $t$.  This means that trying to approximate solutions using localized functions is bound to cause numerical issues.  A safer approach is instead to use a dual sheaf $\dshf{S}_k$ of degree $k$ splines: although they have local control, they can be extended.  Specifically, let $\col{U}$ be a cover of $M$ consisting of open sets with compact closures, each of which is homeomorphic to an open set in $\mathbb{R}^n$.  Construct a partial order\footnote{This partial order is the 1-skeleton of the nerve of $\col{U}$.} formed by elements of $\col{U}$ and their pairwise intersections, with $U \le (U \cap V)$ and the dual sheaf $\dshf{S}_k$ by
\begin{enumerate}
\item Each stalk $\dshf{S}_k(U)$ is the vector space of degree $k$ polynomials in $n$ variables for $U \in \col{U}$,
\item Each stalk $\dshf{S}_k(U \cap V)$ is the vector space of degree $k-1$ polynomials in $n$ variables for $U,V \in \col{U}$,
\item Each extension $\dshf{S}_k( (U \cap V) \to U)$ is the composition of the transition map $(U \cap V) \to U$ in $M$ with the projection from degree $k$ polynomials to degree $k-1$ polynomials.  
\end{enumerate}
The benefit with this construction is that each stalk (of both types, $U$ and $U \cap V$) is directly mapped to a space of continuous functions, so there is still a hybrid morphism $\dshf{S}_k \to \shf{H}$, but compact support is not required.
\end{example}

\begin{example}
  The sheaf $\shf{H}$ in the previous example has a nonlinear restriction map.  It is quite evident how to linearize this restriction map about a local section (see Definition \ref{def:linearized}) -- simply replace the square with the $2$ times the value of the section.
\end{example}

\section{Multi-model systems of differential equations}
\label{sec:multidiffeqn}

In the previous sections, we considered differential equations in the usual sense -- on manifolds.  What happens on stratified manifolds?  If the model is originally formulated on the entire space, then it descends to models on each stratification.  Those models are not independent, but have relationships among them.  Conversely -- and more usefully -- if one starts with models on each stratification and defines various boundary conditions, then a global model can be assembled.  In order to translate models on different portions of a space, we transform sheaves along order-preserving functions using pullbacks (Definition \ref{df:pullback}).

Let us now consider the case where we have a known model on a topological space $X$ encoded as a sheaf $\shf{S}$.  Let $\{X_i\}$ be a finite collection of closed subspaces of $X$ whose union is $X$.   Consider the intersection lattice of $\{X_i\}$: the poset $P$ whose elements are all possible intersections and unions of $X_i$, and the partial order is the subset relation.  (See Figure \ref{fig:localizing}.)  Given the sheaf $\shf{S}$, we can pull back to a sheaf $\shf{S}_i$ on $X_i$ along the inclusion $X_i \to X$.

Given this formal construction, pulling back $\shf{S}$ to each element of $P$ yields a dual sheaf $\dshf{S}$ \emph{of sheaves} on $P$:
\begin{enumerate}
\item For each element $A$ of $P$, $\dshf{S}(A)=i_A^* \shf{S}$, which is the pullback of $\shf{S}$ to $A$ along the inclusion $i_A: A \to X$
\item Since pulling back is a contravariant functor by Lemma \ref{lem:pushpull}, each pair of elements $A,B \in P$ with $A \le B$ has a sheaf morphism induced $\dshf{S}(B) \to \dshf{S}(A)$.  This defines the extension map $\dshf{S}(A \le B)$.
\end{enumerate}

\begin{figure}
\begin{center}
\includegraphics[width=4.5in]{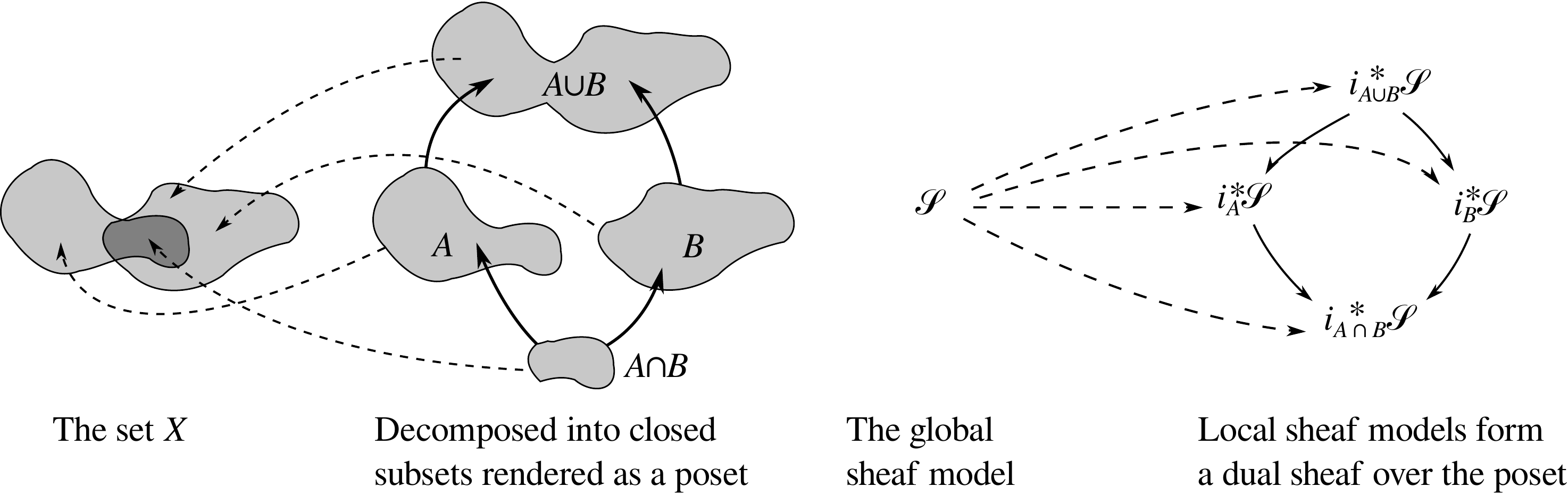}
\caption{Localizing a sheaf model on a space $X$ to subspaces}
\label{fig:localizing}
\end{center}
\end{figure}

The interpretation of $\dshf{S}$ as a dual sheaf supports the intuition that $\shf{S}_i$ represents the solutions to a given model on the interior of a stratum $X_i$, and the solutions are \emph{extended} to the boundary where they may interact with other strata.

From an analysis perspective, we generally want the solutions, proper.  Since the extensions in this dual sheaf $\dshf{S}$ are sheaf morphisms, they induce maps on global sections of each of the pullback sheaves.  Computing global sections on each stratum yields a dual sheaf of sets $\dshf{S}'$ on $P$ given by
\begin{enumerate}
\item For each $A$ in $P$, $\dshf{S}'(A)$ is the set of global sections of the sheaf $\dshf{S}(A)$,
\item For each pair of elements $A \le B$ in $P$, $\dshf{S}'(A \le B)$ is the map induced by Proposition \ref{prop:morphism} on global sections by the extension map $\dshf{S}(A \le B)$, which is itself a dual sheaf morphism.
\end{enumerate}

\begin{proposition}
  The global sections of the dual sheaf $\dshf{S}'$ are precisely the global sections of the original sheaf $\shf{S}$.
\end{proposition}
\begin{proof}
  First, every global section of $\shf{S}$ is taken to a global section in each stalk of $\dshf{S}$, which is an element of $\dshf{S}'$ since the extensions are sheaf morphisms, by Proposition \ref{prop:morphism}.  On the other hand, each global section of $\dshf{S}'$ takes a value at $X$, which is by definition a global section of $\shf{S}$.
\end{proof}

Since $\dshf{S}$ is a dual sheaf \emph{of sheaves}, its space of sections ought to itself be a sheaf!  As the next Proposition indicates, that sheaf is $\shf{S}$.  (From the perspective of category theory, we are merely verifying that pulling back $\shf{S}$ to $P$ produces a diagram with $\shf{S}$ as its limit.)

\begin{proposition}
\label{prop:universality}
For each $A \in P$, there is a sheaf morphism $m_A: \shf{S} \to \dshf{S}(A)$ so that $m_A$ \emph{commutes with the extensions} of $\dshf{S}$, namely $m_A = \dshf{S}(A \le B) \circ m_B$ for every pair of elements $A \le B$ in $P$.  

For any other sheaf $\shf{T}$ that also has morphisms $n_A: \shf{T} \to \dshf{S}(A)$ that commute with the extensions of $\dshf{S}$, there is a sheaf morphism $t: \shf{T} \to \shf{S}$ such that $n_A = m_A \circ t$ for all $A$ in $P$.
\end{proposition}
\begin{proof}
First observe that since the stalk $\dshf{S}(A)$ on $A$ is the pullback $i_A^* \shf{S}$ of $\shf{S}$ along the inclusion map $i_A: A \to X$, we should define $m_A = i_A^* : \shf{S} \to i_A^* \shf{S}$ as given by Definition \ref{df:pullback}.  For the extension maps, suppose that $A \le B \le X$ in $P$.  This can be thought of as a sequence of inclusions $A \to B \to X$, each of which is an order-preserving map.  So by Lemma \ref{lem:pushpull}, these induce sheaf morphisms going the opposite direction, which by the definition of the extension maps of $\dshf{S}$ is precisely $m_A = \dshf{S}(A \le B) \circ m_B$.

Now suppose that $\shf{T}$ is any other sheaf on a poset $Y$ with morphisms $n_A: \shf{T} \to \dshf{S}(A)$ along order preserving maps $g_A: A \to Y$ commuting with the extensions of $\dshf{S}$.  Suppose that $x \in A$, so that there is a map $n_{A,x} : \shf{T}(g_A(x)) \to \dshf{S}(A)(x)$.  We must perform two constructions: we must construct an order preserving map $f: X \to Y$ and the morphism $t: \shf{T} \to \shf{S}$ along $f$.

\begin{description}
\item[Constructing $f$:] Suppose $x\in X$, which is in at least one element of $P$, say $A$.  Observe that because the $n_A$ commute with the extensions of $\dshf{S}$, it must be the case that the $g_A$ maps commute with the inclusions.  Therefore, we can define $f(x) = g_A(x)$, because if $x$ is also in $B$, $g_{A\cap B} = g_A \circ i_{A\cap B\to A}$ where $i_{A\cap B \to A}: A\cap B \to A$ is the inclusion.
\item[Constructing $t$:]  Suppose that $x \in A$, so that the component of the morphism $n_A$ is the map $n_{A,x} : \shf{T}(g_A(x)) \to \dshf{S}(A)(x)$.  However, $\dshf{S}(A)(x) = i_A^* \shf{S}(x) = \shf{S}(x)$ where $i_A : A \to X$ is the inclusion.  We merely need to note that $f(x) = g_A(x)$ to complete the construction.
\end{description}

\end{proof}

These propositions indicate that disassembling the model encoded in $\shf{S}$ into a dual sheaf built on the intersection lattice of some subsets is a faithful representation of the model.  It points the way for the reverse construction, when one doesn't have a sheaf $\shf{S}$ on $X$ to start.  To formulate a collection of interrelated models, one need only build such a dual sheaf of sheaves $\dshf{S}$ (on $P \backslash X$) from the outset and then examine $\dshf{S}'$ to find its solution.

\begin{theorem}
\label{thm:reconstruct}
  Given a dual sheaf $\dshf{S}$ of sheaves of sets or vector spaces on the poset $P$, one can construct a sheaf $\shf{S}$ and a set of sheaf morphisms $m_A: \shf{S} \to \dshf{S}(A)$ for each $A \in P$ such that
  \begin{enumerate}
  \item $m_A = \dshf{S}(X_A \le X_B) \circ m_B$ for each $A \le B \in P$ and
  \item If $\shf{R}$ is any other sheaf with this property, then there is a sheaf morphism $r : \shf{R} \to \shf{S}$ that commutes with all the $m_A$ and extensions of $\dshf{S}$.
  \end{enumerate}
\end{theorem}

\begin{proof}
Because the category of sets (or the category of vector spaces) is complete, then the category of sheaves of sets (or vector spaces) is complete \cite{Gray_1962}.  The sheaf $\shf{S}$ is precisely the category theoretic \emph{limit} of the diagram of sheaves given by $\dshf{S}$, which exists by completeness.   What follows is an explicit construction following \cite{Gray_1962}.

We first need to construct the poset for $\shf{S}$.  This is easily done: let $X$ be the disjoint union of all of the posets for each sheaf $\dshf{S}(A)$ (where $A\in P$) under the equivalence relation in which elements are matched by the order preserving maps for each extension $\dshf{S}(A \le B)$.  For each $x\in X$, observe that we can construct a new dual sheaf $\dshf{S}_x$ from $\dshf{S}$ by
\begin{enumerate}
\item $\dshf{S}_x(A) = \left(\dshf{S}(A)\right)(x)$ (a set!) for each $A \in P$, and
\item $\dshf{S}_x(A \le B) = i_{(A \to B),x}^*$, which is the $x$-component of the dual sheaf morphism $i_{A\to B}^* : \dshf{S}(B) \to \dshf{S}(A)$ induced by the inclusion $A \to B$.
\end{enumerate}
Using this dual sheaf, we define $\shf{S}(x)$ to be the set of global sections of $\dshf{S}_x$.  We define the restriction 
$\shf{S}(x \le y)$ to follow the \emph{restrictions} of the \emph{sheaves} $\left(\dshf{S}(A)\right)(x \le y) : \left(\dshf{S}(A)\right)(x) \to \left(\dshf{S}(A)\right)(y)$ at each element $A$ of $P$.

Each morphism $m_A$ projects out the elements of $\shf{S}(x)$ (already a direct product!) to $\dshf{S}_x(A) = \left(\dshf{S}(A)\right)(x)$.  Given this construction of $\shf{S}$ and morphisms $m_\bullet$, the argument for any other sheaf $\shf{R}$ in Proposition \ref{prop:universality} goes through unchanged.
\end{proof}

\begin{figure}
\begin{center}
\includegraphics[width=3in]{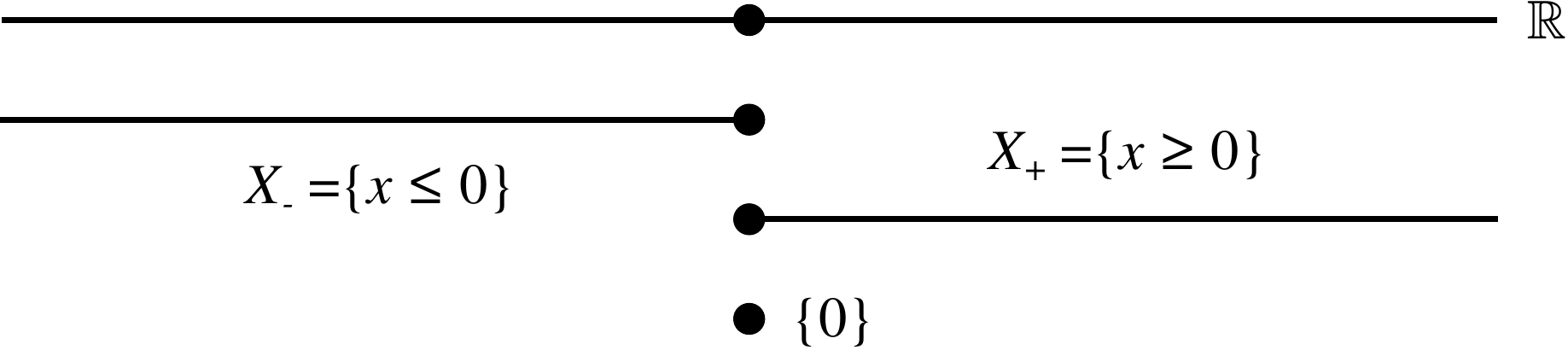}
\caption{Regions for modeling waves scattering along a segmented string}
\label{fig:string_scattering}
\end{center}
\end{figure}

\begin{example}
  Consider the case of waves along a string made of two segments with different phase speeds.  As shown in Figure \ref{fig:string_scattering}, suppose that the string is along the real line, and that the segment $X_-=\{x \le 0\}$ has phase speed $c_-$, the segment $X_+=\{x \ge 0\}$ has phase speed $c_+$, and the ``knot'' between the segments at $x=0$ ensures $C^1$ continuity.  For simplicity, we assume linear wave models on each segment and focus on the single frequency case where all solutions have an $e^{i\omega t}$ dependence.  In this case, the equations are
  \begin{equation*}
    \begin{cases}
      -\omega^2 u_- - c_-^2u''_- = 0&\text{for }x \le 0\\
      -\omega^2 u_+ - c_+^2 u''_+ = 0&\text{for }x \ge 0\\
      u_-(0)=u_+(0)&\\
      u'_-(0)=u'_+(0)&
      \end{cases}
  \end{equation*}
  where we have used $u_{\pm}$ to emphasize that from the outset $u(x)$ for $x>0$ and $x<0$ are unrelated.  The last two equations implement a particular boundary condition at $x=0$.

  We can encode each of these as solution sheaves, individually over $X_-$, $X_+$, and $\{0\}$ as shown by the solid arrows in Figure \ref{fig:string_sheaves}, which is built on the poset $X=\{\{0\},X_-,X_+\}$ with order coming from inclusion.  Within $X_-$, the subspace $S_-$ consists of the space of functions spanned by
  \begin{equation*}
    \{e^{ik_-x}, e^{-ik_-x}\},
  \end{equation*}
  namely a subspace isomorphic to $\mathbb{C}^2$, where $k_-=\omega/c_-$.  Similarly, the subspace $S_+$ is isomorphic to $\mathbb{C}^2$.

  The enforcement of the boundary conditions amounts to constructing extension maps for a dual sheaf $\dshf{S}$ of sheaves shown as the dashed arrows in Figure \ref{fig:string_sheaves}.  Observe that the diagram formed by dashed and solid arrows is commutative.  Without belaboring the point, the dashed arrows in the Figure are evaluations of functions -- the unlabeled functions in the Figure mapping to $\mathbb{C}^2$ compute the value of a function and its derivative at $0$.  Therefore, computing stalk-wise global sections of each sheaf yields a dual sheaf of vector spaces given by the diagram
  \begin{equation*}
\xymatrix{ \mathbb{C}^2 \ar[rr]^-{\begin{pmatrix}1&1\\ik_-&-ik_-\end{pmatrix}} && \mathbb{C}^2 && \mathbb{C}^2 \ar[ll]_-{\begin{pmatrix}1&1\\ik_+&-ik_+\end{pmatrix}} }
\end{equation*}
The space of global sections of this dual sheaf is isomorphic to $\mathbb{C}^2$, because each of the matrices are of full rank if we assume that $k_-$ and $k_+$ are both nonzero.

Now, to derive the sheaf $\shf{S}$ from the dual sheaf $\dshf{S}$ according to Theorem \ref{thm:reconstruct}, we must construct dual sheaves $\dshf{S}_x$ associated to each element of the poset $x\in X$, namely
\begin{eqnarray*}
\dshf{S}_{\{0\}} &=& \xymatrix{ \mathbb{C}^2 \ar[rr]^-{\begin{pmatrix}1&1\\ik_-&-ik_-\end{pmatrix}} && \mathbb{C}^2 && \mathbb{C}^2 \ar[ll]_-{\begin{pmatrix}1&1\\ik_+&-ik_+\end{pmatrix}} },\\
\dshf{S}_{X_-} &=& C^\infty((-\infty,0],\mathbb{C}), \text{ and} \\
\dshf{S}_{X_+} &=& C^\infty([0,\infty),\mathbb{C}).
\end{eqnarray*}
The Theorem has us construct $\shf{S}$ stalkwise as the space of global sections of each of these dual sheaves, and the restrictions of $\shf{S}$ are those maps induced on global sections by the restriction maps in each stalk of the dual sheaf $\dshf{S}$.  Namely the diagram for $\shf{S}$ is
\begin{equation*}
\xymatrix{
C^\infty((-\infty,0],\mathbb{C})& &&& \mathbb{C}^2 \ar[llll]_-{\begin{pmatrix}\frac{-3k_--k_+}{2k_-}e^{ik_-x}&\frac{-3k_-+k_+}{2k_-}e^{ik_-x}\\\frac{k_++k_-}{2k_-}e^{-ik_-x}&\frac{k_--k_+}{2k_-}e^{-ik_-x}\\\end{pmatrix}} \ar[rr]^-{\begin{pmatrix}e^{ik_+x}&0\\0&e^{-ik_+x}\end{pmatrix}} && C^\infty([0,\infty),\mathbb{C}).
}
\end{equation*}
\end{example}

\begin{figure}
\begin{center}
\includegraphics[width=3.5in]{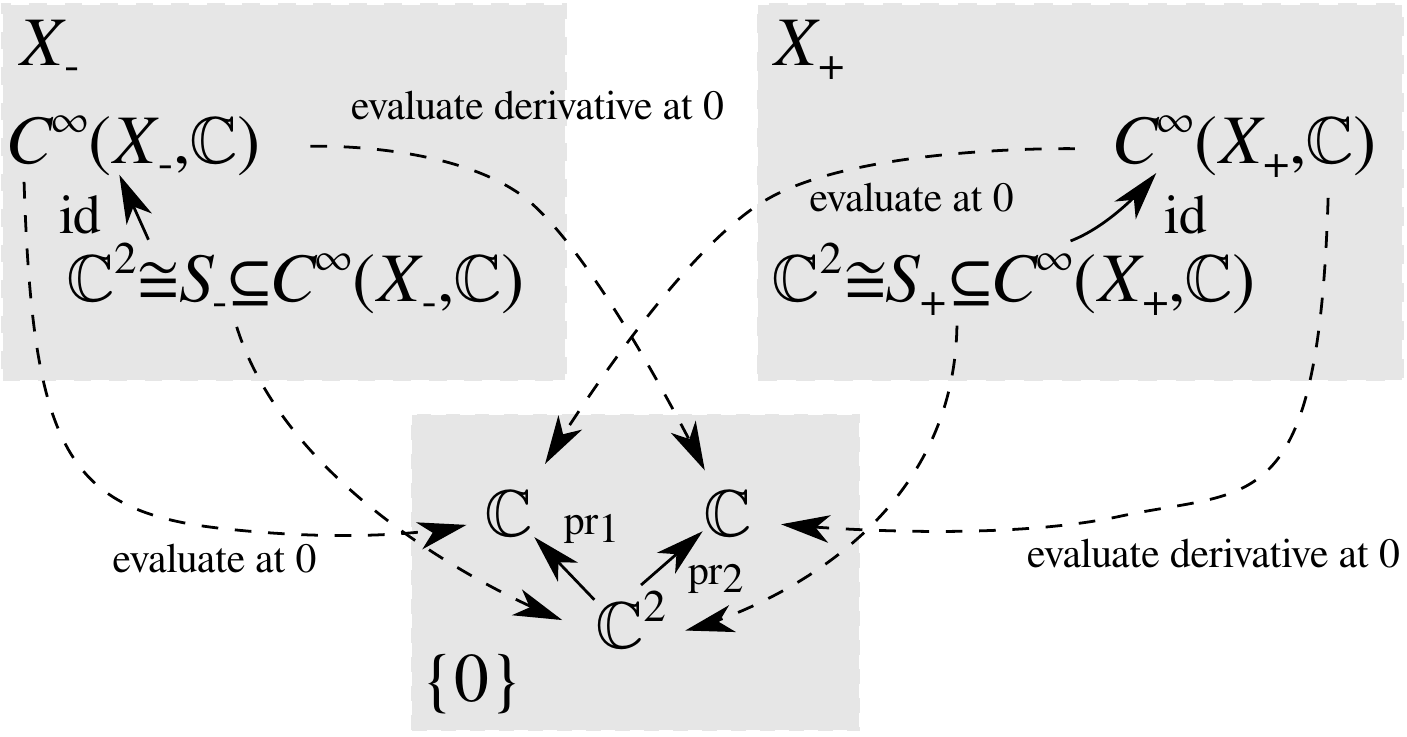}
\caption{Dual sheaf of sheaves describing the propagation of waves along a segmented string.  Solid lines are restriction maps of each sheaf along each segment, marked in the shaded regions.  Dashed lines are the extensions of the dual sheaf.}
\label{fig:string_sheaves}
\end{center}
\end{figure}

\begin{figure}
\begin{center}
\includegraphics[width=2.5in]{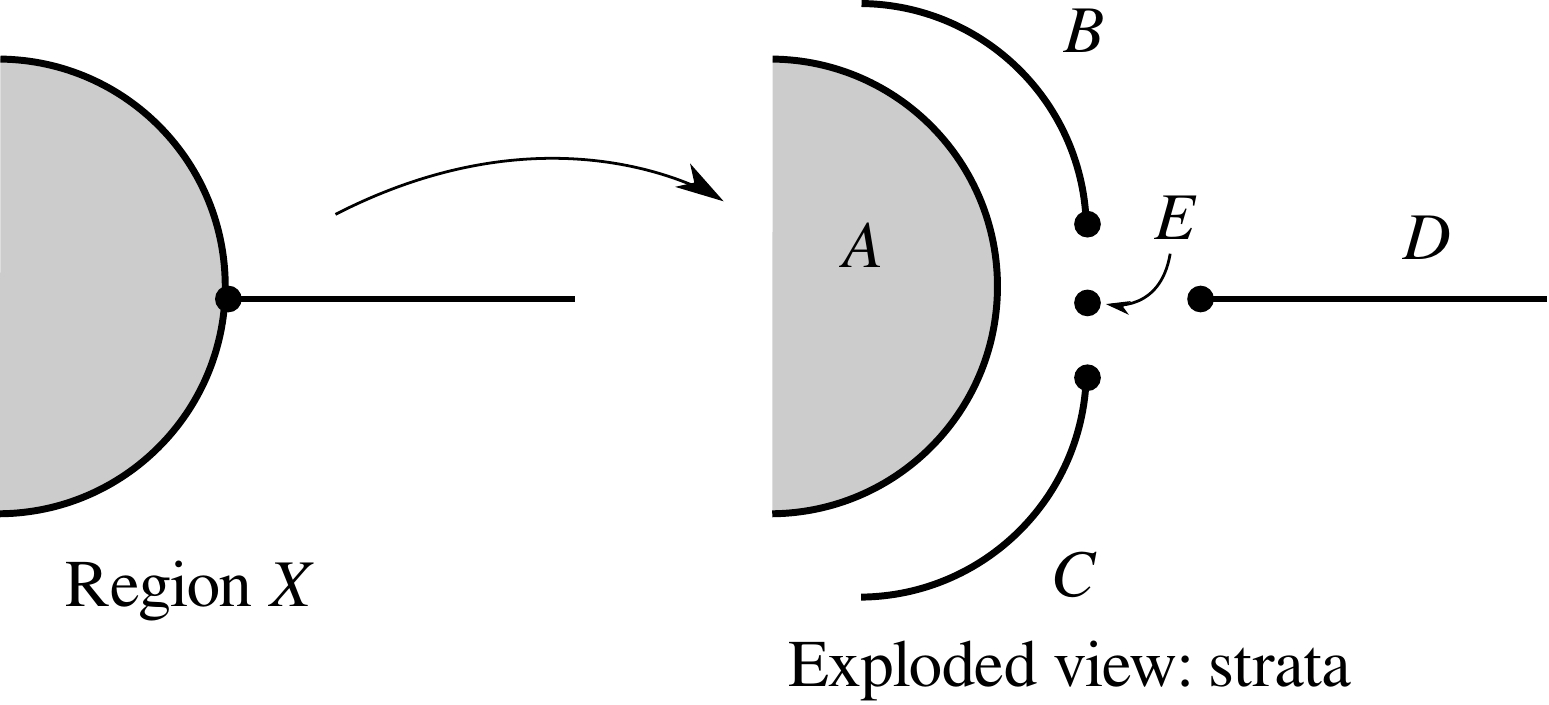}
\caption{Region for modeling the diffraction of waves exiting a channel into an open area (left) and its decomposition into strata (right)}
\label{fig:scattering_diffraction}
\end{center}
\end{figure}

\begin{example}
  If traveling waves along a narrow channel exit into a large open area, diffraction occurs.  Nearly the same formulation as in the previous example works, though the stratification is quite different as is shown in Figure \ref{fig:scattering_diffraction}.  Analogous to wave propagation along a string, the propagation along the narrow channel is split into two traveling waves:
  \begin{equation*}
    u(x) = a e^{ikx} + b e^{-ikx}.
  \end{equation*}
  However, on an open, 2-dimensional region, the solution is given by an integral
  \begin{equation*}
    u(x,y) = \int_0^{2\pi} c(\theta) e^{ik(x \cos \theta + y \sin \theta)} d\theta,
  \end{equation*}
  where $c$ is best thought of as a complex-valued measure on the unit circle.  If we write the space of complex valued measures on a manifold $N$ as $M(N,\mathbb{C})$, then the dual sheaf (of spaces of global sections) that models the propagation of waves on each stratum in Figure \ref{fig:scattering_diffraction} is given by the diagram
  \begin{equation*}
    \includegraphics[width=1.75in]{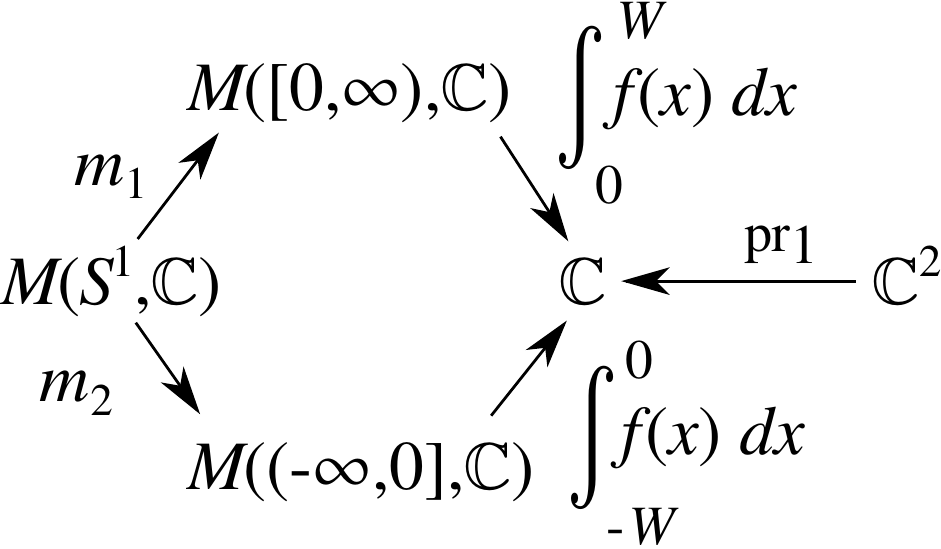}
  \end{equation*}
  in which the map $m_1$ takes the measure $c$ on the circle to the measure
  \begin{equation*}
    \left(m_1 c\right)(t) = \int_0^{2\pi}  c(\theta) e^{ik(p_x(t) \cos \theta + p_y(t) \sin \theta)} d\theta,
  \end{equation*}
  and
  \begin{equation*}
    \left(m_2 c\right)(t) = \int_0^{2\pi}  c(\theta) e^{ik(q_x(t) \cos \theta + q_y(t) \sin \theta)} d\theta,
  \end{equation*}
  in which the paths $(p_x(t),p_y(t))$ and $(q_x(t),q_y(t))$ trace out the coordinates of the upper and lower edges of the 2-dimensional region.

  We can study the Dirichlet problem for this dual sheaf by constraining the values taken by maps $m_1$ and $m_2$ to be zero except at the stratum $E$.  This is done by way of a dual sheaf morphism that annihilates the stalks where waves are allowed to propagate.  From this, we construct a new dual sheaf via stalk-wise quotients of the other two.  These operations are summarized in the diagram
  \begin{equation*}
    \includegraphics[width=4.5in]{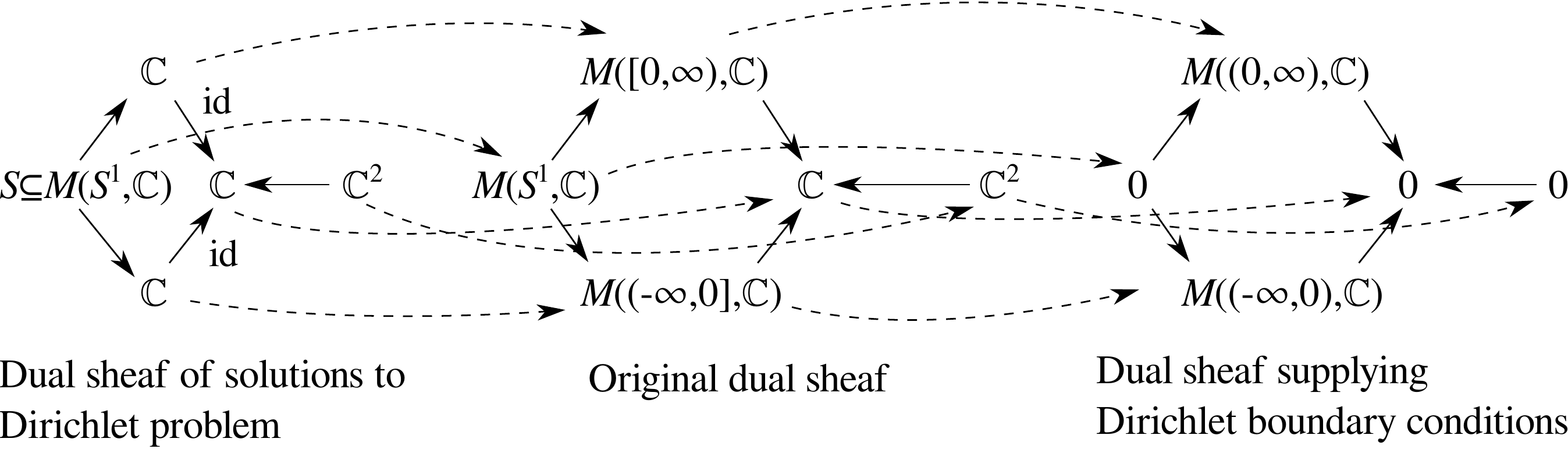}
  \end{equation*}
  Discretizing the leftmost dual sheaf in the above diagram by way of a dual sheaf morphism yields the usual boundary-elements formulation of this kind of problem, with the added benefit that wave propagation along the channel is automatically incorporated into the solution.
\end{example}

\section{Belief propagation networks and graphical models}
\label{sec:belief}

Consider a set of random variables $X_1, \dotsc, X_n$.  Given some knowledge about some of their various joint distributions, how much can be deduced about the others?  There has been considerable attention given to this kind of question in the literature, resulting in several popular algorithms for solving it under certain conditions.  In essence, they all amount to resolving the network into a factor graph.

We can consider the joint distribution over $\{X_1, \dotsc, X_n\}$ and all of its marginal distributions as a set of simultaneous equations according to Section \ref{sec:simultaneous}.  By adding conditional distributions using Bayes' rule, we can model a graphical model as a sheaf.  Belief propagation algorithms are then seen to be approximate methods for computing certain sections of this sheaf.

Assume that $(X_i,\col{M}_i)$ are measurable spaces for $i=1,\dotsc,n$, so that a \emph{random variable} $\Phi_{X_i}$ consists of a signed measure\footnote{We take signed measures rather than probability measures for algebraic convenience.  Throughout, if we start with probability measures, they remain so.  Thus nothing is lost by this perspective.} on $(X_i,\col{M}_i)$.  The space of such signed measures $M(X_i)$ is a vector space in the obvious way, with the sum of two measures on a measurable set being the sum of their respective measures of that set.  A joint distribution on a subset of the random variables, say $X_{i_1}, \dotsc, X_{i_k}$, specifies a probability distribution $\Phi_{X_{i_1} \dotsb X_{i_k}}$ on the measurable space $X_{i_1}\times \dotsb \times X_{i_k}$.

\begin{lemma}
As vector spaces, $M(X_{i_1}\times \dotsb \times X_{i_k}) \cong M(X_{i_1})\otimes \dotsb \otimes M(X_{i_k})$.
\end{lemma}
The proof of this lemma follows directly from the definition of the tensor product.

\begin{corollary}
  \label{cor:multiple_marg}
  The set projection $\pr_j: X_{i_1}\times \dotsb \times X_{i_k} \to X_{i_1}\times \dotsb \widehat{X_{i_j}} \dotsb \times X_{i_k}$ for any $j$ lifts to a linear \emph{marginalization} map $m_j: M(X_{i_1}\times \dotsb \times X_{i_k}) \to M(X_{i_1}\times \dotsb \widehat{X_{i_j}} \dotsb \times X_{i_k})$, where we use the hat to indicate an omitted variable.  We sometimes speak of $m_j$ ``marginalizing out $X_{i_j}$'' from the joint distribution.  Further, marginalizing a pair of random variables out from a joint distribution is independent of their order.
\end{corollary}

\begin{example}
  \label{eg:three_rvs}
  Suppose that $X_1=X_2=X_3=\{0,1\}$, so that the space of signed measures over each is 2-dimensional.  By the Lemma, the space of measures over the product $X_1 \times X_2 \times X_3$ is 8-dimensional.  The projection $\pr_1: X_1 \times X_2 \times X_3 \to X_2 \times X_3$ lifts to the marginalization $m_1$ given by the matrix
  \begin{equation*}
    m_1 = \begin{pmatrix}
      1& 0& 0& 0& 1& 0& 0& 0\\
      0& 1& 0& 0& 0& 1& 0& 0\\
      0& 0& 1& 0& 0& 0& 1& 0\\
      0& 0& 0& 1& 0& 0& 0& 1\\
      \end{pmatrix},
  \end{equation*}
  in which the basis elements are written in lexicographical order.  Similarly, the marginalization $m_2$ is given by
    \begin{equation*}
    m_2 = \begin{pmatrix}
      1& 0& 1& 0& 0& 0& 0& 0\\
      0& 1& 0& 1& 0& 0& 0& 0\\
      0& 0& 0& 0& 1& 0& 1& 0\\
      0& 0& 0& 0& 0& 1& 0& 1\\
      \end{pmatrix}.
  \end{equation*}
    If we marginalize twice, $X_1 \times X_2 \times X_3 \to X_2 \times X_3 \to X_3$ or $X_1 \times X_2 \times X_3 \to X_1 \times X_3 \to X_3$, then we obtain the same map, namely
    \begin{equation*}
      \begin{pmatrix}
        1& 0& 1& 0& 1& 0& 1& 0\\
        0& 1& 0& 1& 0& 1& 0& 1\\
      \end{pmatrix}.
    \end{equation*}
\end{example}

We can use the marginalization maps to describe a set of random variables as a system of (linear) equations.  Specifically, let the set $V$ of variables be the power set of $\{X_1, \dotsc, X_n\}$. For each variable $v=X_I$, let $W_v = M(X_I)$ where $I=\{i_1,\dotsc,i_k\}$.  The set of equations $E$ consists of all possible marginalizations, namely equations of the form\footnote{We need not consider marginalizing multiple variables out because of Corollary \ref{cor:multiple_marg}.}
\begin{equation}
  \label{eq:marginal}
  \Phi_{X_{i_1} \dotsb \widehat{X_{i_j}} \dotsb X_{i_k}} = m_j \Phi_{X_{i_1} \dotsb X_{i_k}}.
\end{equation}
Notice that this system is explicit according to Definition \ref{def:explicit} and has a dependency graph in which all arrows point from joint distributions over a set of variables to subsets of those variables.  Thus, it is straightforward to construct a sheaf model $\shf{J}$ of this system using the techniques of Section \ref{sec:simultaneous}.  Specifically, the poset in question is $V \sqcup E$, and for each variable $\shf{J}(X_I) = M(X_I)$.  Each equation $e$ involves exactly two variables $\{X_{i_1}, \dotsc, X_{i_k}\}$ and $\{X_{i_1}, \dotsc, \widehat{X_{i_j}} ,\dotsc, X_{i_k}\}$, so that
\begin{enumerate}
\item $\shf{J}(e) = M(X_{i_1} \times \dotsb \times X_{i_k})$,
\item $\shf{J}(e \le \{X_{i_1}, \dotsc, X_{i_k}\})$ is the identity map, and
\item the other restriction $\shf{J}(e \le \{X_{i_1}, \dotsc, \widehat{X_{i_j}}, \dotsc, X_{i_k}\})$is the marginalization function $m_j$.
\end{enumerate}

\begin{example}
  \label{eg:big_marginal}
  Continuing Example \ref{eg:three_rvs}, the sheaf $\shf{J}$ associated to the system of random variables is given by the diagram
  \begin{center}
    \includegraphics[width=2.5in]{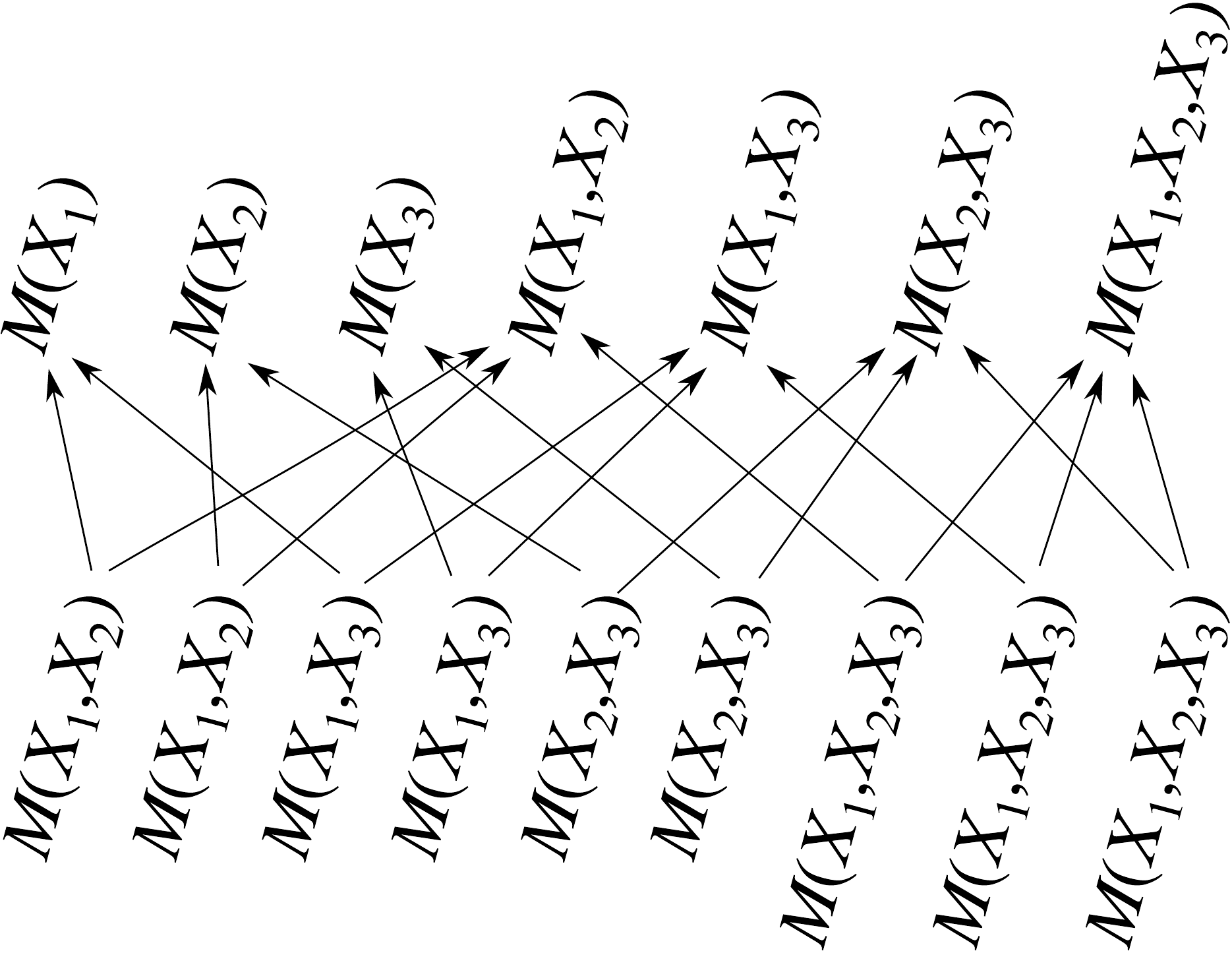}
  \end{center}
  The arrows are labeled either with identity maps or marginalizations as appropriate.
\end{example}

From the example, it is clear that the sheaf $\shf{J}$ contains a number of duplicate stalks with identity maps between them.  Since the dependency graph for the joint distributions is a directed acyclic graph, there is an equivalent sheaf $\shf{J}'$ over a smaller poset.  Consider the partial order $\le$ on only the variables $V$, for which
\begin{equation}
\label{eq:containment_order}
  X_I \le X_J\text{ if }J \subseteq I.
\end{equation}
Then $\shf{J}'$ is given the same stalks as $\shf{J}$ over the variables, but we let $\shf{J}'(X_I \le X_J)$ be the composition of marginalization functions.

\begin{example}
  \label{eg:small_marginal}
  The sheaf diagram in Example \ref{eg:big_marginal} reduces considerably under this process, yielding a diagram consisting only of joint distributions and marginalization functions
  \begin{center}
    \includegraphics[width=2in]{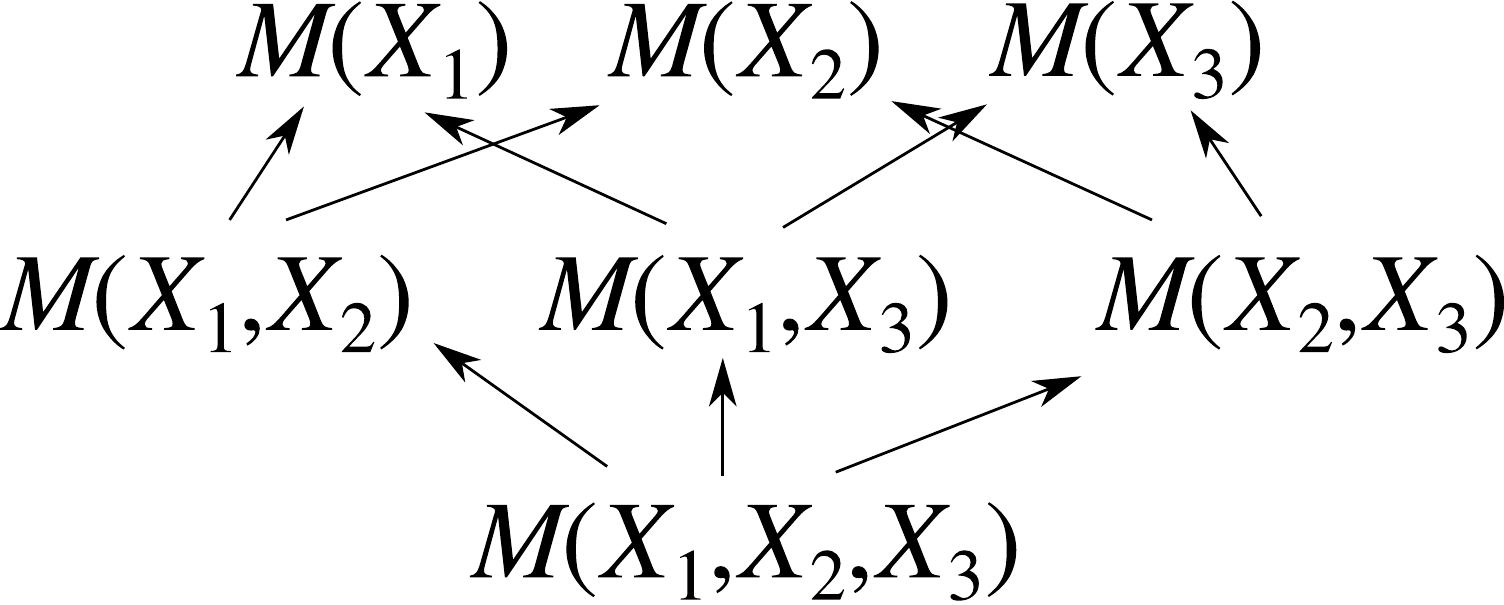}
  \end{center}
\end{example}

The constructions of $\shf{J}'$ and $\shf{J}$ are related by a pushforward along an order preserving function (compare Definition \ref{df:pushforward_dualsheaf}).

 \begin{definition}
\label{df:pushforward_sheaf}
  Suppose $f:X\to Y$ is an order preserving function between posets and that $\shf{R}$ is a sheaf on $X$.  The \emph{pushforward} $f_* \shf{R}$ is a sheaf on $Y$ in which
  \begin{enumerate}
    \item Each stalk $(f_* \shf{R})(c)$ is the space of sections over the set $f^{-1}(c) \subset X$, and
    \item The restriction maps $(f_* \shf{R})(a\le b)$ are given by restricting a section $s$ over $f^{-1}(a)$ to one over $f^{-1}(b)$.
  \end{enumerate}
  This construction yields a sheaf morphism\footnote{Warning! This morphism may not be injective.  A sufficient condition for injectivity is given by the Vietoris Mapping Theorem \cite[Thm. 3, Section II.11]{Bredon}, or \cite[Thm. 4.2]{TSPBook}.} $f_*\shf{R} \to \shf{R}$.
\end{definition}

\begin{proposition}
  \label{prop:marginal_iso}
  Consider the function $f: V \sqcup E \to V$ that
\begin{enumerate}
\item Takes sets of random variables (elements of $V$) to sets of random variables and 
\item Takes each equation $e \in E$ involving exactly two variables ($\{X_{i_1}, \dotsc, X_{i_k}\}$ and $\{X_{i_1}, \dotsc, \widehat{X_{i_j}} ,\dotsc, X_{i_k}\}$) to $\{X_{i_1}, \dotsc, X_{i_k}\}$.  (All equations in the marginalization sheaf are of this form!)
\end{enumerate}
If the domain $V \sqcup E$ is given the partial order that each equation is below the variables it involves, and the domain is given the partial order by set containment \eqref{eq:containment_order}, then $f$ is order-preserving.

Then $\shf{J}' = f_* \shf{J}$, and the morphism $\shf{J}' \to \shf{J}$ induces isomorphisms on the space of global sections of $\shf{J}$ and $\shf{J}'$.
\end{proposition}
\begin{proof}
To see that $f$ is order preserving, merely suppose that $e\in E$ and $v\in V$ with $e \le v$.  Without loss of generality, suppose that $e$ involves $\{X_{i_1}, \dotsc, X_{i_k}\}$ and $\{X_{i_1}, \dotsc, \widehat{X_{i_j}} ,\dotsc, X_{i_k}\}$.  Therefore, $v$ is either $\{X_{i_1}, \dotsc, X_{i_k}\}$ or $\{X_{i_1}, \dotsc, \widehat{X_{i_j}} ,\dotsc, X_{i_k}\}$.  If $v = \{X_{i_1}, \dotsc, X_{i_k}\}$ then there is nothing to prove since $f(e) = f(v)$. In the other case, 
\begin{eqnarray*}
f(e) &=& \{X_{i_1}, \dotsc, X_{i_k}\}\\
&\le&\{X_{i_1}, \dotsc, \widehat{X_{i_j}} ,\dotsc, X_{i_k}\} = f(v).
\end{eqnarray*}

To see that $\shf{J}' = f_* \shf{J}$, we examine the stalks and restriction maps according to Definition \ref{df:pushforward_sheaf}.  Each stalk of $\shf{J}'$ is the space of sections over its preimage through $f$ in $\shf{J}$.  For instance, let $v= \{X_{i_1}, \dotsc, X_{i_k}\}$.  Then its preimage consists of the set
\begin{equation*}
f^{-1}(v) = \left\{\{X_{i_1}, \dotsc, X_{i_k}\}, \{X_{i_2} ,\dotsc, X_{i_k}\}, \dotsc, \{X_{i_1}, \dotsc, \widehat{X_{i_j}} ,\dotsc, X_{i_k}\}, \dotsc \{X_{i_1}, \dotsc, \dotsc, X_{i_{k-1}}\}  \right\}.
\end{equation*}
The stalk in $\shf{J}$ over each element of $f^{-1}(v)$ is the same, and the restriction maps within the preimage are all identity maps. Therefore, the space of sections of $\shf{J}$ over $f^{-1}(v)$ is precisely the stalk over any element of $f^{-1}(v)$, which by construction is precisely the same as the stalk over $v$ in $\shf{J}'$.  The other restriction maps -- the marginalization maps -- in $\shf{J}$ are carried over unchanged into $\shf{J}'$.

Finally, the above argument makes it quite clear that the global sections of $\shf{J}$ and $\shf{J}'$ must be the same.
\end{proof}

We have thus far considered random variables and not graphical models.  A \emph{graphical model} on random variables $X_1, \dotsc, X_n$ consists of the set of all joint distributions and marginalization equations, but adds some equations of the form
\begin{equation}
  \label{eq:conditional}
  \Phi_{X_I} = L_e \Phi_{X_J}, \text{ where }I \le J,
\end{equation}
and where $L_e$ is a stochastic linear map (column sums are all 1).  The system is still explicit, but we can no longer form a partial order on the variables alone!  The sections of the resulting sheaf $\shf{B}$ are solutions to the graphical model.

\begin{figure}
  \begin{center}
    \includegraphics[width=2in]{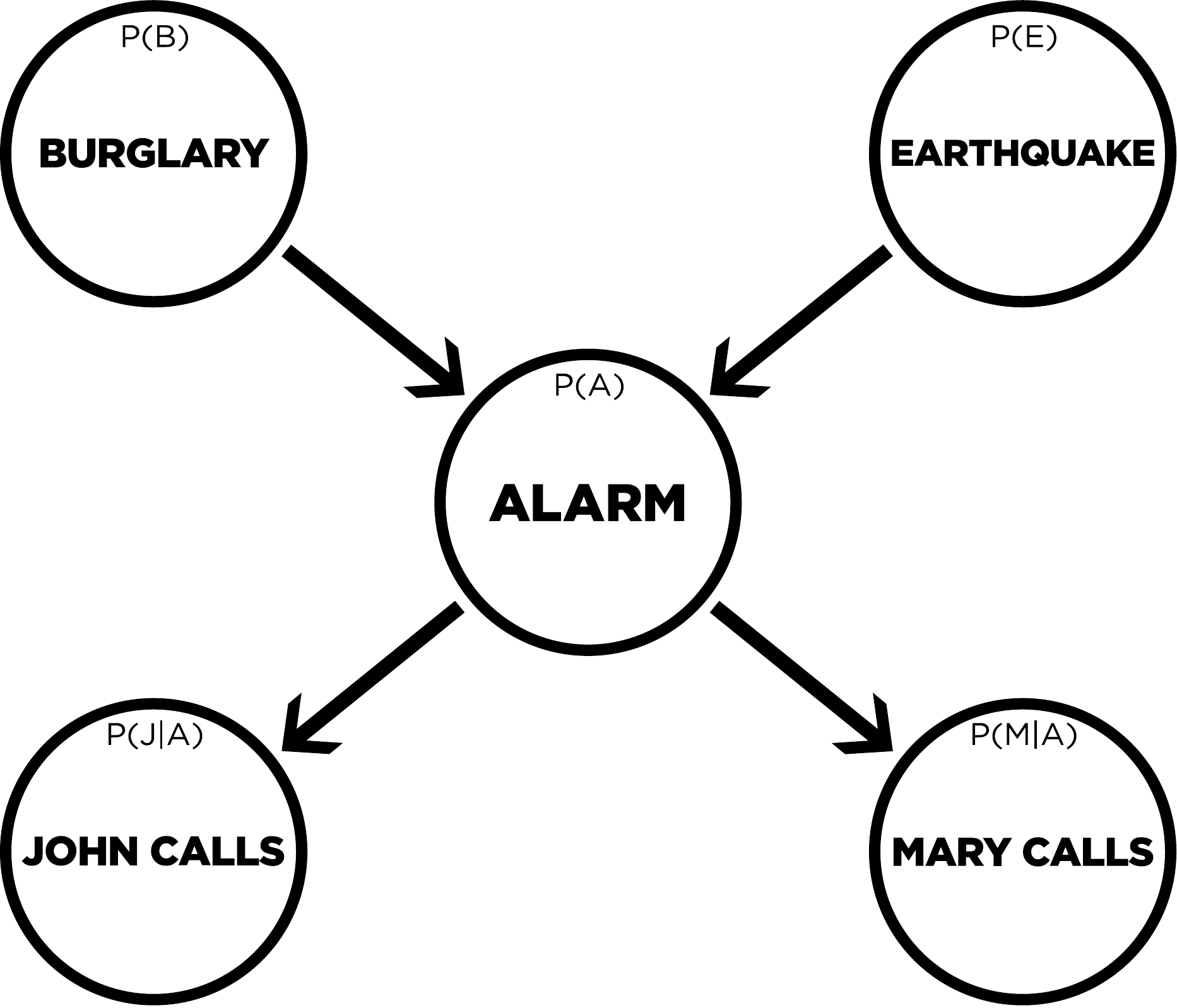}
    \caption{Two events can trigger an alarm.  If the alarm sounds, then with some probability John or Mary will go and investigate the cause of the alarm}
    \label{fig:alarm_causal}
  \end{center}
\end{figure}

\begin{figure}
  \begin{center}
    \includegraphics[width=3in]{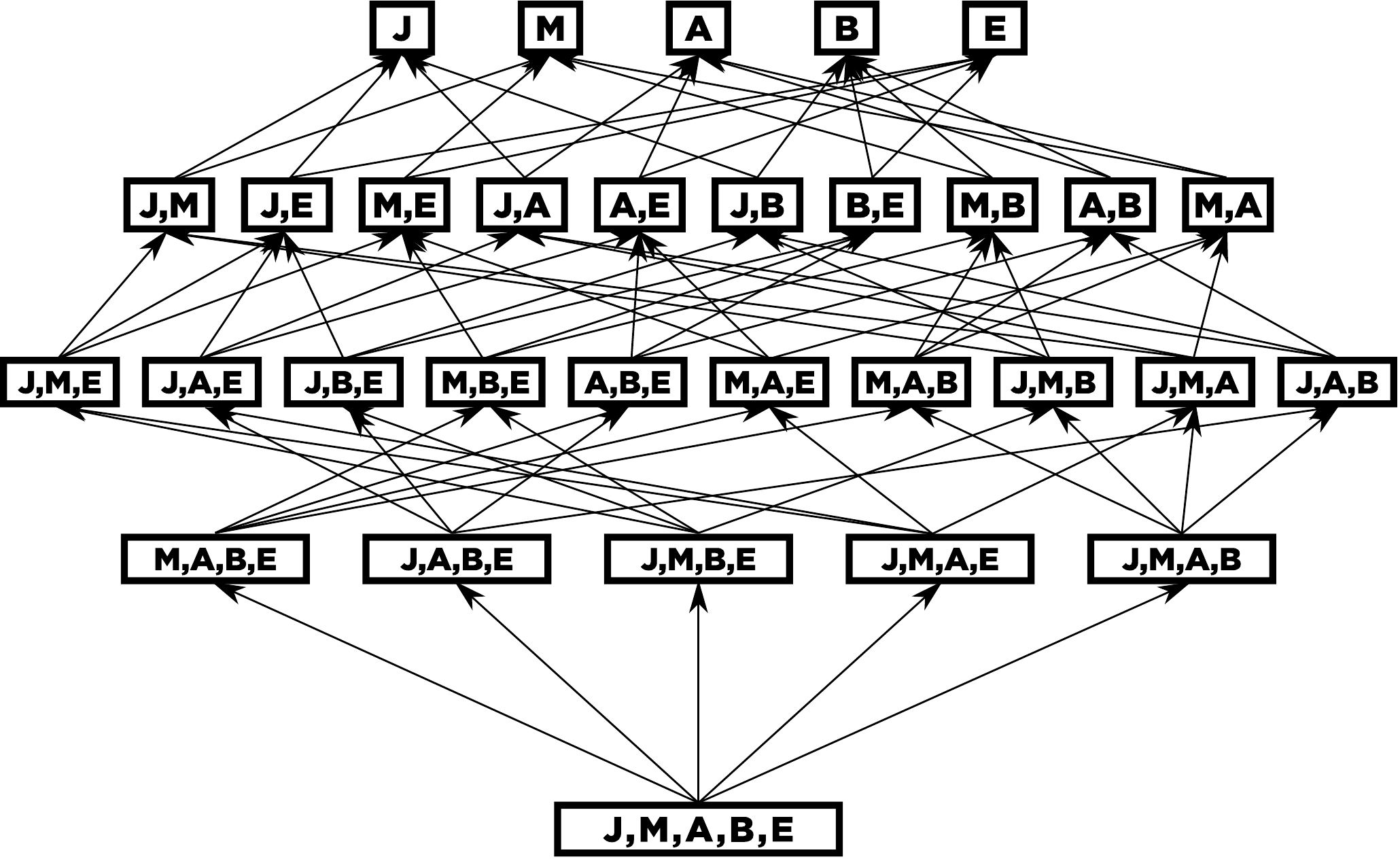}
    \caption{The marginalization sheaf for Example \ref{eg:alarm}}
    \label{fig:alarm_marginals}
  \end{center}
\end{figure}

\begin{figure}
  \begin{center}
    \includegraphics[width=3in]{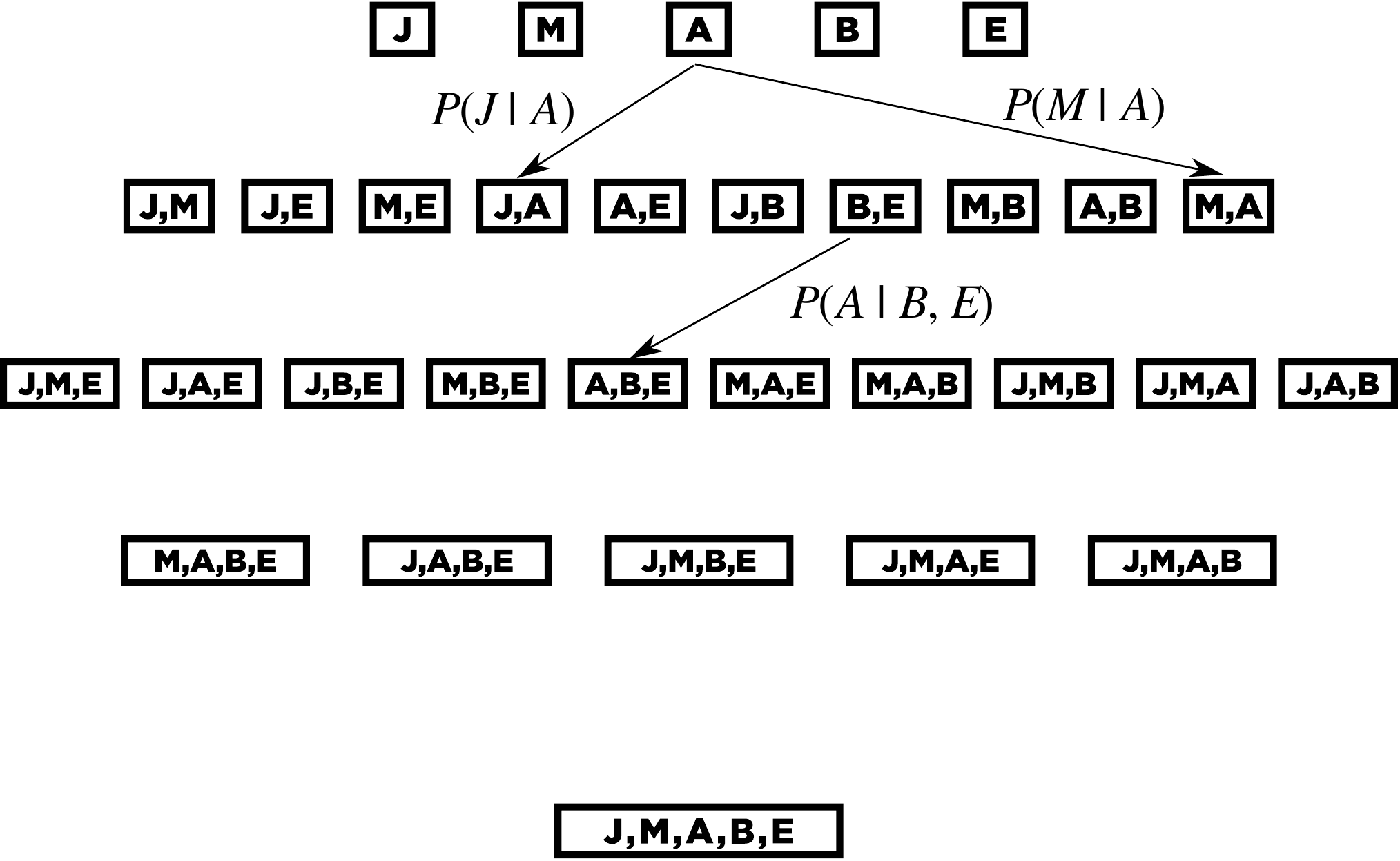}
    \caption{The conditional probability maps for Example \ref{eg:alarm}}
    \label{fig:alarm_conditionals}
  \end{center}
\end{figure}

\begin{example}
  \label{eg:alarm} (Thanks to Olivia Chen for this example and the associated graphics!)
  Consider the situation of two people, John and Mary, in a house in which an alarm sounds.  The alarm can be triggered by two kinds of events: an earthquake or a burglary.  Once the alarm sounds, one of the people may attempt to investigate.  If we view this as a probabilistic situation, we might organize the different events in a causal diagram like the one shown in Figure \ref{fig:alarm_causal}.  There are quite a few marginal probability distributions in this situation, as is shown in Figure \ref{fig:alarm_marginals}.  To complete the graphical model, we add three conditional probabilities, shown in Figure \ref{fig:alarm_conditionals}, corresponding to the arrows marked in Figure \ref{fig:alarm_causal}.
\end{example}

In the literature, such solutions must be \emph{converged} (not change under iterations of the explicit system) and be \emph{consistent} (actually be sections).  Unlike the case of $\shf{J}$ and $\shf{J}' = f_* \shf{J}$ however, $f_* \shf{B}$ and $\shf{B}$ can be rather different.  Belief propagation algorithms generally operate only the variables, and so are reliable when the space of global sections of $f_* \shf{B}$ and $\shf{B}$ are isomorphic.  Although somewhat out of scope from this chapter, sheaf cohomology (Definition \ref{df:cohomology}) provides sufficient conditions for this to occur.  We include the statement for completeness.

 \begin{proposition}
   \label{prop:bayes_localcoh}
   Suppose that $\shf{B}$ is a sheaf model of a graphical model and that $f$ is the order preserving map defined in Proposition \ref{prop:marginal_iso}.  Then the map on global sections induced by $f_* \shf{B} \to \shf{B}$ is an isomorphism whenever $H^k(f^{-1}(v);\shf{B}) = 0$ (see Definition \ref{df:cohomology}) for all sets of random variables $v$ and all $k>0$. 
 \end{proposition}
 The proof of this statement follows immediately from the Vietoris Mapping Theorem \cite[Thm. 3, Section II.11]{Bredon} or \cite[Thm. 4.2]{TSPBook}.  If the hypotheses of this proposition are satisfied, then convergence and consistency are equivalent properties for the graphical model.

\section{Future prospects: homological analysis of multi-model systems}
 
  \label{sec:homological}

Once a system has been encoded in a diagrammatic way -- as a sheaf -- its analysis is effectively a purely mathematical task.  For sheaves over posets whose stalks are vector spaces and whose restrictions (or extensions, for dual sheaves) are linear maps, \emph{homological} invariants can be computed \cite[Sec. 2.5]{Baclawski_1977}.

If the sheaf model does not have linear restriction maps, then it is necessary to linearize them before homological analysis can proceed.  The question of \emph{where} to linearize is easily addressed, at least theoretically: one should linearize about a section!  

\begin{definition}
  \label{def:linearized}
If $\shf{S}$ is a sheaf of smooth manifolds over a poset $P$ and $s$ is a global section of $\shf{S}$, then one can construct the \emph{linearized sheaf $\widetilde{\shf{S}}_s$ about $s$}.  This is defined by 
\begin{enumerate}
\item the stalk $\widetilde{\shf{S}}_s(x)$ over $x \in P$ is the tangent space $T_{s(x)}\shf{S}(x)$, and
\item the restriction $\widetilde{\shf{S}}_s(x \le y)$ for $x \le y \in P$ is the derivative map of the corresponding restriction in $\shf{S}$, namely
\begin{equation*}
d_{s(x)} \shf{S}(x \le y): T_{s(x)}\shf{S}(x) \to T_{s(y)}\shf{S}(y).
\end{equation*}
\end{enumerate}
\end{definition}

Observe that $\widetilde{\shf{S}}_s$ is a sheaf of vector spaces on $P$, whose restriction maps are linear maps.  The global section $s$ of $\shf{S}$ corresponds to the zero section in $\widetilde{\shf{S}}_s$.  Global sections of $\widetilde{\shf{S}}_s$ correspond to perturbations of $s$ in $\shf{S}$, and therefore describe the neighborhood of $s$ in the space of global sections of $\shf{S}$.

\begin{definition}
\label{df:cohomology}
Suppose that $\shf{S}$ is a sheaf of vector spaces with linear restriction maps on a poset $P$.  The \emph{$k$-cochain space} $C^k(\shf{S})$ of $\shf{S}$ consists of the following direct product of stalks at the end of chains of length $k$: 
\begin{equation*}
C^k(P;\shf{S}) = \prod_{a_0 < \dotsb < a_k} \shf{S}(a_k).
\end{equation*}
Each element in $C^k$ is therefore indexed by a chain in $P$ of length $k$, and can therefore be thought of as a function $s$ from the collection of chains in $P$.  The \emph{$k$-coboundary map} $d^k: C^k(P;\shf{S}) \to C^{k+1}(P;\shf{S})$ is given by the formula
\begin{eqnarray*}
\left(d^k s\right)(a_0 < \dotsb a_{k+1}) &=& \sum_{i=0}^{k} (-1)^i s(a_0 < \dotsb \widehat{a_i} < \dotsb a_{k+1}) \\
&& + (-1)^{k+1} \shf{S}(a_k < a_{k+1})\left(s(a_0 < \dotsb < a_k)\right). 
\end{eqnarray*}
The cochain spaces and the coboundary maps form a chain complex, whose homology
\begin{equation*}
H^k(P;\shf{S}) = H^k ( C^\bullet(P;\shf{S}), d^\bullet)
\end{equation*}
is called the \emph{cohomology of the sheaf $\shf{S}$}.
\end{definition}

\begin{remark}
Although the formula for the coboundary map seems a bit unmotivated, it is rather reasonable.  The usual boundary map in simplicial homology is of the form
\begin{eqnarray*}
\partial [v_0, \dotsc, v_{k+1}] &=& \sum_{i=0}^{k+1} (-1)^i [v_0, \dotsc, \widehat{v_i}, \dotsc, v_{k+1}]\\
&=& \left(\sum_{i=0}^k (-1)^i [v_0, \dotsc, \widehat{v_i}, \dotsc, v_{k+1}]\right) + (-1)^{k+1} [v_0, \dotsc, v_k].\\
\end{eqnarray*}
Transferring this to the setting of chains in $P$ (by considering the space of functions on simplices), the terms in the sum all correspond to chains in $P$ that end on $v_{k+1}$, while the final term corresponds to a chain that ends at $v_{k}$.  So if this is really to be a map $C^k(P;\shf{S}) \to C^{k+1}(P;\shf{S})$, all terms in the sum but the final one end up where they ought to: in $C^{k+1}(P;\shf{S})$.  This is easily corrected by moving the final term along a restriction map, which is precisely what the Definition \ref{df:cohomology} prescribes.
\end{remark}

The most effective homological analysis of sheaf models follows the following work flow:
 \begin{enumerate}
\item Encode the diagrammatic model as a sheaf over a poset as described in this chapter,
\item Linearize, if necessary, 
\item Summarize the sheaf model by computing its cohomology, and
\item Reinterpret the cohomology spaces in terms of dynamical invariants.
 \end{enumerate}

\begin{proposition}
  \label{prop:h0_meaning}
$H^0(\shf{S})$ is isomorphic to the space of global sections of $\shf{S}$.
\end{proposition}

This means that computing the space of global sections -- solutions to the multi-model systems developed in this chapter -- amounts to computing the kernel of a linear map.

\begin{proof}
In this degenerate setting, the 0-length chains in the poset $P$ are merely all of the elements.  Thus, $C^k(\shf{S})$ is the product of all stalks of $\shf{S}$.  Then the coboundary maps are given by
\begin{equation*}
\left(d^0 s\right)(a_0 < a_1) = s(a_1) - \shf{S}(a_0 < a_1)s(a_0). 
\end{equation*}
for each length 1 chain $a_0 < a_1$.  Notice that the kernel of $d^0$ expresses the fact that the values chosen on each stalk agree with the values propagated along the restriction maps -- precisely the condition that $s$ is a global section.
\end{proof}

In addition to the global sections, the higher degree sheaf cohomology spaces encapsulate other useful information.  For instance, there is a sheaf-theoretic Nyquist theorem \cite{Robinson_SampleBook} that explains the efficacy of discretization methods described in Section \ref{sec:discretization}.  Briefly, if $\shf{S}$ is a sheaf of solutions of some model and $\shf{S} \to \shf{D}$ is a discretization morphism, then $H^0(\shf{S}/\shf{D})$ and $H^1(\shf{S}/\shf{D})$ describe limits on the kinds of inferences that can be drawn about $\shf{S}$ from sections of $\shf{D}$.  Additionally, $H^1(\shf{S})$ for non-discretized sheaf models of differential equations can describe certain dynamical properties of a system \cite{RobinsonQGTopo}.

Since the solution sheaves, as defined in Section \ref{sec:simultaneous}, are written over posets with two levels -- equations and variables -- the maximum path length is 1.  Therefore, the only nontrivial cohomology spaces can be $H^0$ and $H^1$.  The interpretation of $H^0$ is clear in light of Proposition \ref{prop:h0_meaning}.  $H^1$ consists of the values of variables that are not consistent across all models.  Intuitively, $H^1$ measures the ``degrees of freedom of the model that have been constrained out by the equations.''  When we move to the dual sheaf of sheaves in Section \ref{sec:multidiffeqn}, then other, higher-degree cohomology spaces can become nontrivial.

\begin{acknowledgement}
  The author would like to thank the anonymous referees for the thoughtful suggestions that have improved this chapter considerably.  This work was partially supported under the DARPA SIMPLEX program through SPAWAR, Federal contract N66001-15-C-4040. 
\end{acknowledgement}

\bibliographystyle{plain}
\bibliography{multimodel_bib}
\end{document}